\documentclass[11pt]{article}
\usepackage{graphicx,epsfig}
\usepackage{fancyhdr,fancybox,float}
\usepackage{appendix}
\usepackage{titlesec}
\usepackage{etoolbox, chngcntr}
\usepackage{indentfirst}
\usepackage{verbatim}
\usepackage[sort&compress, numbers]{natbib}
\usepackage{geometry}
\usepackage{extarrows,chemarrow,xypic}
\usepackage[small]{caption2}
\usepackage{microtype}
\usepackage{amssymb,dsfont,amsthm}
\usepackage{hyperref}
\usepackage[dvipsnames, svgnames, x11names]{xcolor}
\DisableLigatures[f]{encoding = *, family = *}
\AtBeginEnvironment{appendices}{
	\titleformat{\section}{\bfseries\large}{\appendixname~\thesection:}{0.5em}{}
	\titleformat{\subsection}{\bfseries\large}{\thesubsection}{0.5em}{}
}

\newcommand{\paperfont}{\fontsize{10pt}{1.2\baselineskip}\selectfont}
\makeatletter
\begin{document}
\allowdisplaybreaks

\theoremstyle{definition}
\makeatletter
\thm@headfont{\bf}
\makeatother
\newtheorem{theorem}{Theorem}[section]
\newtheorem{proposition}[theorem]{Proposition}
\newtheorem{lemma}[theorem]{Lemma}
\newtheorem{corollary}[theorem]{Corollary}
\newtheorem{claim}[theorem]{Claim}
\newtheorem{definition}[theorem]{Definition}
\newtheorem{assumption}{Assumption}
\newtheorem{condition}{Condition}
\newtheorem{remark}[theorem]{Remark}
\newtheorem{example}[theorem]{Example}
\numberwithin{equation}{section}

\newcommand{\Pnum}{\mathbb{P}}
\newcommand{\Enum}{\mathbb{E}}
\newcommand{\Rnum}{\mathbb{R}}
\newcommand{\ms}[1]{{\mathscr #1}}
\newcommand{\id}{{1 \mskip -5mu {\rm I}}}
\newcommand{\sce}{\mathop{\rm sc^-\!}\nolimits}
\newcommand{\supp}{\mathop{\rm supp}\nolimits}
\renewcommand{\div}{\mathop{\rm div}\nolimits}
\newcommand{\Ent}{\mathop{\rm Ent}\nolimits}
\newcommand{\asinh}{\mathop{\rm arcsinh}\nolimits}

\def\bP{{\bf P} }
\def\cL{{\cal L} }
\def\C{{\cal C} }
\title{\textbf{Large deviations for the empirical measure and empirical flow of Markov renewal processes with a countable state space}}
\author{Chen Jia$^{1}$,\;\;\;Da-quan Jiang$^{2,3}$,\;\;\;Bingjie Wu$^{2}$ \\
\footnotesize $^1$ Applied and Computational Mathematics Division, Beijing Computational Science Research Center, Beijing, 100193, China \\
\footnotesize $^2$ LMAM, School of Mathematical Sciences, Peking University, Beijing 100871, China. \\
\footnotesize $^3$ Center for Statistical Science, Peking University, Beijing 100871, China. \\
\footnotesize $^*$ E-mails: chenjia@csrc.ac.cn (C. Jia); jiangdq@math.pku.edu.cn (D.-Q. Jiang); wubingjie@pku.edu.cn (B. Wu)
}
\date{}
\maketitle
\thispagestyle{empty}

\paperfont

{\abstract
Here we propose the Donsker-Varadhan-type compactness conditions and prove the joint large deviation principle for the empirical measure and empirical flow of Markov renewal processes (semi-Markov processes) with a countable state space, generalizing the relevant results for continuous-time Markov chains with a countable state space obtained in [Ann. Inst. H. Poincar\'{e} Probab. Statist. 51, 867-900 (2015)] and [Stoch. Proc. Appl. 125, 2786-2819 (2015)], as well as the relevant results for Markov renewal processes with a finite state space obtained in [Adv. Appl. Probab. 48, 648-671 (2016)]. In particular, our results hold when the flow space is endowed with either the bounded weak* topology or the strong $L^1$ topology.  Even for continuous-time Markov chains, our compactness conditions are weaker than the ones proposed in previous papers. Furthermore, under some stronger conditions, we obtain the explicit expression of the marginal rate function of the empirical flow.}

\medskip


\medskip


\section{Introduction}\label{introduzione}
Semi-Markov processes, which can be viewed as a direct extension of discrete-time and continuous-time Markov chains, are one of the most important classes of non-Markov processes. They have attracted considerable attention in recent years and have found wide applications in physics, chemistry, biology, finance, and engineering \cite{janssen2006applied, asmussen2008applied, limnios2012semi}. The embedded chain of a semi-Markov process is a discrete-time Markov chain, while the waiting times may not be exponentially distributed. Such non-exponential waiting time distributions have been found in many scientific problems such as molecular motors \cite{kolomeisky2000extended, faggionato2017fluctuation}, enzyme kinetics \cite{qian2006generalized, wang2007detailed}, gene networks \cite{pedraza2008effects, jia2022analytical}, and cell cycle dynmics \cite{chao2019evidence, jia2021frequency}. The representation of a semi-Markov process in terms of its embedding chain and waiting times is also called a Markov renewal process.

The mathematical theory of large deviations was initiated by Cram\'er \cite{cramer1938nouveau}, and was later developed by many mathematicians and physicists. In the pioneering work \cite{donsker1975asymptotic1, donsker1975asymptotic2, donsker1976asymptotic, donsker1983asymptotic}, Donsker and Varadhan have established the large deviation principle (LDP) for the empirical measure and for the empirical process associated with a large class of discrete-time and continuous-time Markov processes. The large deviations for the sample mean, empirical measure, and empirical processes are usually said to be at level 1, level 2, and level 3, respectively \cite{dembo1998large}. For discrete-time and continuous-time Markov chains, the number of jumps along each oriented edge of the transition graph per unit time is called the empirical flow. The large deviations for the empirical flow are often said to be at level 2.5 \cite{kesidis1993relative}, since it is between level 2 and level 3. For a discrete-time Markov chain $(X_n)_{n\geq 0}$, the large deviations for the empirical flow can be obtained directly from those for the empirical measure since the binary process $(X_n,X_{n+1})_{n\geq 0}$ is also a discrete-time Markov chain. However, things become much more complicated in the continuous-time case. For continuous-time Markov chains, Fortelle \cite{de2001large} proved a weak joint LDP for the empirical measure and empirical flow and obtained the corresponding rate function. Subsequently, Bertini et al. \cite{bertini2015large, bertini2015flows} proposed some Donsker-Varadhan-type compactness conditions and proved the full joint LDP for the empirical measure and empirical flow.

Since semi-Markov processes are direct generalizations of discrete-time and continuous-time Markov chains, a natural question is whether the large deviations at different levels can be extended to semi-Markov processes. Along this line, Mariani and Zambotti \cite{mariani2016large} proved the joint LDP for the empirical measure and empirical flow of semi-Markov processes with a \emph{finite} state space and expressed the corresponding rate function in terms of relative entropy. The large deviations for the empirical flow of Markov and semi-Markov processes have also been applied in statistical mechanics to study the fluctuation relations of thermodynamical systems far from equilibrium \cite{andrieux2007fluctuation, andrieux2008fluctuation, bertini2015flows, jia2016cycle, chen2016large, ge2017cycle, faggionato2017large, budhiraja2021large, ge2021martingale, jiang2022large}. Up till now, it is still unclear whether these finite state space results can be generalized to semi-Markov processes with a \emph{countable} state space. The aim of the present paper is to fill in this gap.

Here we investigate the joint large deviations for the empirical measure and empirical flow of semi-Markov processes with a countable state space. When the state space has an infinite number of states, the choice of the topology of the flow space, i.e. the value space of all empirical flows, will become very important. Following \cite{bertini2015large, bertini2015flows}, we consider two types of topology of the flow space: the bounded weak* topology and the strong topology. Specifically, we propose two different Donsker-Varadhan-type compactness conditions for the two types of topology, and prove the corresponding LDP for the empirical measure and empirical flow. This is the first main contribution of this paper. In the special case of continuous-time Markov chains, our compactness condition reduces to the one proposed in \cite{bertini2015large} when the flow space is endowed with the bounded weak* topology. However, when the flow space is endowed with the strong topology, our compactness condition is even weaker than the one proposed in \cite{bertini2015flows}.

In \cite{mariani2016large}, Mariani and Zambotti also provided the explicit expression of the marginal rate function for the empirical measure, but they have not given the explicit expression of the marginal rate function for the empirical flow. Here we propose two strong compactness conditions, one for the embedded chain and one for the waiting time distributions, and find that the marginal rate function for the empirical flow can be written down explicitly under these compactness conditions. The relationship between these compactness conditions and the geometric ergodicity of the embedded chain is also clarified. This is
the second main contribution of this paper.

The present paper is organized as follows. In Section \ref{definizioni}, we recall the definitions of Markov renewal processes and semi-Markov processes, as well as the definitions of the empirical measure and empirical flow. In Section \ref{sec:main results}, we propose two different compactness conditions when the flow space is endowed with the bounded weak* topology and the strong topology. Moreover, we also state the main results including the joint and marginal LDPs for the empirical measure and empirical flow. The proof of the LDP for the two types of topology will be given in Sections \ref{section:bwt topology} and \ref{section:s topology}, respectively. In Section \ref{section:contraction principle}, we derive the marginal rate function for the empirical flow.

\section{Preliminaries}\label{definizioni}

\subsection{Markov renewal processes and semi-Markov processes}\label{sec:def and ass}
Let $V$ be a countable set endowed with the discrete topology and the associated Borel $\sigma$-algebra is the collection of all subsets of $V$. The set $[0,\infty]$ is equipped with the topology that is compatible with the natural topology on $[0,\infty)$ so that $(s,\infty]$ is open for any $s\geq 0$. For any Polish space $\mathcal{X}$, let $\mathcal{P}(\mathcal{X})$ denote the collection of Borel probability measures on $\mathcal{X}$. For any $\mu \in \mathcal{P}(\mathcal{X})$ and $f\in L^1(\mu)$, let $\langle\mu,f\rangle$ or $\mu(f)$ denote the integral of $f$ with respect to $\mu$. The set $\mathcal{P}(\mathcal{X})$ is equipped with the topology of weak convergence and the associated Borel $\sigma$-algebra.

We first recall the definition of Markov renewal processes, also called ($J$-$X$)-processes \cite{limnios2012semi}. Here, we adopt the definition in \cite{mariani2016large} and \cite[Chapter VII.4]{asmussen2008applied}.

\begin{definition}\label{def:markov renewal process}
The process $(X,\tau) = \{(X_k)_{k\ge 0},(\tau_{k})_{k\ge 1}\}$ defined on a probability space $(\Omega,\mathcal {F},\mathbb{P})$ is called a \emph{Markov renewal process} if
\begin{itemize}
\item[(a)] $X = (X_k)_{k\geq 0}$ is a discrete-time time-homogeneous Markov chain with countable state space $V$ and transition probability matrix $P = (p_{xy})_{x,y\in V}$.
\item[(b)] $\tau = (\tau_{k})_{k\geq 1}$ is a sequence of positive and finite random variables such that conditioned on $(X_k)_{k\ge 0}$, the random variables $(\tau_{k})_{k\ge 1}$ are independent and have distribution
\begin{equation*}
\mathbb{P}\left(\tau_{i+1}\in\cdot\,|\,(X_k)_{k\geq 0}\right) = \psi_{X_{i},X_{i+1}}(\cdot),
\end{equation*}
where $\psi_{xy}\in \mathcal{P}(0,\infty)$ for any $x,y\in V$. The matrix $\Psi=(\psi_{xy})_{x,y\in V}$ is called the \emph{waiting time matrix}. Note that it is a matrix of probability measures. The pair $(P,\Psi)$ is called the \emph{transition kernel}.
\end{itemize}
\end{definition}

We next recall the definition of semi-Markov processes.

\begin{definition}
For any $t \ge 0$ and $n \ge 1$, let
\begin{equation}\label{def}
S_n = \sum_{i=1}^n \tau_i,\;\;\;
N_t = \sum_{n=1}^\infty 1_{(S_n\leq t)} = \inf\left\{n\geq 0: S_{n+1}>t \right\},
\end{equation}
where $\inf \emptyset := \infty$ and $1_A$ is the indicator function of the set $A$. The process $\xi=(\xi_t)_{t\ge 0}$ with $\xi_t:=X_{N_t}$ is called the \emph{semi-Markov process} associated with the Markov renewal process $(X,\tau)$. Clearly, $\xi$ is a jump process whose trajectories are right-continuous on the state space $V$.
\end{definition}

Clearly, $S_n$ represents the $n$th jump time of the semi-Markov process $\xi$, $N_t$ represents the number of jumps of $\xi$ up to time $t$, and $X = (X_k)_{k\ge 0}$ is the embedded chain of $\xi$. It is easy to see that $N_t=n$ if and only if  $S_n\le t <S_{n+1}.$ In particular, if all the waiting times are equal to $1$, i.e. $\psi_{xy} = \delta_1$ for any $x,y\in V$ with $\delta_1$ being the point mass at $1$, then $\xi$ reduces to a discrete-time Markov chain. If all the waiting times are exponentially distributed, i.e. $\psi_{xy}({\rm d}t) = q_x e^{-q_xt}{\rm d}t$ for any $x,y\in V$, then $\xi$ reduces to a continuous-time Markov chain. In the following, we do not distinguish the Markov renewal process $(X,\tau)$ and the associated semi-Markov process $\xi$ since they are totally equivalent.

The transition diagram of the embedded chain $X$ is a directed graph $(V,E)$, where the edge set
\begin{equation*}
E = \left\{(x,y) \in V\times V  \,:\, p_{xy}>0\right\}
\end{equation*}
is composed of all directed edges with positive transition probabilities. Throughout this paper, we impose the following basic assumptions on the Markov renewal process $(X,\tau)$.

\begin{assumption}\label{ass:irreducibility}
The embedded chain $X$ is irreducible.
\end{assumption}

\begin{assumption}\label{ass:recurrence}
The embedded chain $X$ is recurrent.
\end{assumption}

\begin{assumption}\label{ass:dti}
The waiting time distribution only depends on the current state, i.e. $\psi_{xy}=\psi_{x}$ for any $x,y\in V$.
\end{assumption}

\begin{assumption}\label{ass:locally finite}
For any $x\in V$, the number of incoming edges into $x$ of the graph $(V,E)$ and the number of outgoing edges from $x$ are both finite.
\end{assumption}

Assumptions \ref{ass:irreducibility} and \ref{ass:recurrence} are standard in the literature \cite{asmussen2008applied}. Assumption \ref{ass:dti} is also common in previous papers \cite{mariani2016large} which guarantees that the large deviation rate function of the empirical measure and empirical flow has the form of relative entropy. Note that if Assumption \ref{ass:dti} does not hold for $(X,\tau)$, then the process $(Y,\tau) = \{(Y_k)_{k\ge 0},(\tau_{k})_{k\ge 1}\}$ with $Y_k:=(X_k,X_{k+1})$ is also a Markov renewal process which satisfies Assumption \ref{ass:dti} since
\begin{equation*}
\mathbb{P}(\tau_{k+1}\in\cdot\;|\;(X_k)_{k\ge 0}) = \psi_{X_{k},X_{k+1}}(\cdot)=\psi_{Y_k}(\cdot).
\end{equation*}
Moreover, we emphasize that Assumption \ref{ass:locally finite} is not needed if $\xi$ is a continuous-time Markov chain \cite{bertini2015large, bertini2015flows}. Here we need this assumption in order to prove Lemma \ref{lemma:IhHdelta lower semicontinuous} below. In \cite{bertini2015large, bertini2015flows}, the counterpart of this lemma is proved by using the classical level 3 large deviation results of Donsker and Varadhan \cite{donsker1983asymptotic} and the contraction principle. However, since there is no level 3 large deviation results for Markov renewal processes, we need to impose the above assumption to overcome some technical difficulties.

Recall that the semi-Markov process $\xi$ is called \emph{non-explosive} if the explosion time $S_{\infty} := \lim_{n\to\infty} S_n$ satisfies
\begin{equation*}
\Pnum_x(S_{\infty} = \infty) = 1
\end{equation*}
for all $x\in V$, where $\Pnum_x(\cdot) = \Pnum(\cdot|X_0 = x)$. In fact, Assumptions \ref{ass:irreducibility} and \ref{ass:recurrence} ensure that $\xi$ is non-explosive. The proof is similar to the one given in \cite{norris1998markov} for continuous-time Markov chains and thus is omitted.

\subsection{Empirical measure and empirical flow}\label{sec:emef}
Next we introduce the definitions of the empirical measure and empirical flow for Markov renewal processes \cite{mariani2016large}. For any $t>0$, the empirical measure $\mu_t :\Omega \to \mathcal P(V \times (0,\infty])$ of $(X,\tau)$ is defined by
\begin{equation}\label{def:emp mes}
\mu_t = \frac{1}{t}\int_0^{t}\delta_{(X_{N_s},\tau_{N_s+1})}{\rm d}s,
\end{equation}
where $\delta.$ denotes the Dirac delta measure. In other words, for any $x\in V$ and $A\subset(0,\infty]$, we have
\begin{equation*}
\mu_t(\{x\},A) = \frac{1}{t}\int_0^{t}1_{(X_{N_s}=i,\tau_{N_s+1}\in A)}{\rm d}s.
\end{equation*}
Then $\mu_t$ is a random probability measure such that for any Borel measurable function $f$ on $V \times (0,\infty]$,
\begin{equation*}
\begin{split}
\langle\mu_t,f\rangle =\frac{1}{t}\int_0^{t} f(X_{N_s},\tau_{N_s+1}) {\rm d} s=
\sum_{k=1}^{N_t}\frac{\tau_{k}}{t} \,f(X_{k-1},\tau_k)+\frac{t-S_{N_t}}{t}f(X_{N_t},\tau_{N_t+1}).
\end{split}
\end{equation*}
Moreover, for any $t>0$, the empirical flow $Q_t:\Omega\to [0,\infty]^E$ of $(X,\tau)$ is defined by
\begin{equation}\label{def:empflow}
Q_t = \frac{1}{t} \sum_{k=1}^{N_t+1}\delta_{(X_{k-1},X_{k})}.
\end{equation}
In other words, for any $x,y\in V$, we have
\begin{equation*}
Q_t(x,y) = \frac{1}{t} \sum_{k=1}^{N_t+1}1_{(X_{k-1}=x,X_{k}=y)}.
\end{equation*}
Intuitively, $Q_t(x,y)$ represents the number of times that $\xi$ transitions from $x$ to $y$ per unit time. For any $n\ge 0$, let
\begin{equation*}
\mathcal{F}_n=\sigma\left((X_{k},\tau_{k})_{0\le k\le n}\right),\quad\mathcal{F}_{\infty}=\sigma\left(\bigcup_{n=1}^{\infty}\mathcal{F}_n\right).
\end{equation*}
be a filtration. Then $N_t+1$ is an $\{\mathcal{F}_n\}$-stopping time. It is easy to see that $\langle\mu_t,f\rangle$ and $Q_t(x,y)$ are $\mathcal{F}_{N_t+1}$-measurable random variables.

\begin{remark}\label{empiricalmeasure}
For the semi-Markov process $\xi$, a more natural definition of the empirical measure $\pi_t:\Omega\to\mathcal{P}(V)$ is given by
\begin{equation}\label{def:empirical measure pi_t}
\pi_t(x)=\frac{1}{t}\int_0^t1_{(\xi_s=x)}{\rm d}s=\mu_t(x,(0,\infty]),\quad x\in V.
\end{equation}
Comparing \eqref{def:emp mes} and \eqref{def:empirical measure pi_t}, we can see that $\pi_t$ only focuses on the spatial variable and $\mu_t$ focus on both the spatial and temporal variables. The reason why we use $\mu_t$ rather than $\pi_t$ in the study of the joint LDP is that only by using $\mu_t$ can we obtain a concise expression of the rate function. It is easy to verify that $(X_k,\tau_{k+1})_{k\ge 0}$ is a Markov process and hence $(X_{N_t},\tau_{N_t+1})_{t\ge 0}$ is also a semi-Markov process. In fact, the empirical measure $\mu_t$ for the process $\xi_t = X_{N_t}$ is exactly the empirical measure $\pi_t$ for the process $(X_{N_t},\tau_{N_t+1})_{t\ge 0}$.
\end{remark}

Let $L^1(E)$ denote the set of absolutely summable functions on $E$ and let $\|\cdot\|$ denote the associated $L^1$-norm. The set of nonnegative elements of $L^1(E)$ is denoted by $L^1_+(E)$, which is called the \emph{flow space}. An element in $L^1_+(E)$ is called a \emph{flow}. Since $(X,\tau)$ is non-explosive, it is easy to see that $Q_t\in L^1_+(E)$ for any $t>0$. For any flow $Q\in L^1_+(E)$, let the exit-current $Q^+\colon V\to \mathbb R$ and entrance-current $Q^-\colon V\to \mathbb R$ be defined by
\begin{equation}
\label{exit entrance current}
Q^+ (x)= \sum _{y: \, (x,y)\in E} Q(x,y),\;\;\;Q^-(x)= \sum_{y:\, (y,x)\in E} Q(y,x).
\end{equation}
Intuitively, the exit-current and entrance-current at $x$ are the flows exiting from $x$ and entering into $x$, respectively. In particular, if $Q^+(x)=Q^-(x)$ for any $x\in V$, then $Q$ is called a \emph{divergence-free flow} and we define
\begin{equation}\label{def:current J}
Q_x = Q^+(x) = Q^-(x),\qquad x\in V.
\end{equation}
Note that both currents map $L^1_+(E)$ into $L^1_+(V)$. Here $L^1_+(V)$ is defined in the same way as $L^1_+(E)$ and let $\|\cdot\|$ denote the associated $L^1$-norm.

In this paper, we will consider two types of topology on $L^1(E)$: the strong topology generated by the $L^1$-norm and the bounded weak* topology, which is defined as follows \cite{bertini2015large}. Let $C_0(E)$ denote the collection of continuous functions $f:E\to\Rnum$ vanishing at infinity, i.e. for any $\epsilon>0$, there exists a finite set $K\subset E$ such that
\begin{equation*}
\sup_{(x,y)\in K^c}|f(x,y)|\le \epsilon,
\end{equation*}
and it is endowed with the $L^\infty$-norm. It is well-known that the dual space of $C_0(E)$ is $L^1(E)$ endowed with the strong topology. For any $\ell>0$, let $B_\ell := \big\{Q\in L^1(E):\,\|Q\| \le\ell\big\}$ denote the closed ball of radius $\ell$ in $L^1(E)$. In view of the separability of $C_0(E)$ and the Banach-Alaoglu theorem, the closed ball $B_\ell$ endowed with the weak* topology is a compact Polish space. The bounded weak* topology on $ L^1(E)$ is then defined by declaring the set $A\subset L^1(E)$ to be open if and only if $A\cap B_\ell$ is open in the weak* topology of $B_\ell$ for any $\ell>0$. In fact, the bounded weak* topology is stronger than the weak* topology and is weaker than the strong topology. Moreover, for each $\ell>0$, the closed ball $B_\ell$ is compact with respect to the bounded weak* topology. In particular, the three types of topology on $L^1(E)$ coincide only when $E$ is finite. The proof of the above statements can be found in \cite[Section 2.7]{megginson2012introduction}.

For both the strong topology and the bounded weak* topology, we regard $L^1_+(E)$ as a closed subset of $L^1(E)$ and endow it with the relative topology and the associated Borel $\sigma$-algebra. The product space $V\times (0,\infty]$ is equipped with the product topology so that $V\times (0,\infty]$ is a Polish space. The set $\mathcal{P}(V\times (0,\infty])$ of Borel probability measures on $V\times (0,\infty]$ is endowed with the topology of weak convergence. Moreover, the product space $\Lambda = \mathcal{P}(V\times(0,\infty])\times L^1_+(E)$ is endowed with the product topology. Then for any $t>0$, the pair $(\mu_t,Q_t)$, where $\mu_t$ is the empirical measure and $Q_t$ is the empirical flow, can be viewed as a measurable map from $\Omega$ to $\Lambda$.

\section{Main results}\label{sec:main results}

\subsection{Compactness conditions}\label{sec:compactness condition}
The aim of the present paper is to establish the LDP for the empirical measure and empirical flow of semi-Markov processes. Similarly to the LDP of Markov processes \cite{donsker1983asymptotic}, we need some compactness conditions to control the convergence rate at infinity when the state space $V$ has an infinite number of states. For Markov processes, the infinitesimal generator is often used to establish the compactness conditions \cite{donsker1983asymptotic}. In the semi-Markov case, the \emph{transition kernel} $(P,\Psi)$ plays the role of the generator.

Before stating the compactness conditions, we need the following notation. Let $\zeta:V\rightarrow[0,\infty]$ be a function defined by
\begin{equation}\label{def:zeta}
\zeta(x) = \sup \left\{\lambda\in\mathbb{R}:\: \psi_x\left(e^{\lambda \tau}\right) <\infty \right\},
\end{equation}
where $\psi_x(e^{\lambda \tau})=\int_{(0,\infty)}e^{\lambda s}\psi_{x}({\rm d}s)$. For any $x\in V$, let $\theta_x:(0,\infty)\to(-\infty,\infty]$ be a function defined by
\begin{equation}\label{def:theta_x}
\theta_x(t) = \sup \left\{\lambda\in\mathbb{R}:\: \psi_x\left(e^{\lambda \tau}\right) \le t \right\}.
\end{equation}
For any $f: V\to(0,\infty)$, let $Pf:V\to(0,\infty]$ be a function defined by
\begin{equation*}
Pf(x) = \sum_{y:\,(x,y)\in E}p_{xy}f(y),
\end{equation*}
and let $Lf:V\to(-\infty,\infty]$ be a function defined by $Lf(x) = \theta_x(f(x))$.

\begin{lemma}\label{lemma:continuous for theta}
Let $\theta_x$ be the function defined in \eqref{def:theta_x}. Then
\begin{itemize}
\item[(a)] $\theta_x$ is an increasing continuous function. Moreover, $\theta_x$ is strictly increasing on $(0,\psi_x(e^{\zeta(x)\tau}))$.
\item[(b)] $\theta_x(1)=0$, $\lim_{t\downarrow 0}\theta_x(t) = -\infty$, and $\lim_{t\uparrow\infty}\theta_x(t) = \zeta(x)$.
\item[(c)] For any $t>0$ and $x\in V$, we have
\begin{equation*}
\psi_x\left(e^{\theta_x(t)\tau}\right) = \left\{
\begin{aligned}
&\psi_x\left(e^{\zeta(x)\tau}\right),
&&\textrm{if }t>\psi_x\left(e^{\zeta(x)\tau}\right), \\
& t,  &&\textrm{otherwise}.
\end{aligned}\right.
\end{equation*}
\end{itemize}	
\end{lemma}

\begin{proof}
For any $x\in V$, let $\Theta_x(\lambda)=\psi_x(e^{\lambda \tau})$ be a function on $\mathbb{R}$. By the dominated convergence theorem, it is easy to see that $\Theta_x\in C((-\infty,\zeta(x)))$ is a strict increasing function. Moreover, it is clear that $\Theta_x(0)=1$ and
\begin{equation*}
\lim_{\lambda\to-\infty}\Theta_x(\lambda)=0, \qquad \Theta_x(\lambda)=\infty,\quad\lambda> \zeta(x).
\end{equation*}
Since $\theta_x$ is the inverse function of $\Theta_x$ for $0<t<\Theta_x(\zeta(x))$, we immediately obtain (a)-(c).
\end{proof}

In \cite{donsker1976asymptotic,donsker1983asymptotic}, Donsker and Varadhan proposed a compactness condition for the generator which was then used to prove the LDP for the empirical measure of Markov processes. In \cite{bertini2015flows, bertini2015large}, Bertini et al. proposed a Donsker-Varadhan-type condition which was then used to prove the LDP for the empirical flow of continuous-time Markov chains. In what follows, we will provide two Donsker-Varadhan-type conditions which are needed for the joint LDP for the empirical measure and empirical flow of semi-Markov processes when the flow space $L^1_+(E)$ is endowed with the bounded weak* topology and strong topology. The following compactness condition is needed for the joint LDP when $L^1_+(E)$ is endowed with the bounded weak* topology.

\begin{condition}\label{condition:ccomp}
There exists a sequence of functions $u_n\colon V \to (0,\infty)$ such that
\begin{itemize}
\item [(a)] For any $x\in V$ and $n\geq 0$, we have $Pu_n(x)<\infty$;
\item [(b)] There exists a constant $c>0$ such that $u_n(x)\ge c$ for any $x\in V$ and $n\geq 0$;
\item[(c)] For any $x\in V$, there exists a constant $C_x$ such that $u_n(x)\le C_x$ for any $n\geq 0$;
\item[(d)] The functions $u_n/Pu_n$ converge pointwise to some $\hat{u}:V\to (0,\infty)$;
\item[(e)] For each $\ell\in \mathbb R$, the level set $\big\{x \in V \,:\, L\hat{u}(x)\leq \ell\big\}$ is finite;
\item[(f)] There exist $\sigma,C>0$, $\eta\in(0,1)$, and a finite set $K\subset V$ such that $L\hat{u}(x)\geq -\sigma\theta_x(\eta)-C1_K(x)$ for any $x\in V$ and $\hat{u}(x)<\psi_x(e^{\zeta(x)\tau})$ for any $x\in K^c$.
\end{itemize}
\end{condition}

\begin{remark}
For continuous-time Markov chains, the term $C1_K$ in item (f) can be replaced by a constant $C$ and the condition $\hat{u}(x)<\psi_x(e^{\zeta(x)\tau})$ can be removed since $\psi_x(e^{\zeta(x)\tau}) = \infty$ for any $x\in V$. In this case, Condition \ref{condition:ccomp} reduces to the compactness condition proposed in \cite{bertini2015large}. Moreover, it follows from Lemma \ref{lemma:continuous for theta} that $L\hat{u}(x)=\theta_x(\hat{u}(x))\le \zeta(x)$ for any $x\in V$. Hence item (e) implies that the level sets of $\zeta$ are also finite.
\end{remark}

Since the strong topology is stronger than the bounded weak* topology, we need to impose a stronger compactness condition, which is essentially the Donsker-Varadhan-type condition for discrete-time Markov chains \cite{donsker1976asymptotic}. Both the following condition and Condition \ref{condition:ccomp} are needed for the joint LDP when $L^1_+(E)$ is endowed with the strong topology.

\begin{condition}\label{condition:ccomp2}
There exists a sequence of functions $u_n\colon V \to (0,\infty)$ such that
\begin{itemize}
\item [(a)] For any $x\in V$ and $n\geq 0$, we have $Pu_n(x)<\infty$;
\item [(b)] There exists a constant $c>0$ such that $u_n(x)\ge c$ for any $x\in V$ and $n\geq 0$;
\item [(c)] For any $x\in V$, there exists a constant $C_x$ such that $u_n(x)\le C_x$ for any $n\geq 0$;
\item [(d)] The functions $u_n/Pu_n$ converge pointwise to some $\hat{u}:V\to (0,\infty)$;
\item [(e)] For each $\ell\in \mathbb R$, the level set $\big\{x\in V \,:\, \log\hat{u}(x)\leq \ell\big\}$ is finite.
\end{itemize}
\end{condition}

It is clear that Condition \ref{condition:ccomp2} only depends on the embedded chain $X$ of the semi-Markov process and is independent of the waiting time distributions. When $L^1_+(E)$ is endowed with the strong topology, Bertini et al. \cite{bertini2015flows} have proposed another compactness condition for continuous-time Markov chains (see Condition \ref{condition:strong topology} below, which is rewritten for semi-Markov process). However, that condition is more complicated than Condition \ref{condition:ccomp2} and more difficult to verify. In fact, when $X$ is irreducible, Condition \ref{condition:ccomp2} is not only easier to verify, but also even weaker than Condition \ref{condition:strong topology}. The proof of this fact can be found in Appendix \ref{appendix: A}. This explains why we impose Condition \ref{condition:ccomp2} rather than Condition \ref{condition:strong topology} here.

\begin{remark}\label{equivalence}
For discrete-time Markov chains, we have $Lf(x)=\theta_x(f(x))=\log f(x)$ for any $f:V\rightarrow(0,\infty)$. Note that item (e) in Condition \ref{condition:ccomp2} implies that the set $K = \{x\in V:\log \hat{u}(x)\le -\log\eta\}$ is finite for any $\eta\in(0,1)$. Hence if we take $\sigma=1$, then we can always find $C>0$ such that $\log\hat{u}\ge -\sigma\log \eta-C1_{K}$. Furthermore, it is easy to check that $\psi_x(e^{\zeta(x)\tau}) = e^{\zeta(x)} = \infty$ for any $x\in V$. This shows if Condition \ref{condition:ccomp2} holds, then items (e) and (f) in Condition \ref{condition:ccomp} are automatically satisfied. Hence for discrete-time Markov chains, Condition \ref{condition:ccomp} is equivalent to Condition \ref{condition:ccomp2}.
\end{remark}

\subsection{Joint LDP for the empirical measure and empirical flow}\label{s:ldef}
Let $\mu$ and $\nu$ be two probability measures on a measurable space $(\mathcal{X},\mathcal{F})$. Recall that the relative entropy of $\mu$ with respect to $\nu$ has the following variational expression \cite{donsker1983asymptotic}:
\begin{equation}\label{def:H}
H(\mu\,|\,\nu) = \sup_{\varphi \in \mathcal{B}_{b}(\mathcal{X})} \left\{\langle\mu,\varphi\rangle-\log \langle\nu,e^{\varphi}\rangle\right\}
= \left\{\begin{aligned}
&\int_{\mathcal{X}}\left(\log\frac{{\rm d}\mu}{{\rm d}\nu}\right){\rm d}\mu, && \text{if } \mu \ll\nu\,,\\
&\infty, && \text{otherwise},
\end{aligned}\right.
\end{equation}
where $\mathcal{B}_{b}(\mathcal{X})$ denotes the space of bounded measurable functions on $\mathcal{X}$. Moreover, if $\mathcal{X}$ is a Polish space and $\mathcal{F}$ is the associated Borel $\sigma$-field, then \eqref{def:H} still holds when $\mathcal{B}_{b}(\mathcal{X})$ is replaced by $C_{b}(\mathcal{X})$, the space of bounded continuous functions on $\mathcal{X}$ \cite{donsker1983asymptotic}. If we set $\varphi'=\varphi-\log\langle\nu,e^{\varphi}\rangle$, then $\langle\mu,\varphi'\rangle = \langle\mu,\varphi\rangle - \log\langle\nu,e^{\varphi}\rangle$ and it is easy to see that the relative entropy $H(\mu\,|\,\nu)$ can be represented as
\begin{equation}\label{relative entropy}
H(\mu\,|\,\nu) = \sup_{\{\varphi\in C_b(\mathcal{X}):\langle\nu,e^{\varphi}\rangle=1\}} \langle\mu,\varphi\rangle. 	
\end{equation}

Let $\mathcal{D}$ be a subset of $\Lambda = \mathcal{P}(V\times(0,\infty])\times L^1_+(E)$ defined by
\begin{equation}\label{def:D0}
\mathcal{D}= \left\{(\mu,Q)\in \Lambda\, :\:
\int_{(0,\infty]}\frac{1}{t}\,\mu(x,{\rm d}t) = Q^+(x) = Q^-(x), \, \forall  x\in V\right\},
\end{equation}
where $Q^+$ and $Q^-$ are the exit-current and entrance-current of the flow $Q$, respectively. For each $(\mu,Q) \in \mathcal{D}$, we introduce the transition probabilities $(\tilde{Q}_{xy})_{x,y\in V}$ and the waiting time distributions $(\tilde{\mu}_x)_{x\in V}$ as
\begin{equation}
\label{def:pQ,psi mu}
\tilde{Q}_{xy} = \frac{Q(x,y)}{Q_x},
\quad \tilde{\mu}_x(\mathrm{d}t)=\frac{1}{Q_x t}\,\mu(x,{\rm d}t),
\end{equation}
where $Q_x = Q^+(x) = Q^-(x)$ is defined in \eqref{def:current J} and we set $\tilde{Q}_{xy} = p_{xy}$ and $\tilde{\mu}_x=\psi_x$ if $Q_x = 0$. Let $I:\Lambda \to [0,\infty]$ be a function defined by
\begin{equation}\label{def:I}
I(\mu,Q) = \left\{
\begin{aligned}
&\sum_{x\in V}
\left[
Q_xH\big( \tilde{Q}_{x,\cdot} \, |\,  p_{x,\cdot}  \big)
+ Q_xH\big( \tilde{\mu}_x \, | \, \psi_x  \big) +  \zeta(x)  \mu(x,\{\infty\}) \right],
&&\textrm{if } (\mu,Q)\in \mathcal{D}, \\
& \infty,  &&\textrm{otherwise}.
\end{aligned}\right.
\end{equation}

We are now in a position to state the main results of the present paper. The following theorem, whose proof can be found in Section \ref{section:bwt topology}, gives the joint LDP for the empirical measure and empirical flow when $L^1_+(E)$ is endowed with the bounded weak* topology.

\begin{theorem}\label{LDP:bounded weak topology}
Suppose that Assumptions \ref{ass:irreducibility}-\ref{ass:locally finite} and Condition \ref{condition:ccomp} hold. Let $L^1_+(E)$ be endowed with the bounded weak* topology. Then under $\mathbb{P}_x$, the law of $(\mu_t,Q_t)$ satisfies an LDP with good and convex rate function $I:\Lambda\to [0,\infty]$. In particular, for any closed set $\mathcal{C}\subset \Lambda$ and open set $\mathcal{A} \subset \Lambda$, we have
\begin{equation}\label{neq:ubldp}
\begin{split}
& \varlimsup_{t\to\infty}\;\frac 1t \log \mathbb  P_x \Big( (\mu_t,Q_t) \in \mathcal{C} \Big)
\le -\inf_{(\mu,Q)\in \mathcal{C}} I(\mu,Q),\\
& \varliminf_{t\to\infty}\;\frac 1t \log \mathbb P_x \Big( (\mu_t,Q_t) \in \mathcal A \Big)
\ge -\inf_{(\mu,Q)\in \mathcal A} I(\mu,Q).
\end{split}
\end{equation}
\end{theorem}

The following theorem, whose proof can be found in Section \ref{section:s topology}, gives the joint LDP when $L^1_+(E)$ is endowed with the strong topology.

\begin{theorem}\label{LDP:strong topology}
Suppose that Assumptions \ref{ass:irreducibility}-\ref{ass:locally finite} and Conditions \ref{condition:ccomp} and \ref{condition:ccomp2} hold. Let $L^1_+(E)$ be endowed with the strong topology. Then under $\mathbb{P}_x$, the law of $(\mu_t,Q_t)$ satisfies an LDP with good and convex rate function $I:\Lambda\to [0,\infty]$.
\end{theorem}

\begin{remark}\label{remark:gamma}
In the above two theorems, we have established the joint LDP when the semi-Markov process starts from a fixed initial state $x\in V$. For any initial distribution $\gamma\in \mathcal{P}(V)$, the joint LDP still holds when making some slight changes to the compactness conditions. In fact, if items (c) in Conditions \ref{condition:ccomp} and \ref{condition:ccomp2} are both replaced by
\begin{itemize}
\item[(c*)] There exists a constant $C_{\gamma}$ such that $\sum_{x\in V}\gamma(x)u_n(x)\le C_{\gamma}$ for any $n\geq 0$,
\end{itemize}
then the conclusions of Theorems \ref{LDP:bounded weak topology} and \ref{LDP:strong topology} remain valid under $\mathbb{P}_{\gamma}$, where $\mathbb{P}_{\gamma}(\cdot)=\sum_{x\in V}\gamma(x)\mathbb{P}_x(\cdot)$ is the probability measure under initial distribution $\gamma$.
\end{remark}

The above two theorems can be applied to obtain the joint LDP for the empirical measure and empirical flow of discrete-time and continuous-time Markov chains with countable state space. For discrete-time Markov chains, we have the following results.

\begin{corollary}
Let $\xi$ be a discrete-time Markov chain satisfying Assumptions \ref{ass:irreducibility}-\ref{ass:locally finite} and Condition \ref{condition:ccomp2}. Let $L^1_+(E)$ be endowed with the bounded weak* topology or strong topology. Then under $\mathbb{P}_x$, the law of $(\mu_t,Q_t)$ satisfies an LDP with good and convex rate function $I:\Lambda\to [0,\infty]$.
\end{corollary}

\begin{proof}
In Remark \ref{equivalence}, we have shown that Conditions \ref{condition:ccomp} and \ref{condition:ccomp2} are equivalent for discrete-time Markov chains. Then the results follow directly from Theorems \ref{LDP:bounded weak topology} and \ref{LDP:strong topology}.
\end{proof}

We have seen that Condition \ref{condition:ccomp} is crucial for the joint LDP of the empirical measure and empirical flow, no matter whether the flow space is endowed with the bounded weak* topology or the strong topology. In general, Condition \ref{condition:ccomp} is difficult to verify because we need to find a sequence of functions $u_n$ satisfying both items (a)-(d), which are related to the embedded chain, and items (e)-(f), we are related to the waiting time distributions. In other words, the conditions imposed on the embedded chain and the conditions imposed on waiting time distributions are intertwined with each other. Next, we provide some novel compactness conditions which can be verified much more easily. In the novel compactness conditions, the ones imposed on the embedded chain and the ones imposed on waiting time distributions can be disassembled and are not intertwined with each other.

The conditions imposed on the embedded chain are as follows.

\begin{condition}\label{condition:ccomp3}
There exists a sequence of functions $u_n:V\to(0,\infty)$ satisfying items (a)-(d) in Condition \ref{condition:ccomp2} and
\begin{itemize}
\item[(e)] There exists a constant $\ell>0$ such that the level set $\big\{x \in V \,:\, \log\hat{u}(x)\leq \ell\big\}$ is finite.
\end{itemize}
\end{condition}

Note that item (e) in Condition \ref{condition:ccomp3} is weaker than item (e) in Condition \ref{condition:ccomp2}. The conditions imposed on the waiting time distributions are as follows.

\begin{condition}
\label{condition:ccomp4}
There exist a probability measure $\psi\in\mathcal{P}(0,\infty)$ and a function $q:V\to (0,\infty)$ such that
\begin{itemize}
\item [(a)] The waiting time distributions $(\psi_x)_{x\in V}$ satisfy $\psi_x(A/q_x) = \psi(A)$ for any $x\in V$ and any Borel set $A\subset (0,\infty)$, where $A/q_x = \{t/q_x:t\in A\}$;
\item [(b)] The probability measure $\psi$ satisfies $\psi\left(e^{\zeta \tau}\right)=\infty$, where $\zeta = \sup\left\{\lambda\ge 0:\psi(e^{\lambda \tau})<\infty\right\}$;
\item[(c)] For each $\ell\in\mathbb{R}$, the level set $\{x\in V:q_x\le\ell\}$ is finite.
\end{itemize}
\end{condition}

\begin{remark}\label{remark:example}
It is easy to check that the following distribution families satisfy items (a) and (b) in Condition \ref{condition:ccomp4}:
\begin{itemize}
\item[(a)] Exponential distribution:
$\psi_x(\mathrm{d}t)=q_xe^{-q_x t}\mathrm{d}t$;
\item[(b)] Gamma distribution:
$\psi_x(\mathrm{d}t)=(q_x^{\alpha}t^{\alpha-1}e^{-q_x t}/\Gamma(\alpha))\mathrm{d}t$ with any $\alpha>0$;
\item[(c)] Dirac distribution:
$\psi_x(\mathrm{d}t)=\delta_{1/q_x}(\mathrm{d}t)$;
\item[(d)] Rayleigh distribution:
$\psi_x(\mathrm{d}t)=(t/q_x^2)e^{-t^2/(2q_x^2)}\mathrm{d}t$.
\end{itemize}
In particular, items (a) and (b) in Condition \ref{condition:ccomp4} automatically hold for continuous-time Markov chains.
\end{remark}

The following theorem shows that the joint LDP also holds under the new compactness conditions given above.

\begin{theorem}\label{theorem:LDP application}
Suppose that Assumptions \ref{ass:irreducibility}-\ref{ass:locally finite} are satisfied.
\begin{itemize}
\item[(a)] Suppose that Conditions \ref{condition:ccomp3} and \ref{condition:ccomp4} both hold. Let $L^1_+(E)$ be equipped with the bounded weak* topology. Then under $\mathbb{P}_x$, the law of $(\mu_t,Q_t)$ satisfies an LDP with good and convex rate function $I:\Lambda\to [0,\infty]$.
\item[(b)] Suppose that Conditions \ref{condition:ccomp2} and \ref{condition:ccomp4} both hold. Let $L^1_+(E)$ be endowed with the strong topology. Then under $\mathbb{P}_x$, the law of $(\mu_t,Q_t)$ satisfies an LDP with good and convex rate function $I:\Lambda\to [0,\infty]$.
\end{itemize}
\end{theorem}

\begin{proof}
(a) By Theorem \ref{LDP:bounded weak topology}, we only need to check Condition \ref{condition:ccomp} for the sequence of functions $u_n$ in Condition \ref{condition:ccomp3}. In fact, items (a)-(d) in Condition \ref{condition:ccomp} are trivial. We next prove items (e) and (f) in Condition \ref{condition:ccomp}.

Let $\theta:(0,\infty)\to (-\infty,\infty]$ be a function defined by
\begin{equation}\label{theta}
\theta(s)=\sup\left\{\lambda\in\mathbb{R}:\psi(e^{\lambda \tau})\le s\right\}.
\end{equation}
By item (b) in Condition \ref{condition:ccomp4}, it is easy to see that the function $\theta_x$ defined in \eqref{def:theta_x} and the function $\theta$ defined above are related by $\theta_x(s)=q_x\theta(s)$. Lemma \ref{lemma:continuous for theta} implies that $\theta$ is an increasing function. It follows from item (e) in Condition \ref{condition:ccomp3} and item (c) in Condition \ref{condition:ccomp4} that the set $K:=\{x\in V:\log\hat{u}(x)\le \ell\}$ is finite and the set $\{x\in V:q_x\le s\}$ is also finite for any $s\in\mathbb{R}$. For any $\ell'>0$, we have
\begin{equation*}
\left\{x\in V:\hat{u}(x)>e^{\ell}\right\}\cap\left\{x\in V:q_x>\frac{\ell'}{\theta(e^{\ell})}\right\}\subseteq\left\{x\in V:q_x\theta(\hat{u}(x))>\ell'\right\}.
\end{equation*}
Then we have
\begin{equation*}
\{x\in V:L\hat{u}(x)\le \ell'\}\subseteq K\cup\left\{x\in V:q_x\le\frac{\ell'}{\theta(e^{\ell})}\right\}.
\end{equation*}
This implies item (e) in Condition \ref{condition:ccomp}.

Take $\eta=1/2$ and $\sigma=-\theta(e^{\ell})/\theta(1/2)$. By Lemma \ref{lemma:continuous for theta}, it is easy to check that $\sigma>0$. Then we have
\begin{equation*}
q_x\theta\left(\hat{u}(x)\right)\ge -\sigma q_x\theta(1/2)-C1_K(x),
\end{equation*}
where $C=1\vee\max_{x\in K}\{-q_x[\theta(\hat{u}(x))+\sigma\theta(1/2)]\}$. By item (a) in Condition \ref{condition:ccomp4}, it is easy to see that $\psi_x(e^{\zeta(x)\tau})=\infty$ for any $x\in V$. This implies item (f) in Condition \ref{condition:ccomp}.

(b) Note that Condition \ref{condition:ccomp2} implies Condition \ref{condition:ccomp3}. Then the proof of (b) follows directly from Theorem \ref{LDP:strong topology}.
\end{proof}

Note that Condition \ref{condition:ccomp4} is easy to verify and we have given several examples for it to hold in Remark \ref{remark:example}. Next we will give some criterions for Condition \ref{condition:ccomp3} to hold. Before doing this, we recall the following geometric ergodic theorem for discrete-time Markov chains, whose proof can be found in \cite[Chapter 15]{meyn2012markov}.

\begin{lemma}[geometric ergodic theorem]\label{lemma:geometric ergodicity}
Suppose that the embedded chain $X$ is irreducible and aperiodic. Then the following three conditions are equivalent:
\begin{itemize}
\item[(a)] There exist a finite set $K\subset V$ and constants $\nu_K>0$, $\rho_K < 1$, and $M_K <\infty$ such that
    \begin{equation*}
    \sup_{x\in K}\left|\left(\sum_{y\in K}p^n_{xy}\right) - \nu_K\right| \le M_K\rho_K^n.
    \end{equation*}
\item[(b)] There exist a finite set $K\subset V$ and a constant $\kappa>1$ such that
    \begin{equation*}
    \sup_{x\in K}\mathbb{E}_x\left[\kappa^{T_K}\right] < \infty,
    \end{equation*}
    where $T_K$ is the first hitting time of $X$ on $K$ (see \eqref{hitting time}).
\item[(c)] There exist a finite set $K\subset V$, constants $b<\infty$, $\lambda<1$, and $c>0$, and a function $u:V\to [c,\infty)$ satisfying the drift condition
    \begin{equation*}\label{drift}
    Pu(x) \le \lambda u(x) + b1_{K}(x), \quad x\in V.
    \end{equation*}
    \end{itemize}
\end{lemma}

The following corollary gives a simple criterion for Condition \ref{condition:ccomp3} to hold.

\begin{corollary}
If item (c) in Lemma \ref{lemma:geometric ergodicity} holds, then Condition \ref{condition:ccomp3} also holds. If the embedded chain $X$ is irreducible and aperiodic, then any one of the three conditions in Lemma \ref{lemma:geometric ergodicity} implies Condition \ref{condition:ccomp3}.
\end{corollary}

\begin{proof}
Let the set $K$, the constants $\lambda,c$, and the function $u$ be as in (c) in Lemma \ref{lemma:geometric ergodicity}. We next check Condition \ref{condition:ccomp3} for the sequence of functions $u_n\equiv u$. Obviously, items (a)-(d) are trivial. Let $\hat{u}=u/Pu$ and $\ell=-(\log\lambda)/2>0$. It is easy to see that
\begin{equation*}
\{x\in V:\log\hat{u}(x)\le\ell\}\subseteq\{x\in V:Pu(x)>\lambda u(x)\}\subseteq K.
\end{equation*}
This completes the proof of this corollary.
\end{proof}

The above corollary shows that geometric ergodicity of the embedded chain $X$ implies Condition \ref{condition:ccomp3}. Hence this condition can also be easily verified by using the classical ergodic theory of Markov chains \cite{chen2006eigenvalues, meyn2012markov, balaji2000multiplicative, wei2019nonzero}.

\subsection{Marginal LDP for the empirical measure and empirical flow}
Thus far, we have established the joint LDP for the empirical measure $\mu_t$ and empirical flow $Q_t$. However, as discussed in Remark \ref{empiricalmeasure}, a more natural definition of the empirical measure $\pi_t:\Omega\to\mathcal{P}(V)$ is given by
\begin{equation*}
\pi_t(x)=\frac{1}{t}\int_0^t1_{(\xi_s=x)}{\rm d}s=\mu_t(x,(0,\infty]),\quad x\in V.
\end{equation*}
The reason why we use $\mu_t$ rather than $\pi_t$ in the study of the joint LDP is that only by using $\mu_t$ can we obtain a concise expression of the rate function.

Next we focus on the marginal LDP for the empirical measure $\pi_t$ and for the empirical flow $Q_t$. By the contraction principle, the rate function of the marginal LDP can be obtained from the rate function $I:\Lambda\to[0,\infty]$ of the joint LDP as defined in \eqref{neq:ubldp}. In \cite{mariani2016large}, the authors gave a variational expression of the rate function $I_1:\mathcal{P}(V)\to[0,\infty]$ for the empirical measure $\pi_t$. Here we will give the variational expression of the rate function $I_2:L^1_+(E)\to[0,\infty]$ for the empirical flow $Q_t$. More importantly, we will also give the explicit expression of $I_2$ when the waiting time distributions satisfy some additional constraints.

Before stating our results, we introduce some notation. Recall that the rate function $I_{DV}:\mathcal{P}(V)\to[0,\infty]$ for the empirical measure of the embedded chain $X$ is the Donsker-Varadhan functional \cite{donsker1975asymptotic1}
\begin{equation*}
I_{DV}(\nu)=\sup_{u\in (0,\infty)^V}\sum_{x\in V}\nu_x\log\left(\frac{u}{Pu}(x)\right).
\end{equation*}
Let $G_x(\lambda)=\log(\psi_x(e^{\lambda \tau}))$ be a function on $\Rnum$ and let
\begin{equation*}
G_x^*(a)=\sup_{\lambda\in\mathbb{R}}(a \lambda-G_x(\lambda)),\;\;\;a\in \mathbb{R}
\end{equation*}
be the Fenchel-Legendre transform of $G_x$. Moreover, let $I_1:\mathcal{P}(V)\to[0,\infty]$ be a functional defined by
\begin{equation}\label{I_1}
I_1(\pi)=\inf_{r>0}\inf_{\nu\in \mathcal{P}(V)}\left(rI_{DV}(\nu)+\sum_{x\in V}r\nu_xG_x^*\left(\frac{\pi_x}{r\nu_x}\right)\right).
\end{equation}
The following proposition, whose proof can be found in Section \ref{section:contraction principle}, gives the marginal LDP for the empirical measure and for the empirical flow.

\begin{proposition}\label{proposition:rate function for empirical flow}
Suppose that Assumptions \ref{ass:irreducibility}-\ref{ass:locally finite} and Condition \ref{condition:ccomp} hold.
\begin{itemize}
\item[(a)] Under $\mathbb{P}_x$, the law of $\pi_t$ satisfies an LDP with good and convex rate function $I_1:\mathcal{P}(V)\to[0,\infty]$;
\item[(b)] Let $L^1_+(E)$ be endowed with the bounded weak* topology. Then under $\mathbb{P}_x$, the law of $Q_t$ satisfies an LDP with good and convex rate function
    \begin{equation*}
    I_2(Q)=\inf_{(\mu,Q)\in\mathcal{D}}I(\mu,Q),\;\;\;Q\in L^1_+(E);
    \end{equation*}
\item[(c)] Let $L^1_+(E)$ be endowed with the strong topology. If Condition \ref{condition:ccomp2} is also satisfied, then under $\mathbb{P}_x$, the law of $Q_t$ satisfies an LDP with good and convex rate function $I_2:L^1_+(E)\to[0,\infty]$.
\end{itemize}
Moreover, if Condition \ref{condition:ccomp4} is also satisfied, then the rate function $I_2$ has the following explicit expression:
\begin{equation}\label{I_2}
I_2(Q)=\left\{\begin{aligned}
&\sum_{x\in V}Q_x H\left(\tilde{Q}_{x,\cdot}\big|p_{x,\cdot}\right)+\sup_{\lambda<\inf_{x\in V}\zeta(x)}\left\{\lambda-\sum_{x\in V}Q_x\log\psi_x\left(e^{\lambda \tau}\right)\right\}, && \text{if } Q^+=Q^-\,,\\
&\infty, && \text{otherwise}.\end{aligned}\right.
\end{equation}
\end{proposition}

\begin{remark}
This proposition shows that the rate function for the empirical flow has the explicit expression \eqref{I_2} when Condition \ref{condition:ccomp4} is satisfied. In fact, \eqref{I_2} may be still valid when Condition \ref{condition:ccomp4} is broken. For example, if $V$ is finite and $\psi_x(e^{\zeta(x)\tau})=\infty$ for any $x\in V$, then similarly to the proof of Proposition \ref{proposition:rate function for empirical flow}, it can be shown that the rate function $I_2$ is also given by \eqref{I_2}.
\end{remark}

\subsection{Examples}
Our abstract theorems can be applied to many specific Markov renewal processes. We next focus on two specific examples: birth and death processes and random walks with confining potential and external force. These two examples can be viewed as direct generalizations of the ones studied in \cite{bertini2015large, bertini2015flows}. Here we will apply Theorems \ref{LDP:bounded weak topology} and \ref{LDP:strong topology} to birth and death processes and will apply Theorem \ref{theorem:LDP application} to random walks with confining potential and external force.

\subsubsection{Birth and death processes}
Consider a birth and death Markov renewal process on the set of nonnegative integers $\mathbb{N} = \{0,1,2\cdots\}$ with transition kernel $(P=(p_{xy})_{x,y\in\mathbb{N}},\Psi=(\psi_x)_{x\in\mathbb{N}})$, where
\begin{equation*}
p_{01}=1,\qquad p_{x,x+1}=p_x>0,\qquad p_{x,x-1}=q_x:=1-p_x>0,\;\;\;x\ge 1.
\end{equation*}
In fact, Assumptions \ref{ass:irreducibility}, \ref{ass:dti}, and \ref{ass:locally finite} are trivial. In addition, it is well-known that Assumption \ref{ass:recurrence} holds if and only if \cite{qian2011Applied}
\begin{equation}\label{recurrent}
\sum_{k=1}^{\infty}\frac{q_1q_2\cdots q_k}{p_1p_2\cdots p_k}=\infty.
\end{equation}
Applying Theorems \ref{LDP:bounded weak topology} and \ref{LDP:strong topology} to the above model, we obtain the following proposition.

\begin{proposition}\label{birth and death process}
Let $p=\varlimsup_{x\to\infty}p_x$. Suppose that $\psi_x(e^{\zeta(x)\tau})=\infty$ for any $x\in V$.
\begin{itemize}
\item[(a)] Let $L^1_+(E)$ be endowed with the bounded weak* topology. Suppose that $p<1/2$ and suppose that there exist constants $\kappa<(4p(1-p))^{-1/2}$, $\sigma>0$, and $\eta\in (0,1)$ such that
    \begin{equation}\label{bwt}
    \varliminf_{x\to\infty}\theta_x(\kappa)=\infty,\qquad\theta_x(\kappa)\ge-\sigma\theta_x(\eta),\quad x\in\mathbb{N}.
    \end{equation}
    Then under $\mathbb{P}_x$, the law of $(\mu_t,Q_t)$ satisfies an LDP with good and convex rate function $I:\Lambda\to[0,\infty]$.
\item[(b)] Let $L^1_+(E)$ be endowed with the strong topology. Let $p_x'=\sup_{k\ge x}p_k$ for each $x$. Suppose that $p=0$ and suppose that there exist constants $\sigma>0$ and $\eta\in (0,1)$ such that
    \begin{equation}\label{st}
    \varliminf_{x\to\infty}\theta_x\left((9p'_{x-1})^{-1/2}\right)=\infty,\qquad \theta_x\left((9p'_{x-1})^{-1/2}\right)\ge-\sigma\theta_x(\eta),\quad x\in\mathbb{N}.
    \end{equation}
    Then under $\mathbb{P}_x$, the law of $(\mu_t,Q_t)$ satisfies an LDP with good and convex rate function $I:\Lambda\to[0,\infty]$.
\end{itemize}
\end{proposition}

\begin{proof}
(a) Since $\varlimsup_{x\to\infty}p_x<1/2$, it is easy to see that \eqref{recurrent} holds. For any $x\in\mathbb{N}$ and $n\ge 0$, let $u_n(x)=a^x$, where $a>1$ is a constant to be chosen later. By Theorem \ref{LDP:bounded weak topology}, we only need to check Condition \ref{condition:ccomp} for the sequence of functions $u_n$. Since $u_n$ do not depend on $n$, items (a)-(d) in Condition \ref{condition:ccomp} are automatically satisfied. Moreover, it is clear that the set $K=\{x\in\mathbb{N}:p_x>p'\}$ is finite. Note that for any $\ell\in\mathbb{R}$,
\begin{equation*}
\{x\in \mathbb{N}:L\hat{u}(x)\le\ell\}\subseteq\{x\in\mathbb{N}:\theta_x(\kappa)\le\ell\}\cup K.
\end{equation*}
This implies item (e) in Condition \ref{condition:ccomp}. Letting $C=\max_{x\in K}(L\hat{u}(x)-\sigma\theta_x(\eta))$, it is easy to see that
\begin{equation*}
L\hat{u}(x)\ge-\sigma\theta_x(\eta)-C1_K(x), \quad x\in\mathbb{N}.
\end{equation*}
On the other hand, $\hat{u}(x)<\infty=\psi_x(e^{\zeta(x)\tau})$ for any $x\in K^c$. These imply item (g) in Condition \ref{condition:ccomp}. Thus far, we have check all items in Condition \ref{condition:ccomp} and thus the desired result follows from Theorem \ref{LDP:bounded weak topology}.

(b) Let $q_x'=1-p_x'$ for each $x$. For any $n\ge0$, let $u_n(0)=u_n(1)=1$ and $u_n(x)=(\prod_{k=1}^{x-1}(q_k'/p_k'))^{1/2}$ for any $x\ge 2$. Similarly, we only need to check item (e) in Condition \ref{condition:ccomp2} and items (e) and (f) in Condition \ref{condition:ccomp} for such $u_n$. Note that $q'_{x-1}\ge1/2$ for sufficiently large $x$. Then we have
\begin{equation*}
\hat{u}(x)=\frac{1}{p_x(q'_{x}/p'_{x})^{1/2}+q_x(p'_{x-1}/q'_{x-1})^{1/2}}\ge\frac{1}{(q'_xp'_x)^{1/2}+(2p'_{x-1})^{1/2}}\ge\frac{1}{3(p'_{x-1})^{1/2}}.
\end{equation*}
Since $\lim_{x\to\infty}p'_{x-1}=0$, item (e) in Condition \ref{condition:ccomp2} holds. Let $K=\{x\in\mathbb{N}:q'_{x-1}<1/2\}\cup\{0,1\}$ is finite. The rest of the proof is similar to the proof of (a).
\end{proof}

We emphasize that if $\psi_x$ is chosen to be an exponential distribution for each $x$, then Proposition \ref{birth and death process} reduces to the results obtained in \cite{bertini2015large, bertini2015flows}. Moreover, if $\psi_x$ is chosen to be the gamma distribution \begin{equation*}
\psi_x(\mathrm{d}t)=\left(q_x^{\alpha_x}t^{\alpha_x-1}e^{-q_xt}/\Gamma(\alpha_x)\right)\mathrm{d}t
\end{equation*}
for each $x$, where $q_x,\alpha_x$ are parameters depending on $x$, then we can give the following more specific characterizations of conditions \eqref{bwt} and \eqref{st}.

\begin{lemma}
\begin{itemize}
\item[(a)] Let $\kappa_0 = [1+(4p(1-p))^{-1/2}]/2$. Suppose that
    \begin{equation*}
    \varliminf_{x\to\infty}q_x\left(1-\kappa_0^{-1/\alpha_x}\right)=\infty
    \end{equation*}
    and suppose that the parameters $\alpha_x$ have a uniform positive lower bound, i.e. $\alpha_x\ge c$ for some constant $c>0$ and any $x\in\mathbb{N}$. Then condition \eqref{bwt} holds.
\item[(b)] Let $p_x'=\sup_{k\ge x}p_k$ for each $x$. Suppose that
    \begin{equation*}
    \varliminf_{x\to\infty}q_x\left(1-(9p'_{x-1})^{1/(2\alpha_x)}\right)=\infty
    \end{equation*}
    and suppose that the parameters $\alpha_x$ have a uniform positive lower bound. Then condition \eqref{st} holds.
\end{itemize}
\end{lemma}

\begin{proof}
Here we only give the proof of (a); the proof of (b) is similar. Taking $\kappa=\kappa_0$ in Proposition \ref{birth and death process}, straightforward computations show that
\begin{equation*}
\theta_x(t)=q_x\left(1-t^{-1/\alpha_x}\right),\quad t>0.
\end{equation*}
This shows that
\begin{equation*}
\varliminf_{x\to\infty}\theta_x(\kappa) = \varliminf_{x\to\infty}q_x(1-\kappa_0^{-1/\alpha_x}) = \infty.
\end{equation*}
Moreover, letting $\sigma=\kappa_0^{-1/c}>0$ and $\eta=1/\kappa_0\in(0,1)$, we have
\begin{equation*}
\theta_x(\kappa_0)=q_x\left(1-\kappa_0^{-1/\alpha_x}\right)=-\kappa_0^{-1/\alpha_x}q_x\left(1-\eta^{-1/\alpha_x}\right)\ge-\sigma\theta_x(\eta).
\end{equation*}
This completes the proof of this lemma.
\end{proof}

\subsubsection{Random walks with confining potential and external force}
We now apply our main theorems to a nearest neighbor random walk on $\mathbb{Z}^d$ with confining potential and external force, whose transition kernel $(P=(p_{xy})_{x,y\in\mathbb{Z}^d},\psi=(\psi_x)_{x\in\mathbb{Z}^d})$ has the form of
\begin{equation*}
p_{xy}=\frac{1}{C_x}\exp\left\{-\frac{1}{2}(U(y)-U(x))+\frac{1}{2}F(x,y)\right\},\quad (x,y)\in E,
\end{equation*}
where $E=\{(x,y)\in\mathbb{Z}^d\times\mathbb{Z}^d:|x-y|=1\}$ is the collection of nearest neighbours in $\mathbb{Z}^d$, $U:\mathbb{Z}^d\to\mathbb{R}$ is a potential function satisfying $\sum_{y\in\mathbb{Z}^d}e^{-U(y)}<\infty$, $F\in L^{\infty}(E)$ represents the external force, and \begin{equation*}
C_x=\sum_{y:|x-y|=1}e^{-(U(y)-U(x))/2+F(x,y)/2},\;\;\;x\in\mathbb{Z}^d
\end{equation*}
are normalization constants. It is clear that $U$ has compact level sets, i.e. for any $\ell\in\mathbb{R}$, the set $\{x\in\mathbb{Z}^d:U(x)\le \ell\}$ is finite. We emphasize that if $\psi_x(\mathrm{d}t) = C_xe^{-C_xt}\mathrm{d}t$ is chosen to be an exponential distribution for each $x$, then the above model reduces to the Markov chain model studied in \cite[Section 10.2]{bertini2015flows}. We next focus on general waiting time distributions. Applying Theorem \ref{theorem:LDP application} to the above model, we obtain the following proposition.

\begin{proposition}\label{random walk}
Let $r(x)=\sum_{y:|x-y|=1}e^{-(U(y)-U(x))/2}$.
\begin{itemize}
\item[(a)] Let $L^1_+(E)$ be endowed with the bounded weak* topology. If
    \begin{equation*}
    \varliminf_{|x|\to\infty}r(x) > 2de^{\|F\|_{\infty}},
    \end{equation*}
    then Condition \ref{condition:ccomp3} holds. Moreover, if Condition \ref{condition:ccomp4} also holds, then under $\mathbb{P}_x$, the law of $(\mu_t,Q_t)$ satisfies an LDP with good and convex rate function $I:\Lambda\to[0,\infty]$.
\item[(b)] Let $L^1_+(E)$ be endowed with the strong topology. If
    \begin{equation*}
    \lim_{|x|\to\infty}r(x) = \infty,
    \end{equation*}
    then Condition \ref{condition:ccomp2} holds. Moreover, if Condition \ref{condition:ccomp4} also holds, then under $\mathbb{P}_x$, the law of $(\mu_t,Q_t)$ satisfies an LDP with good and convex rate function $I:\Lambda\to[0,\infty]$.
\end{itemize}
\end{proposition}

\begin{proof}
Here we only give the proof of (a); the proof of (b) is similar. We first check Condition \ref{condition:ccomp3}. Let $u_n(x)=e^{U(x)/2}$ for any $x\in\mathbb{Z}^d$ and $n\ge 0$. Since $u_n$ do not depend on $n$, items (a)-(d) in Condition \ref{condition:ccomp2} are automatically satisfied. Note that
\begin{equation*}
\hat{u}(x)=\frac{e^{U(x)/2}}{\sum_{y:(x,y)\in E}p_{xy}e^{U(y)/2}}=\frac{\sum_{y:|x-y|=1}e^{-(U(y)-U(x))/2+F(x,y)/2}}{\sum_{y:|x-y|=1}e^{F(x,y)/2}}\ge\frac{r(x)}{2de^{\|F\|_{\infty}}}.
\end{equation*}
Since $\varliminf_{|x|\to\infty}\hat{u}(x)>1$, there exists $\kappa>1$ such that $\varliminf_{|x|\to\infty}\hat{u}(x)\geq\kappa$. Letting $\ell=\log\kappa>0$, it is easy to see that the set $\{x\in\mathbb{Z}^d:\log\hat{u}(x)\le\ell\}$ is finite. This implies item (e) in Condition \ref{condition:ccomp3}.

We next prove the joint LDP. It is easy to see that Assumptions \ref{ass:irreducibility}, \ref{ass:dti}, and \ref{ass:locally finite} are trivial. Similarly to the above proof, we can also validate item (c) in Lemma \ref{lemma:geometric ergodicity} with $u(x)=e^{U(x)/2}$. It then follows from \cite[Theorems 8.0.2 and 15.0.1]{meyn2012markov} that Assumption \ref{ass:recurrence} holds. By Theorem \ref{theorem:LDP application}, we complete the proof of this lemma.
\end{proof}

Note that if $\psi_x(\mathrm{d}t) = C_xe^{-C_xt}\mathrm{d}t$ is chosen to be an exponential distribution for each $x$, then the above model reduces to the continuous-time Markov chain model studied in \cite[Section 10.2]{bertini2015flows}. It is easy to check that
\begin{equation*}
e^{-\|F\|_{\infty}/2}r(x)\le C_x\le e^{\|F\|_{\infty}/2}r(x).
\end{equation*}
Then $\lim_{|x|\to\infty}r(x)=\infty$ if and only if $\lim_{|x|\to\infty}C_x=\infty$. Hence our compactness conditions in item (a) exactly coincide with the ones proposed in \cite{bertini2015flows} when $L^1_+(E)$ is endowed with the bounded weak* topology. However, since we do not need to verify Condition \ref{condition:strong topology}, our compactness conditions in item (b) is weaker than the ones proposed in \cite{bertini2015flows} when $L^1_+(E)$ is endowed with the strong topology.

\begin{remark}
When $U\in C^1(\mathbb{R}^d)$, we consider the orthogonal decomposition
\begin{equation*}
\nabla U(x)=\langle \nabla U(y),\hat{y}\rangle \hat{y}+W(y),\quad y\in\mathbb{R}^d\setminus\{0\},
\end{equation*}
where $\hat{y}=y/|y|$, $\langle y,W(y)\rangle=0$, and $\langle\cdot,\cdot\rangle$ is the standard inner product in $\mathbb{R}^d$. We say that the potential $U\in C^1(\mathbb{R}^d)$ has diverging radial variation which dominates the transversal variation \cite{bertini2015flows} if there exist $\alpha\in[0,1)$ and $C>0$ such that
\begin{equation*}
\lim_{|x|\to\infty}\langle \nabla U(y),\hat{y}\rangle=\infty,\qquad|W(y)|\le\frac{\alpha}{d^{1/2}}\langle\nabla U(y),\hat{y}\rangle+C.
\end{equation*}
In fact, it follows from \cite[Lemma 10.3]{bertini2015flows} that if $U\in C^1(\mathbb{R}^d)$ has diverging radial variation which dominates the transversal variation, then $\lim_{|x|\to\infty}r(x)=\infty$.
\end{remark}

\section{Proof of Theorem \ref{LDP:bounded weak topology}}\label{section:bwt topology}
Note that $L^1(E)$ endowed with the bounded weak* topology is a locally convex, complete linear topological space and a completely regular space, i.e. for every closed set $C \subset L^1(E)$ and every element $Q \in L^1(E) \setminus C$, there exists a continuous function $f:L^1(E) \to[0,1]$ such that $f(Q)=1$ and $f(Q')=0$ for any $Q'\in C$) \cite[Theorem 2.7.2]{megginson2012introduction}.

It is a well-known result that if an exponentially tight family of probability measures satisfies a weak LDP with rate function $I$, then $I$ is good and the (full) LDP holds \cite[Lemma 1.2.18]{dembo1998large}. Hence to prove Theorem \ref{LDP:bounded weak topology}, we will first prove the exponential tightness for the empirical measure and empirical flow under Condition \ref{condition:ccomp}, and then prove the weak joint LDP without any compactness conditions. We will directly consider the case where $\xi$ starts from a general initial distribution $\gamma$ (see Remark \ref{remark:gamma}).

\subsection{Exponential local martingales}\label{subsection:elm}
We start by considering the change of probability measures for Markov renewal processes. Let $\Gamma$ be the set of measurable functions $(F,h)$ defined by
\begin{equation}\label{Gamma}
\begin{split}
\Gamma = \Bigg\{(F,h)\, : \: & F:E\to\mathbb{R} \,\,\text{such that } \sum_{z\in V} p_{xz} \, e^{F(x,z)} < \infty,\\ & h:V\times (0,\infty]\to \mathbb{R} \text{ such that } \int_{(0,\infty)}e^{sh(x,s)}\psi_x({\rm d}s)<\infty \text{ for any } x\in V\Bigg\}.
\end{split}
\end{equation}
For any $(F,h)\in \Gamma$, let $g_{F,h}:V\to \mathbb{R}$ be a function given by
\begin{equation}\label{def:g}
g_{F,h}(x)=\log\sum_{z\in V}p_{xz}e^{F(x,z)}+\log\int_{(0,\infty)}e^{sh(x,s)}\psi_x({\rm d}s).
\end{equation}
To proceed, we define a new transition kernel $(P^F,\Psi^h)$ as
\begin{equation}\label{pF psih}
p^F_{xy} = \frac{p_{xy} \, e^{F(x,y)}}{\sum_{z\in V} p_{xz} \, e^{F(x,z)}}, \qquad
\psi^h_x({\rm d}u) = \frac{e^{u \, h(x,u)} \, \psi_{x}({\rm d}u)}{\int_{(0,\infty)} e^{s \, h(x,s)} \, \psi_{x}({\rm d}s)},
\end{equation}
and let $\mathbb{P}^{F,h}_{x}$ be the probability measure under which $(X,\tau)$ is a Markov renewal process with transition kernel $(P^F,\Psi^h)$ and initial state $X_0=x$. Note that the semi-Markov process $\xi$ may be explosive under $\mathbb{P}^{F,h}_x$. As a result, we need to consider $\mathbb{P}^{F,h}_x$ and $\mathbb{P}_x$ restricted to the set $\{N_t<\infty\}$, i.e.
\begin{equation*}
\mathbb{P}^{F,h}_{x,t}(A)=\mathbb{P}^{F,h}_x(A\cap\{N_t<\infty\}),\quad\mathbb{P}_{x,t}(A)=\mathbb{P}_x(A\cap\{N_t<\infty\}),\quad A\in\mathcal{F}_{\infty}.
\end{equation*}
Moreover, we denote by $\mathbb{P}^{F,h}_{x,t}|_{\mathcal{F}_{N_t+1}}$ and $\mathbb{P}_{x,t}|_{\mathcal{F}_{N_t+1}}$ the restrictions of $\mathbb{P}^{F,h}_{x,t}$ and $\mathbb{P}_{x,t}$ to $\mathcal{F}_{N_t+1}$, respectively. It is easy to verify that $\mathbb{P}^{F,h}_{x,t}|_{\mathcal{F}_{N_t+1}}$ is absolutely continuous with respect to $\mathbb{P}_{x,t}|_{\mathcal{F}_{N_t+1}}$ and has the Radon-Nykodim derivative
\begin{align}
&\; \frac{1}{t}\log\frac{{\rm d}\mathbb{P}^{F,h}_{x,t}|_{\mathcal{F}_{N_t+1}}}{{\rm d}\mathbb{P}_{x,t}|_{\mathcal{F}_{N_t+1}}}  \notag\\
=&\;\frac{1}{t}\sum_{i=1}^{N_t+1}\log\frac{e^{F(X_{i-1},X_{i})}}{\sum_{z\in V} p_{X_{i-1},z}\, e^{F(X_{i-1},z)}}\,
+\frac{1}{t}\sum_{i=1}^{N_t+1}\log\frac{ e^{\tau_i \, h(X_{i-1},\tau_i)}}{ \int_{(0,\infty)} e^{s \, h(X_{i-1},s)} \, \psi_{X_{i-1}
}({\rm d}s)}\notag\\
=&\;\sum_{(x,y)\in E}  Q_t(x,y)  \left[F(x,y)-\left(\log\sum_{z\in V}p_{xz}e^{F(x,z)}+\log\int_{(0,\infty)}e^{sh(x,s)}\psi_x({\rm d}s)\right)\right]\notag\\
&\;+\sum_x \int_{(0,\infty]}h(x,u)\mu_t(x,{\rm d}u)+\frac{\tau_{N_t+1}-t+S_{N_t}}{t}h(X_{N_t},\tau_{N_t+1})\notag\\\label{formula:change measure}
=&\;\langle Q_t,F-g_{F,h}\rangle+\langle \mu_t,h\rangle+\frac{S_{N_t+1}-t}{t}h(X_{N_t},\tau_{N_t+1})
\end{align}
where $\langle Q_t,F\rangle=\sum_{(x,y)\in E}Q_t(x,y)F(x,y)$ and $\langle Q_t,g_{F,h}\rangle=\sum_{(x,y)\in E}Q_t(x,y)g_{F,h}(x)$.

For convenience, for any Borel measurable function $f$ on $V\times (0,\infty]$, let
\begin{equation*}
\langle\hat{\mu}_t,f\rangle =\langle \mu_t,f\rangle+\frac{S_{N_t+1}-t}{t}f(X_{N_t},\tau_{N_t+1})=
\frac{1}{t}\sum_{k=1}^{N_t+1}\tau_{k}f(X_{k-1},\tau_k).
\end{equation*}
Then as an immediate consequence of the Radon-Nykodim derivative \eqref{formula:change measure}, we obtain the following result.

\begin{lemma}\label{lem:mhF}
For any $(F,h)\in \Gamma$ and $t\ge 0$, let $\mathcal{M}^{F,h}_t: \Omega \to (0,\infty)$ be a function defined by
\begin{equation}
\label{expm2}
\mathcal{M}^{F,h}_t = \exp\Big\{t\big[\langle Q_t,F-g_{F,h}\rangle+\langle\hat{\mu}_t,h\rangle\big]\Big\}.
\end{equation}
Then for each $x\in V$ and $t\ge 0$, we have $\mathbb E_{x} \big( \mathcal{M}^{F,h}_t \big) \le 1$.
\end{lemma}

\begin{proof}
Note that $\{N_t<\infty\}\in \mathcal{F}_{N_t+1}$. It follows from \eqref{formula:change measure} that
$\mathbb E_{x}( \mathcal{M}_t ^{F,h})=\mathbb{P}^{F,h}_{x}(N_t<\infty)\leq
1$.
\end{proof}

In fact, it is easy to check that under $\mathbb P_{x}$, the process $\mathcal{M}^{F,h}$ is a positive local martingale and a supermartingale  with respect to $(\mathcal{F}_{N_t+1})_{t\ge 0}$. The next statement can be deduced from Lemma \ref{lem:mhF}  by choosing specific $F$ and $h$.

\begin{lemma}\label{t:em1}
Let $u\colon V \to (0,\infty)$ be a function satisfying $Pu(x)<\infty$ for any $x\in V$. Let $A\subset V$ be a set satisfying $u(x)/Pu(x)\le \psi_x(e^{\zeta(x)\tau})$ for any $x\in A^c$. Let $F(x,y) = \log [ u(y)/u(x)]$ for any $(x,y)\in E$ and
\begin{equation}\label{h u A}
h(x,s)=h^{u,A}(x,s):=\frac{\log \frac{u}{Pu}(x)}{s}1_A(x)+L\frac{u}{Pu}(x)1_{A^c}(x),\quad (x,s)\in V\times (0,\infty].
\end{equation}
For any $t\ge 0$, let $\mathcal M^{u,A}_t: \Omega \to (0,\infty)$ be a function defined as
\begin{equation}
\label{martingle2}
\mathcal M^{u,A}_t =  \frac{u(X_{N_t+1})}{u(X_0)}
\exp\Big\{t\left\langle \hat{\mu}_t, h^{u,A}\right\rangle\Big\}.
\end{equation}
Then $(F,h)\in \Gamma$ and $\mathbb E_{x} \big(\mathcal M^{u,A}_t \big) \le 1$ for any $x\in V$ and $t\ge 0$.
\end{lemma}
\begin{proof}
For any $x\in V$, it is clear that $\sum_{y\in V}p_{xy}e^{F(x,y)}=Pu(x)/u(x)$ and
\begin{equation*}
\int_{(0,\infty)}e^{sh(x,s)}\psi_x(\mathrm{d}s)=\frac{u(x)}{Pu(x)}1_A(x)+\psi_x\left(e^{\theta_x(u(x)/Pu(x))\tau}\right)1_{A^c}(x).
\end{equation*}
Therefore, by item (c) in Lemma \ref{lemma:continuous for theta}, we have $\int_{(0,\infty)}e^{sh(x,s)}\psi_x(\mathrm{d}s)=u(x)/Pu(x)$. This implies that $(F,h)\in \Gamma$ and $g_{F,h}\equiv 0$. On the other hand, it is easy to check that $\langle Q_t,F\rangle=u(X_{N_t+1})/u(X_0)$. This completes the proof of this lemma.
\end{proof}

\subsection{Exponential tightness}\label{subsection:et}
We will next prove the exponential tightness of the empirical measure and empirical flow under Condition \ref{condition:ccomp} with item (c) is replaced by item (c*) in Remark \ref{remark:gamma}.

Let the function $\hat{u}$, the sequence of functions $u_n$, the set $K$, and the constants $c,C_{\gamma},\sigma,\eta$ be as in Condition \ref{condition:ccomp} and item (c*) in Remark \ref{remark:gamma}. Recall the definition of $h^{u,A}$ in \eqref{h u A}. Let $h^{\hat{u}}:V\times (0,\infty]\to \mathbb{R}$ be a function defined by
\begin{equation*}
h^{\hat{u}}(x,s)=\lim_{n\to \infty}h^{u_n,K}(x,s)=\frac{\log\hat{u}(x)}{s}1_K(x)+L\hat{u}(x)1_{K^c}(x),\quad (x,s)\in V\times (0,\infty],
\end{equation*}
where we have used the condition that $u_n/Pu_n$ converges pointwise to $\hat{u}$.

\begin{lemma}\label{t:letem}
Assume Condition \ref{condition:ccomp} to hold. Then
\begin{equation}\label{sirena}
\mathbb E_{\gamma}    \Big( e^{ t  \langle\hat{\mu}_t, h^{\hat{u}}\rangle } \Big) \le
\frac {C_{\gamma}}{c}\,,\qquad	\mathbb E_{\gamma}    \Big( e^{ t  \langle\mu_t, h^{\hat{u}}\rangle } \Big) \le
\frac {C_{\gamma}e^N}{c}\,,
\end{equation}
where $N=1\vee (-\inf_{x\in K}\log\hat{u}(x))$.
\end{lemma}

\begin{proof}
By Fatou's lemma, we have
\begin{equation}\label{eq1 for lemma 4.3}
\mathbb E_{\gamma} \Big( e^{ t \langle\hat{\mu}_t, h^{\hat{u}}\rangle } \Big)\le\sum_{x\in V}\gamma(x)\varliminf_{n\to\infty}\mathbb E_{x} \Big( e^{ t \langle\hat{\mu}_t,h^{u_n,K} \rangle } \Big).
\end{equation}
For any $x\in K^c$, it follows from items (d) and (f) in Condition \ref{condition:ccomp} that there exists $n_x\in \mathbb{N}$ such that $u_n(x)/Pu_n(x)<\psi_x(e^{\zeta(x)\tau})$ for any $n\ge n_x$. By Lemma \ref{t:em1} and item (b) in Condition \ref{condition:ccomp}, we have
\begin{equation}\label{eq2 for lemma 4.3}
\mathbb E_{x} \Big( e^{ t \langle\hat{\mu}_t,h^{u_n,K} \rangle } \Big)=\mathbb E_{x} \left( \frac{u_n(X_0)}{u_n(X_{N_t+1})}\mathcal{M}^{u_n,K}_t \right)\le \frac{u_n(x)}{c},\quad n\ge n_x.
\end{equation}
Combining \eqref{eq1 for lemma 4.3}, \eqref{eq2 for lemma 4.3}, and item (c*) in Remark \ref{remark:gamma}, we have
\begin{equation*}
\mathbb E_{\gamma} \Big( e^{ t \langle\hat{\mu}_t, h^{\hat{u}}\rangle } \Big)\le \sum_{x\in V}\gamma(x)\frac{u_n(x)}{c}	\le \frac {C_{\gamma}}{c}.
\end{equation*}
Moreover, by item (f) in Condition \ref{condition:ccomp}, we obtain
\begin{align}\label{h hat u}
h^{\hat{u}}(x,s)\ge-\frac{N}{s}1_K(x)-\sigma L\eta(x)1_{K^c}(x).
\end{align}
Since $N>0$ and $L\eta(x)<0$ for any $x\in K^c$, we have
\begin{equation*}
\begin{split}
\langle \hat{\mu}_t, h^{\hat{u}}\rangle=&\;\langle \mu_t, h^{\hat{u}}\rangle+\frac{S_{N_t+1}-t}{t}h^{\hat{u}}(X_{N_t},\tau_{N_t+1})\\
\ge &\;\langle \mu_t, h^{\hat{u}}\rangle-\frac{S_{N_t+1}-t}{t}\frac{N}{\tau_{N_t+1}}\\
\ge &\;\langle \mu_t, h^{\hat{u}}\rangle-\frac{N}{t}.
\end{split}
\end{equation*}
This implies the second inequality in \eqref{sirena}.
\end{proof}

For any $B\subset V$, let $(T^k_B)_{k\ge 0}$ be a sequence of stopping times defined by
\begin{equation}\label{hitting time}
T^0_B=0,\qquad T^{k}_B=\inf\left\{n\ge T^{k-1}_B+1:X_n\in B\right\},\quad k\ge 1.
\end{equation}
Clearly, $T^k_B$ is the $k$th hitting time of $(X_n)_{n\ge 0}$ on the set $B$. The first hitting time $T^1_B$ is always abbreviated as $T_B$ in what follows. If we only focus on the behavior of $(X,\tau)$ in the set $B$, we obtain a new process $(\bar{X},\bar{\tau})=\{(\bar{X}_k)_{k\ge 0},(\bar{\tau}_{k})_{k\ge 1}\}$, which is defined by
\begin{equation}\label{hat X tau}
\bar{X}_k=X_{T^k_B},\qquad \bar{\tau}_{k+1}=\sum_{i=T^k_B+1}^{T^{k+1}_B}\tau_i,\quad k\ge 0.
\end{equation}

\begin{proposition}\label{proposition:subset markov renewal process}
Suppose that $(X,\tau)$ satisfies Assumptions \ref{ass:irreducibility} and \ref{ass:recurrence}. Then $\{(\bar{X}_k)_{k\ge 1},(\bar{\tau}_{k})_{k\ge 2}\}$ is a Markov renewal process and also satisfies Assumptions \ref{ass:irreducibility} and \ref{ass:recurrence}. In other words, we have
\begin{itemize}
\item[(a)] $(\bar{X}_k)_{k\ge 1}$ is an irreducible and recurrent discrete-time Markov chain with state space $B$ and transition probability matrix $(\bar{p}_{xy})_{x,y\in B}$ given by $\bar{p}_{xy}=\mathbb{P}_x\left(X_{T_B}=y\right)$.
\item[(b)] $(\bar{\tau}_{k})_{k\ge 2}$ is a sequence of positive and finite random variables such that conditioned on $(\bar{X}_k)_{k\ge 1}$ the random variables $(\bar{\tau}_{k})_{k\ge 2}$ are independent and the waiting time matrix $(\bar{\psi}_{xy})_{x,y\in B}$ is given by
\begin{equation*}
\bar{\psi}_{xy}(\cdot)=\mathbb{P}\left(\sum_{i=1}^{T_B}\tau_i\in \cdot\Bigg|X_0=x,X_{T_B}=y\right),\quad x,y\in B,
\end{equation*}
where the definition of the waiting time matrix can be found in Definition~\ref{def:markov renewal process}.
\end{itemize}
\end{proposition}

\begin{proof}
(a) Since $(X_k)_{k\ge 0}$ is irreducible and recurrent, we have $T^k_B<\infty$, $\mathbb{P}_x$-a.s. for any $x\in V$ and $k\ge 0$. For any $n\ge 0$ and $x_0,x_1,\cdots,x_{n+1}\in B$,
\begin{align*}
\mathbb{P}\left(\bar{X}_{n+1}=x_{n+1}\Big|\bar{X}_{n}=x_n,\cdots,\bar{X}_0=x_0\right)=&\;\mathbb{P}\left(X_{T^{n+1}_B}=x_{n+1}\Big|X_{T^n_B}=x_n,\cdots,X_{T^0_B}=x_0\right)\\
=&\;\mathbb{P}_{x_n}\left(X_{T_B}=x_{n+1}\right),
\end{align*}
where the last step follows from the strong Markov property. In fact, we can obtain the recurrence of $(\bar{X}_k)_{k\ge 1}$ directly from the recurrence of $(X_k)_{k\ge 0}$. For any $x,y\in B$, since $(X_k)_{k\ge 0}$ is irreducible, there exist a positive integer $n\ge 1$ and a sequence of states $x_0,x_1,\cdots,x_n\in V$ with $x_0 = x$ and $x_n = y$ such that $p_{x_0,x_1}p_{x_1,x_2}\cdots p_{x_{n-1},x_n}>0$. Select all states in $\{x_k\}_{0\le k\le n}$ in the set $B$ and write them as $x_{i_1},\cdots,x_{i_s}$ with $0=i_1<i_2<\cdots<i_s=n$. Then
\begin{equation*}
\bar{p}_{x_{i_j}x_{i_{j+1}}}\ge p_{x_{i_j}x_{i_j+1}}\cdots p_{x_{i_{j+1}-1}x_{i_{j+1}}}>0,\quad 1\le j\le s-1.
\end{equation*}
This implies that $(\bar{X}_k)_{k\ge 1}$ is irreducible.

(b) Since $T^k_B<\infty$, $\mathbb{P}_x$-a.s. for any $x\in V$ and $k\ge 0$, it is easy to see that $(\bar{\tau}_{k+1})_{k\ge 0}$ is a sequence of positive and finite random variables. Note that conditioned on $(X_k)_{k\ge 0}$ the random variables $(\tau_{k})_{k\ge 1}$ are independent and have distribution $\mathbb{P}\left(\tau_{i+1}\in \cdot \, | \, (X_k)_{k\geq 0}\right) = \psi_{X_{i},X_{i+1}}(\cdot)$. For any $f\in\mathcal{B}_b(0,\infty)$, we have
\begin{equation}\label{E f bar tau}
\mathbb{E}\left[f(\bar{\tau}_k)|(X_k)_{k\ge 0}\right]=\int f\Bigg(\sum_{i=T^{k-1}_B+1}^{T^k_B}t_i\Bigg)\prod_{i=T^{k-1}_B+1}^{T^k_B}\psi_{X_{i-1}X_{i}}({\rm d}t_{i}),
\end{equation}
where we use the fact that $({T^k_B})_{k\ge 0}$ is $\sigma((X_k)_{k\ge 0})$-measurable. For any $n\ge 2$ and $f_1,\cdots,f_n\in \mathcal{B}_b(0,\infty)$, similarly to \eqref{E f bar tau}, we have
\begin{equation}\label{E f1}
\begin{split}
\mathbb E\left[f_1(\bar{\tau}_1)\cdots f_n(\bar{\tau}_n)|(X_k)_{k\ge 0}\right]
& =\int f_1\Bigg(\sum_{i=T^0_B+1}^{T^1_B}t_i\Bigg)\cdots f_n\Bigg(\sum_{i=T^{n-1}_B+1}^{T^n_B}t_i\Bigg)\prod_{i=T^{0}_B+1}^{T^n_B}\psi_{X_{i-1}X_{i}}({\rm d}t_{i})\\
&=\mathbb{E}\left[f_1(\bar{\tau}_1)|(X_k)_{k\ge 0}\right]\cdots\mathbb{E}\left[f_n(\bar{\tau}_n)|(X_k)_{k\ge 0}\right].
\end{split}
\end{equation}
Note that $\bar{X}_k=X_{T^k_B}$. By the strong Markov property of $(X_k)_{k\geq 0}$, it is clear that conditioned on $\{\bar{X}_{k-1}=x,\bar{X}_k=y\}$ the random variable $\mathbb{E}\left[f(\bar{\tau}_k)|(X_k)_{k\ge 0}\right]$ is independent of $(X_k)_{0\le k\le T^{k-1}_B}$ and $(X_k)_{ k\ge T^{k}_B}$. Moreover, for any $x,y\in B$, the distribution of $\mathbb{E}[f(\bar{\tau}_k)|(X_k)_{k\ge0}]$ conditioned on $\{\bar{X}_{k-1}=x,\bar{X}_k=y\}$ is equal to the distribution of $\mathbb{E}[f(\bar{\tau}_1)|(X_k)_{k\ge0}]$ conditioned on $\{\bar{X}_0=x,\bar{X}_1=y\}$. It follows from \eqref{E f1} that
\begin{align*}
\mathbb E\left[f_1(\bar{\tau}_1)\cdots f_n(\bar{\tau}_n)\Big|(\bar{X}_k)_{k\ge 0}\right]&=\mathbb{E}\left[\mathbb E\left[f_1(\bar{\tau}_1)\cdots f_n(\bar{\tau}_n)|(X_k)_{k\ge 0}\right]\Big|(\bar{X}_k)_{k\ge 0}\right]\\
&=\mathbb E\left[f_1(\bar{\tau}_1)\Big|(\bar{X}_k)_{k\ge 0}\right]\cdots \mathbb E\left[f_n(\bar{\tau}_n)\Big|(\bar{X}_k)_{k\ge 0}\right].
\end{align*}
Moreover, for any $x,y\in B$, we have
\begin{align*}
\mathbb{E}\left[f(\bar{\tau}_k)\Big|\bar{X}_{k-1}=x,\bar{X}_k=y\right]&=\mathbb{E}\left[\mathbb{E}\left[f(\bar{\tau}_k)|(X_k)_{k\ge0}\right]\Big|\bar{X}_{k-1}=x,\bar{X}_k=y\right]\\
&=\mathbb{E}\left[f(\bar{\tau}_1)\Big|\bar{X}_0=x,\bar{X}_1=y\right].
\end{align*}
This completes the proof of this proposition.
\end{proof}

\begin{remark}\label{remark:delayed Markov renewal process}
In the above proposition, we have proved that conditioned on $(\bar{X}_k)_{k\ge0}$, the random variables $(\bar{\tau}_k)_{k\ge 1}$ are independent. This implies that $(X,\tau)=\{(\bar{X}_k)_{k\ge 0},(\bar{\tau}_k)_{k\ge 1}\}$ is a delayed Markov renewal process, whose transition probability matrix and waiting time matrix of the first step is different from the remaining steps \cite[Chapter 4.12]{janssen2006applied}.
\end{remark}

\begin{lemma}\label{lemma:sbuset exponential}
Let $K$ be a finite subset of $V$. Then for any $\ell\in \mathbb{N}$, there exists a real sequence $A_{\ell}\uparrow\infty$ such that
\begin{equation}
\varlimsup_{t\to \infty}\frac{1}{t}\log\mathbb{P}_{\gamma}(\langle Q_t,1_K\rangle>A_{\ell})\le-\ell,
\end{equation}
where $\langle Q_t,1_K\rangle=\sum_{(x,y)\in E}Q_t(x,y)1_K(x)$.
\end{lemma}

\begin{proof}
Let $(\bar{X},\bar{\tau})$ be defined as in \eqref{hat X tau} with $B=K$. By Remark \ref{remark:delayed Markov renewal process}, we have seen that $(\bar{X},\bar{\tau})$ is a delayed Markov renewal process. Similarly to Markov renewal processes, we can also define the $n$th jump time $\bar{S}_n$, the number $\bar{N}_t$ of jumps up to time $t$, and the empirical flow $\bar{Q}_t$ for $(\bar{X},\bar{\tau})$. Note that $\sum_{i=1}^{N_t}1_K(X_i)$ is the number of times that $\xi$ jumps into the set $K$ in the time interval $(0,t]$. It is easy to see that
\begin{equation*}
\bar{N}_t+1=1+\sum_{i=1}^{N_t}1_K\left(X_i\right)\ge\sum_{i=1}^{N_t+1}1_K\left(X_{i-1}\right)= t\left\langle Q_t,1_K\right\rangle.
\end{equation*}
By the exponential Chebyshev inequality, we have
\begin{equation}\label{eq1 for lemma 4.6}
\mathbb{P}_{\gamma}\left(\langle Q_t,1_K\rangle>A_{\ell}\right)\le \mathbb{P}_{\gamma}\left(\bar{N}_t+1>A_{\ell}t\right)= \mathbb{P}_{\gamma}\left(\bar{S}_{\lfloor A_{\ell}t\rfloor}\le t\right)\le e^t\mathbb{E}_{\gamma}\left(e^{-\bar{S}_{\lfloor A_{\ell}t\rfloor}}\right),
\end{equation}
where $\lfloor A_{\ell}t\rfloor$ denotes the integer part of $A_{\ell}t$. By Proposition \ref{proposition:subset markov renewal process} and Remark \ref{remark:delayed Markov renewal process}, it is clear that conditioned on $(\bar{X}_k)_{k\ge 0}$, the random variables $(\bar{\tau}_{k})_{k\ge 1}$ are independent. Then we obtain
\begin{equation}\label{eq2 for lemma 4.6}
\begin{split}
&\;\mathbb{E}_{\gamma}\left(e^{-\bar{S}_{\lfloor A_{\ell}t\rfloor}}\right)\\
=&\;
\sum_{x_1,\cdots,x_{\lfloor A_{\ell}t\rfloor}\in K}\mathbb{P}_{\gamma}\left(\bar{X}_1=x_1,\cdots,\bar{X}_{\lfloor A_{\ell}t\rfloor}=x_{\lfloor A_{\ell}t\rfloor}\right)\mathbb{E}_{\gamma}\left[e^{-\bar{S}_{\lfloor A_{\ell}t\rfloor}}\Big|\bar{X}_1=x_1,\cdots,\bar{X}_{\lfloor A_{\ell}t\rfloor}=x_{\lfloor A_{\ell}t\rfloor}\right]\\
=&\;\sum_{x_1,\cdots,x_{\lfloor A_{\ell}t\rfloor}\in K}\mathbb{P}_{\gamma}\left(\bar{X}_1=x_1\right)\mathbb{E}_{\gamma}\left[e^{-\bar{\tau}_1}\Big|\bar{X}_1=x_1\right]\prod_{i=2}^{\lfloor A_{\ell}t\rfloor}\left(\bar{p}_{x_{i-1}x_i}\bar{\psi}_{x_{i-1}x_i}\left(e^{-\tau}\right)\right)\\
\le&\; \left(\sup_{x,y\in K}\bar{\psi}_{xy}(e^{-\tau})\right)^{\lfloor A_{\ell}t\rfloor-1}\mathbb{E}_{\gamma}\left[e^{-\bar{\tau}_1}\right],
\end{split}
\end{equation}
where $\bar{\psi}_{xy}(e^{-\tau})=\int_{(0,\infty)}e^{-s}\bar{\psi}_{xy}(\mathrm{d}s)$. Since $K$ is finite, it is clear that $\sup_{x,y\in K}\bar{\psi}_{xy}(e^{-\tau})<1$. Combining \eqref{eq1 for lemma 4.6} and \eqref{eq2 for lemma 4.6}, we have
\begin{equation*}
\varlimsup_{t\to \infty}\frac{1}{t}\log\mathbb{P}_{\gamma}(\langle Q_t,1_K\rangle>A_{\ell})\le 1+(A_{\ell}-1)\log C.
\end{equation*}
This completes the proof by choosing $A_{\ell}=1-(1+\ell)/\log C$.
\end{proof}

The following proposition states the exponential tightness of the empirical measure and empirical flow.

\begin{proposition}\label{proposition:etem}
Assume Condition \ref{condition:ccomp} to hold. Then there exists a sequence $\{\mathcal K_\ell\}$ of
compact sets in $\mathcal P(V\times (0,\infty])$ and a real sequence $A _\ell\uparrow \infty$
such that for any $\ell\in\mathbb N$
\begin{align}
&  \varlimsup_{t\to \infty} \; \frac 1t
\log \mathbb  P_{\gamma} \big( \mu_t \not\in \mathcal K_\ell \big) \le -\ell\,,
\label{muT exponential}\\
&  \varlimsup_{t\to \infty} \; \frac 1t
\log \mathbb  P_{\gamma} \big( \| Q_t\| >  A_\ell \big) \le -\ell\,.
\label{QT exponential}
\end{align}
In particular, the empirical measure and empirical flow are exponentially tight.
\end{proposition}

\begin{proof}
We first prove \eqref{QT exponential}. We consider the exponential local martingale in Lemma \ref{lem:mhF} by choosing $F\equiv 0$ and
\begin{equation*}
h(x,s)=h^{\eta}(x,s):=\frac{N}{\sigma s}1_K(x)+L\eta(x)1_{K^c}(x),\quad (x,s)\in V\times(0,\infty],
\end{equation*}
where $N$ is as in Lemma \ref{t:letem}. By Lemma \ref{lemma:continuous for theta}, it is easy to check that $(F,h)\in \Gamma$ and $g_{F,h}=(N/\sigma)1_K+\log\eta1_{K^c}$.
It then follows form \eqref{expm2} that
\begin{equation*}
\mathcal{M}^{F,h}_t=\exp\left\{t\left[\left\langle Q_t,-\frac{N}{\sigma}1_K-\log\eta1_{K^c}\right\rangle+\left\langle \hat{\mu}_t,h^{\eta}\right\rangle\right]\right\}.
\end{equation*}
Combining \eqref{h hat u} and the first inequality in Lemma \ref{t:letem}, we have
\begin{equation*}
\mathbb E_{\gamma}    \Big( e^{ t  \langle\hat{\mu}_t,-\sigma h^{\eta}\rangle } \Big)=\mathbb E_{\gamma}    \Big( e^{ t  \langle\hat{\mu}_t,-\frac{N}{s}1_K -\sigma L\eta1_{K^c}\rangle } \Big) \le\mathbb E_{\gamma}    \Big( e^{ t \langle\hat{\mu}_t, h^{\hat{u}}\rangle } \Big)\le
\frac {C_{\gamma}}{c}.
\end{equation*}
By exponential Chebyshev inequality, we obtain
\begin{equation*}
\mathbb P_{\gamma} \big(
\langle\hat{\mu}_t, -\sigma h^{\eta} \rangle > \ell  \big)
\le
\frac{C_{\gamma}}{c} e^{-t\ell}\,.
\end{equation*}
On the other hand, by Lemma \ref{lemma:sbuset exponential}, there exists a real sequence $A_{\ell}'\uparrow\infty$ such that
\begin{equation}\label{QT1K}
\varlimsup_{t\to \infty}\frac{1}{t}\log\mathbb P_{\gamma} \big(
\langle Q_t,1_K \rangle > A_{\ell}'  \big)
\le	-\ell\,.
\end{equation}
Hence it is enough to show that there exists a sequence $A_\ell\uparrow \infty$ such that for any $t>0$ and $\ell\in\mathbb N$,
\begin{equation}
\label{enough1}
\mathbb P_{\gamma} \Big( \|Q_t\|> A_\ell \,,\,
\left\langle\hat{\mu}_t, -\sigma h^{\eta} \right\rangle  \le \ell\,,\,\langle Q_t,1_K \rangle \le A_{\ell}' \Big)  \le   e^{-t\,\ell}\,.
\end{equation}
Since $\log\eta<0$, we have
\begin{equation*}\label{def:Q_T}
\begin{split}
&\; \mathbb P_{\gamma} \left( \|Q_t\|> A_\ell \,,\,\langle \hat{\mu}_t,-\sigma h^{\eta}\rangle\le \ell,\langle Q_t,1_K\rangle\le A_{\ell}' \right)
\\
=&\;
\mathbb E_{\gamma} \left( e^{ t[\log\eta\|Q_t\|+\langle Q_t,(\frac{N}{\sigma}-\log\eta)1_{K}\rangle-\langle \hat{\mu}_t,h^{\eta}\rangle]}
\: \mathcal{M}^{F,h}_t  \: 1_{\{\|Q_t\|> A_\ell\}}
\,1_{\{\langle \hat{\mu}_t,-\sigma h^{\eta}\rangle\le \ell\}}\, 1_{\{\langle Q_t,1_K\rangle \le A_\ell'\}}\right)
\\
\le&\;
\exp\left\{ t \left[ \log\eta A_\ell + \left(\frac{N}{\sigma}-\log\eta\right)A_\ell'+\frac{1}{\sigma}\ell\right] \right\},
\end{split}
\end{equation*}
where we have used Lemma \ref{lem:mhF} in the last step. The proof of \eqref{enough1} is now completed by choosing
\begin{equation*}
A_\ell = \frac{(N/\sigma-\log\eta) A_\ell'+(1+1/\sigma)\ell}{-\log\eta}.
\end{equation*}
Recall that the closed ball in $L^1_+(E)$ is compact with respect to the bounded weak* topology. Then the exponential tightness of the empirical flow follows from \eqref{QT exponential}.

We next prove \eqref{muT exponential}. For a sequence of constants $a_m\uparrow \infty$ to be chosen later, set $W_m =\{ x\in V :\, L\hat{u}(x) \le a_m\}\cup K$. In view of items (e) and (f) in Condition \ref{condition:ccomp}, it is clear that $W_m$ is a finite subset of $V$. Now set
\begin{align*}
\mathcal K_\ell^1 &= \bigcap_{m\ge\ell} \left\{\mu\in\mathcal P(V\times (0,\infty])\,:\:
\mu\left(W_m^c\times (0,\infty]\right) \le \frac 1{m} \right\},\\
\mathcal K_\ell^2 &= \left\{\mu\in\mathcal P(V\times (0,\infty])\,:\:
\left\langle\mu,\frac{1}{s}\right\rangle\le A_\ell\right\},
\end{align*}
where $\langle\mu,1/s\rangle=\sum_{x\in V}\int_{(0,\infty]}\mu(x,\mathrm{d}s)/s$.
Then for any $\mu\in \mathcal K_\ell=\mathcal K_\ell^1\cap \mathcal K_\ell^2$, we have
\begin{equation*}
\mu(V\times (0,\epsilon))=\sum_{x\in V}\int_{(0,\epsilon)}\frac{s}{s}\mu\left(x,{\rm d}s\right)\le \epsilon \left\langle\mu,\frac{1}{s}\right\rangle\le \epsilon A_\ell,
\end{equation*}
\begin{equation*}
\mu\Big(\big(W_m\times [\epsilon,\infty]\big)^c\Big)\le \mu(W_m^c\times (0,\infty])+\mu(V\times (0,\epsilon))\le \frac{1}{m}+\epsilon A_\ell,\quad m\ge \ell.
\end{equation*}
Since $W_m\times [\epsilon,\infty]$ a compact subset of $V\times (0,\infty]$, it follows from Prokhorov's theorem that $\mathcal K_\ell$ is a compact subset of $\mathcal P(V\times (0,\infty])$. Since
\begin{equation*}
\varlimsup_{t\to \infty} \; \frac 1t
\log \mathbb  P_{\gamma} \left( \mu_t \not\in \mathcal K_\ell \right) \le \varlimsup_{t\to \infty} \; \frac 1t
\log \left(P_{\gamma} \left( \mu_t \not\in \mathcal K_\ell^1 \right)+P_{\gamma} \left( \mu_t \not\in \mathcal K_\ell^2 \right)\right),
\end{equation*}
we only need to prove
\begin{equation}\label{formula:Kl}
\varlimsup_{t\to \infty} \; \frac 1t
\log \mathbb  P_{\gamma} \left( \mu_t \not\in \mathcal K_\ell^1 \right)\le -\ell\quad \text{  and  }\quad \varlimsup_{t\to \infty} \; \frac 1t
\log \mathbb  P_{\gamma} \left( \mu_t \not\in \mathcal K_\ell^2 \right)\le -\ell.
\end{equation}
Note that
\begin{equation*}
\left\langle\mu_t,\frac{1}{s}\right\rangle=\frac{1}{t}N_t+\frac{t-S_{N_t}}{t\tau_{N_t+1}}\le \frac{1}{t}(N_t+1)=\| Q_t\|.
\end{equation*}
Then we can obtain the second inequality in \eqref{formula:Kl} from \eqref{QT exponential}, i.e.
\begin{equation*}
\begin{split}
\varlimsup_{t\to \infty} \; \frac 1t
\log \mathbb  P_{\gamma} \left( \mu_t \not\in \mathcal K_\ell^2 \right)=&\;\varlimsup_{t\to \infty} \; \frac 1t
\log \mathbb  P_{\gamma} \left( \left\langle\mu_t,\frac{1}{s}\right\rangle>A_{\ell} \right)\\
\le&\; \varlimsup_{t\to \infty} \; \frac 1t
\log \mathbb  P_{\gamma} \left(\|Q_t\|>A_{\ell} \right)\\
\le&\; -\ell.
\end{split}
\end{equation*}
We next prove the first inequality in \eqref{formula:Kl}. By item (f) in Condition \ref{condition:ccomp}, it is easy to see that $L\hat{u}(x)\ge-\sigma L\eta(x)\ge 0$ for any $x\in K^c$. Recall the definition of $W_m$, $h^{\hat{u}}$, and $N$. It is easy to see that
\begin{equation*}
h^{\hat{u}}(x,s)=\frac{\log\hat{u}(x)}{s}1_K(x)+L\hat{u}(x)1_{K^c}(x)
\ge - \frac{N}{s}1_{K}(x)+a_m 1_{W_m^c}(x).
\end{equation*}
Note that
\begin{equation*}
\left\langle \mu_t,\frac{1}{s}1_K\right\rangle\le\left\langle \hat{\mu}_t,\frac{1}{s}1_K\right\rangle= \langle Q_t,1_K\rangle.
\end{equation*}
By the exponential Chebyshev inequality and Lemma \ref{t:letem}, we obtain
\begin{equation}\label{mut Wcl}
\begin{split}
\mathbb P_{\gamma} \left( \mu_t\left( W_m^c\times (0,\infty] \right) > \frac{1}{m}\right)
\le&\;\mathbb P_{\gamma} \left( \left\langle\mu_t, h^{\hat{u}}+\frac{N}{s}1_K\right\rangle >
\frac{a_m}{m}\right)\\
\le&\;\mathbb P_{\gamma} \left( \left\langle\mu_t, h^{\hat{u}}\right\rangle +N\langle Q_t,1_K\rangle >
\frac{a_m}{m}\right)\\
\le&\;\mathbb P_{\gamma} \left( \left\langle\mu_t, h^{\hat{u}}\right\rangle >
\frac{a_m}{2m}\right)+\mathbb P_{\gamma} \left( N\langle Q_t,1_K\rangle>
\frac{a_m}{2m}\right)\\
\le&\; \frac{C_{\gamma}e^N}{c}\exp\left\{ -t  \frac{a_m}{2m}
\right\} +\mathbb P_{\gamma} \left( \langle Q_t,1_K\rangle>
\frac{a_m}{2Nm}\right).
\end{split}
\end{equation}
From \eqref{QT1K}, the proof is now easily concluded by choosing $a_m = 2m^2 +2NA_{m}'m$.
\end{proof}

\begin{remark}
In fact, the empirical measure $\pi_t$ defined in \eqref{def:empirical measure pi_t} is also exponentially tight and the proof is similar to the first inequality in \eqref{formula:Kl}. Moreover, at this time we only need item (f) in Condition \ref{condition:ccomp} with $\sigma=0$. Note that we will not resort to the compactness conditions anymore in the following proof of upper bound and lower bound. This means that the marginal LDP of the empirical measure $\pi_t$ in Proposition \ref{proposition:rate function for empirical flow} still holds under Condition \ref{condition:ccomp} with $\sigma=0$.
\end{remark}

\subsection{Upper bound}\label{subsection:ub}
We next prove the upper bound of the LDP. Since we have proved the exponential tightness for the empirical measure and empirical flow, to prove the upper bound of the LDP for closed sets, we only need to prove the upper bound for compact sets \cite[Lemma 1.2.18]{dembo1998large}. Before stating the upper bound, we introduce the following notation. For any Polish space $\mathcal{X}$, let $C_c(\mathcal{X})$ be the collection of all continuous functions $f: X\to\mathbb R$ with compact supports. For any functions $M:V\to[0,\infty)$, $\varphi \in C_c(V\times (0,\infty))$, and $c \in C_c(V)$ such that $ 0\le c \le\zeta$ and $c(x)<\zeta(x)$ for any $x\in V$ with $\zeta(x)>0$, let $h^{\varphi,c,M} \colon V\times(0,\infty]\to\mathbb{R}$ be a function defined by
\begin{equation}\label{def:h}
h^{\varphi,c,M}(x,s) = \frac{\varphi_x(s)}s + c(x) \, 1_{(M(x),\infty]}(s), \quad (x,s)\in V\times (0,\infty].
\end{equation}
Recall the definition of $\Gamma$ in \eqref{Gamma}. Let $\Gamma_0$ be the subset of $\Gamma$ defined by
\begin{equation*}
\begin{split}
\Gamma_0=\left\{(F,h)\in \Gamma:\,F\in C_c(E), h=h^{\varphi,c,M} \text{ for some }\varphi,c,M \text{ as above},\text{ and }g_{F,h}\equiv 0\right\}.
\end{split}
\end{equation*}
For any $(F,h)\in\Gamma_0$, let $I_{F,h}: \Lambda\mapsto \mathbb{R}$ be a functional defined by
\begin{equation}\label{def:I_hF}
\begin{split}
I_{F,h}(\mu,Q)
= &\langle Q,F\rangle+\langle \mu,h\rangle.
\end{split}	
\end{equation}
Based on the proof of Lemma \ref{lemma:IhHdelta lower semicontinuous}, it is easy to see that $I_{F,h}$ is a lower semicontinuous function on $\Lambda$. For any $\delta>0$, set
\begin{equation}\label{C delta}
\mathcal{C}_{\delta}=\left\{(\mu,Q)\in \Lambda:\max_{x\in V}|Q^+(x)-Q^-(x)|\le \delta\right\}.
\end{equation}
Let $I_{F,h,\delta}:\Lambda\to \mathbb{R}$ be a functional defined by
\begin{equation}\label{I h F delta}
I_{F,h,\delta}(\mu,Q)=\left\{
\begin{aligned}
&I_{F,h}(\mu,Q), && \text{if } (\mu,Q) \in \mathcal{C}_{\delta},\\
&\infty, && \text{otherwise}.
\end{aligned}\right.
\end{equation}

\begin{lemma}\label{t:pub}
For any $(F,h)\in \Gamma_0$, $\delta>0$, and measurable $\mathcal O\subset \Lambda$, we have
\begin{equation}\label{e:ldopen}
\varlimsup_{t\to\infty} \frac{1}{t} \log \mathbb{P}_{\gamma}((\mu_t,Q_t)\in \mathcal{O}) \le
- \inf_{(\mu,Q) \in \mathcal O} I_{F,h,\delta}(\mu,Q).
\end{equation}
\end{lemma}

\begin{proof}
We first prove that for any measurable $\mathcal B\subset \Lambda$,
\begin{equation}\label{neq:upper}
\varlimsup_{t\to\infty}\; \frac 1t
\log \mathbb  P_{\gamma} \Big( (\mu_t,Q_t) \in \mathcal B \Big)
\le -\inf_{(\mu,Q)\in \mathcal B}  I_{F,h} (\mu,Q).
\end{equation}
Since $c\ge 0$ and $\varphi\in C_c(V\times(0,\infty))$, we have
\begin{equation}\label{varphi infty}
(S_{N_t+1}-t)h(X_{N_t},\tau_{N_t+1})\ge(S_{N_t+1}-t)\frac{\varphi_{X_{N_t}}(\tau_{N_t+1})}{\tau_{N_t+1}}\ge -\|\varphi\|_{\infty},\quad \mathbb{P}_{\gamma}\text{-a.s.},
\end{equation}
where $\|\cdot\|_{\infty}$ denotes the standard $L^{\infty}$-norm. Recall the definition of the semimartingale $\mathcal{M}^{F,h}$ in Lemma \ref{lem:mhF}. For each $t>0$ and measurable set	$\mathcal B\subset \Lambda$, it follows from \eqref{def:I_hF} and \eqref{varphi infty} that
\begin{equation*}
\begin{split}
&\;\mathbb P_{\gamma} \big( (\mu_t,Q_t) \in \mathcal B \big)\\
= &\;\mathbb E_{\gamma} \left(
\exp\left\{ -t \, I_{F,h} (\mu_t,Q_t)  -(S_{N_t+1}-t)h(X_{N_t},\tau_{N_t+1})\right\}
\: \mathcal{M}_t^{F,h} \: 1_{\mathcal B}\left(\mu_t,Q_t\right) \right)\\
\le&\;
\exp\left\{\|\varphi\|_{\infty}\right\}\sup_{(\mu,Q)\in\mathcal B} \exp\left\{- t \, I_{F,h} (\mu,Q) \right\}
\; \mathbb E_{\gamma} \left(
\: \mathcal{M}_t^{F,h} \: 1_{\mathcal B}(\mu_t,Q_t) \right).
\end{split}
\end{equation*}
Hence we have proved \eqref{neq:upper}. It is easy to see that  $\mathbb{P}_{\gamma}((\mu_t,Q_t)\in \mathcal{C}_{\delta})=1$ for any $t\ge 1/\delta$. Finally, taking $\mathcal B=\mathcal O \cap \mathcal{C}_{\delta}$ in \eqref{neq:upper}, we obtain
\begin{equation*}
\begin{split}
\varlimsup_{t\to\infty}\frac{1}{t}\log\mathbb{P}_{\gamma}((\mu_t,Q_t)\in\mathcal{O})
&= \varlimsup_{t\to\infty}\frac{1}{t}\log\mathbb{P}_{\gamma}((\mu_t,Q_t)\in\mathcal{O}\cap \mathcal{C}_{\delta})
\\
&\le -\inf_{(\mu,Q)\in\mathcal{O}\cap\mathcal{C}_{\delta}}I_{F,h}(\mu,Q)
\\
&=- \inf_{(\mu,Q)\in\mathcal O}I_{F,h,\delta}(\mu,Q).
\end{split}
\end{equation*}
This completes the proof of this lemma.
\end{proof}

\begin{lemma}\label{lemma:IhHdelta lower semicontinuous}
Suppose that Assumption \ref{ass:locally finite} holds. For any $(F,h)\in \Gamma_0$ and $\delta>0$, $I_{F,h,\delta}$ is a convex and lower semicontinuous function on $\Lambda$, where $L^1_+(E)$ is endowed with the bounded weak* topology.
\end{lemma}
\begin{proof}
Since a linear functional on $L^1_+(E)$ is continuous with respect to the bounded weak* topology if and only if it is continuous with respect to the weak* topology \cite[Theorem 2.7.8]{megginson2012introduction}, the weak* topology on $L^1(E)$ is the smallest topology such that the maps $Q\in L^1(E)\to \langle Q,f\rangle\in \mathbb R$ with $f\in C_0(E)$ being continuous. Let
\begin{equation*}
f^1_x(y,z)=1_x(y),\qquad f^2_x(y,z)=1_x(z), \quad (y,z) \in E.
\end{equation*}	
By Assumption \ref{ass:locally finite}, the graph $(V,E)$ is locally finite. It is easy to check that $f^1_x,f^2_x\in C_0(E)$ for any $x\in V$. Note that
\begin{align*}
\mathcal{C}_{\delta}=&\;\bigcap_{x\in V}\left(\Big\{(\mu,Q)\in \Lambda:Q^+(x)-Q^-(x)\le \delta\Big\}\bigcap\Big\{(\mu,Q)\in \Lambda:Q^-(x)-Q^+(x)\le \delta\Big\}\right)\\
=&\;\bigcap_{x\in V}\left(\Big\{(\mu,Q)\in \Lambda:\langle Q,f^1_x-f^2_x\rangle\le \delta\Big\}\bigcap\Big\{(\mu,Q)\in \Lambda:\langle Q,f^2_x-f^1_x\rangle\le \delta\Big\}\right).
\end{align*}
This implies that $\mathcal{C}_{\delta}$ is a closed subset of $\Lambda$. Moreover, it is easy to see that $\mathcal{C}_{\delta}$ is convex. Thus we only need to prove that $I_{F,h}$ is a convex and lower semicontinuous function on $\Lambda$.

Since $I_{F,h}$ is a linear functional, it must be convex. Since $F\in C_c(E)\subset C_0(E)$, the map $Q\mapsto\langle Q,F\rangle$ is continuous. On the other hand, since $\varphi\in C_c(V\times (0,\infty))$, there exists a finite set $K\subset V$ such that $\varphi_x\equiv 0$ for $x\in K^c$ and $\varphi_x\in C_c(0,\infty)$ for $x\in K$. Then $h^{\varphi,c,M}$ is a bounded lower semicontinuous function. For any $\mu\in\mathcal{P}(V\times(0,\infty])$, let $\mu_n\in\mathcal{P}(V\times(0,\infty])$ be a sequence of probability measures such that $\mu_n$ weakly converges to $\mu$. It then follows from the Portmanteau theorem \cite[Corollary 8.2.5]{bogachev2007measure} that
\begin{equation*}
\varliminf_{n\to\infty}\left\langle \mu_n,h^{\varphi,c,M}\right\rangle\ge\left\langle \mu,h^{\varphi,c,M}\right\rangle.
\end{equation*}
Since $\mathcal{P}(V\times (0,\infty])$ is a metric space with the topology of weak converge, it is clear that $\langle\mu,h\rangle$ is lower semicontinuous with respect to $\mu$. This completes the proof of this lemma.
\end{proof}

\begin{lemma}
\label{lem:I:=sup I_hH}
For any $(\mu,Q)\in \Lambda$ satisfying $Q^+=Q^-$, we have
\begin{equation} \label{e:muf}
I(\mu,Q) = \sup_{(F,h)\in\Gamma_0}I_{F,h}(\mu,Q).
\end{equation}
\end{lemma}

\begin{proof}
The proof for $I(\mu,Q) \ge \sup_{(F,h)\in\Gamma_0}I_{F,h}(\mu,Q)$ is similar to \cite[Proposition 2.1]{mariani2016large}. Here only prove the converse inequality
\begin{equation}\label{le}
I(\mu,Q) \le \sup_{(F,h)\in\Gamma_0} I_{F,h}(\mu,Q).
\end{equation}

Case 1: $(\mu,Q) \notin\mathcal{D}$.
Then there exists $x\in V$ such that
\begin{equation*}
\mu\left(x,\frac{1}{\tau}\right):=\int_{(0,\infty]}\frac{1}{s}\,\mu(x,{\rm d}s)\neq Q_x.
\end{equation*}
Without loss of generality, we assume that $\mu(x,1/\tau)< Q_x$. For any $C>0$, $(y,z)\in E$, and $(y,s)\in V\times (0,\infty]$, set
\begin{equation*}
F_C(y,z) = C1_x(y), \qquad h_C(y,s) = -\frac{C}{s}1_x(y).
\end{equation*}
It is easy to check that $g_{F_C,h_C}\equiv 0$ and
\begin{equation*}
I_{F_C,h_C}(\mu,Q)=C\left[Q_x-\mu\left(x,\frac{1}{\tau}\right)\right].
\end{equation*}
Here $(F_C,h_c)\notin\Gamma_0$, but we still define $I_{F_C,h_C}(\mu,Q)$ as in \eqref{def:I_hF}. As $C\to \infty$, we have $I_{F_C,h_C}(\mu,Q)\to\infty$. Now for any $C>0$, we construct a sequence $(F_n,h_n)\in\Gamma_0$ such that
\begin{equation}\label{IhnFn}
\lim_{n\to \infty}I_{F_n,h_n}(\mu,Q)=I_{F_C,h_C}(\mu,Q).
\end{equation}
For any $y\in V$, it is easy to see that there exists a non-negative function $g_y\in C_c(0,\infty)$ such that $\psi_y(\{t\in(0,\infty):g_y(t)>0\})>0$. Let
\begin{equation}\label{Fn hn}
F_n(y,z)=C1_x(y)1_{K_n}(z),\quad h_n(y,s)=\frac{C}{s}(f_n(s)+c_ng_x(s))1_x(y),
\end{equation}
where $f_n:(0,\infty]\to \mathbb{R}$ is a sequence of continuous functions satisfying $f_n(s)=0$ if $s\in (0,1/(n+1))\cup(n+1,\infty]$ and $f_n(s)=-1$ if $s\in (1/n,n)$, $c_n$ is a sequence of constants to be chosen later such that $g_{F_n,h_n}\equiv 0$, and $K_n$ is a sequence of finite sets such that $K_n\uparrow V$ and
\begin{equation*}
\left[1+(e^C-1)\sum_{y\in K_n}p_{xy}\right]\langle\psi_x,e^{Cf_n}\rangle\le 1.
\end{equation*}
Since $\lim_{n\to\infty}\langle\psi_x,e^{Cf_n}\rangle=-C$, such $K_n$ must exist. It is easy to check that $g_{F_n,h_n}(y)=0$ for any $y\neq x$ and
\begin{equation}\label{gFnhn}
g_{F_n,h_n}(x)=\log\left(1+(e^C-1)\sum_{y\in K_n}p_{xy}\right)+\log\left\langle\psi_x, e^{C(f_n+c_ng_x)}\right\rangle.
\end{equation}
Note that
\begin{equation*} \lim_{c_n\to\infty}g_{F_n,h_n}(x)=\infty,\quad\lim_{c_n\to0}g_{F_n,h_n}(x)=\log\left(1+(e^C-1)\sum_{y\in K_n}p_{xy}\right)+\log\left\langle\psi_x, e^{Cf_n}\right\rangle\le 0.
\end{equation*}
Since $g_{F_n,h_n}(x)$ is continuous with respect to $c_n$, by the intermediate value theorem, there exists a sequence of constants $c_n$ such that $g_{F_n,h_n}(x)\equiv 0$, which implies that $(F_n,h_n)\in \Gamma_0$. Taking $n\to\infty$ on both sides of \eqref{gFnhn}, it is easy to see that $c_n\rightarrow0$. Thus in the sense of pointwise convergence, we have $(F_n,h_n)\to(F_C,h_C)$. By the dominated convergence theorem, we obtain \eqref{IhnFn}. This further implies \eqref{le}.

Case 2: $(\mu,Q) \in\mathcal{D}$. Then for any $x\in V$, we have
\begin{equation*}
\quad\mu\left(x,\frac{1}{\tau}\right)= Q_x<\infty.
\end{equation*}
Recall the definition of $\tilde{Q}_{xy}$ and $\tilde{\mu}_x$ in \eqref{def:pQ,psi mu}. Since $V$ and $(0,\infty)$ are both Polish spaces, it follows from \eqref{relative entropy} that for any $x\in V$,
\begin{equation*}
H\big( \tilde{Q}_{x,\cdot} \, |\,  p_{x,\cdot}  \big)=\sup_{\{F_x\in C_b(V):\sum_{y\in V} p_{xy}e^{F_x(y)}=1\}}\sum_{y\in V}\tilde{Q}_{xy}F_x(y),
\end{equation*}
\begin{equation*}
H\big( \tilde{\mu}_x \, | \, \psi_x  \big)=\sup_{\{\varphi_x\in C_b(0,\infty):\langle\psi_x,e^{\varphi_x}\rangle=1\}}\int_{(0,\infty)}\varphi_x(s)\tilde{\mu}_x({\rm d}s).
\end{equation*}
Recall the definition of $I$ in \eqref{def:I}. Then for any $F\in C_b(E)$, $\varphi\in C_b(V\times (0,\infty))$ satisfying $\sum_{y\in V}p_{xy}e^{F(x,y)}=1$ and $\langle\psi_x,e^{\varphi_x}\rangle=1$ for any $x\in V$, and $c:V\to [0,\infty)$ satisfying $0\le c\le \zeta$ and $c(x)<\zeta(x)$ for any $x\in V$ with $\zeta(x)>0$, we only need to find a sequence $(F_n,h_n)\in \Gamma_0$ such that
\begin{equation}\label{lim I hn Fn}
\varliminf_{n\to \infty}I_{F_n,h_n}(\mu,Q)\ge \sum_{x\in V}\left[Q_x\sum_{y\in V}\tilde{Q}_{xy}F(x,y)+Q_x\int_{(0,\infty)}\varphi_x(s)\tilde{\mu}_x({\rm d}s)+c_x\mu(x,\{\infty\})\right].
\end{equation}
For any $(x,y)\in E$ and $(x,s)\in V\times (0,\infty]$, let
\begin{equation*}
F_n(x,y)=(F(x,y)-a_x^n)1_{K_n}(x)1_{K_n}(y),\quad h_n(x,s)=\left[\frac{1}{s}(\varphi^n_x(s)-b^n_x g_x(s))+c_x1_{(n+1,\infty]}(s)\right]1_{K_n}(x),
\end{equation*}
where $\varphi_x^n:(0,\infty]\to\mathbb{R}$ is a sequence of continuous functions satisfying $\varphi_x^n(s)=0$ if $s\in (0,1/(n+1))\cup(n+1,\infty)$ and $\varphi_x^n(s)=\varphi_x(s)$ if $s\in (1/n,n)$, $a_x^n$ and $b^n_x$ are two sequences of constants satisfying $\sum_{y\in V}p_{xy}e^{F_n(x,y)}=1$ as well as $\langle\psi_x,e^{s h_n(x,s)}\rangle=1$, $g_x$ is the function used in \eqref{Fn hn}, and $K_n$ is a sequence of finite sets such that $K_n\uparrow V$. For any $x\in V$ and sufficiently large $n$ that may depend on $x$, straightforward computations show that
\begin{equation*}
a_x^n=\log\frac{\sum_{y\in K_n}p_{xy}e^{F(x,y)}}{\sum_{y\in K_n}p_{xy}},\qquad \int_{(0,n+1]}e^{\varphi_x^n(s)-b_x^ng_x(s)}\psi_x(\mathrm{d}s)+\int_{(n+1,\infty)}e^{c_xs}\psi_x(\mathrm{d}s)=1.
\end{equation*}
Moreover, it is easy to see that such sequence of constants $b^n_x$ exists, and $a_x^n\to 0$ and $b_x^n\to 0$ as $n\to\infty$. For all $x\in V$, we have
\begin{align*}
&\;\varliminf_{n\to\infty}\left[\sum_{y\in V}Q(x,y)F_n(x,y)+\int_{(0,\infty]}h_n(x,s)\mu(x,{\rm d}s)\right]\\
\ge&\;\sum_{y\in V}Q(x,y)F(x,y)+\int_{(0,\infty)}\frac{1}{s}\varphi_x(s)\mu(x,{\rm d}s)+c_x\mu(x,\{\infty\}).
\end{align*}
Then \eqref{lim I hn Fn} follows from Fatou's lemma. This completes the proof of this lemma.
\end{proof}

We are now in a position to prove the upper bound of the LDP.

\begin{proposition}\label{pro:weak LDP ub}
Suppose that Assumptions \ref{ass:irreducibility}-\ref{ass:locally finite} are satisfied. Let $L^1_+(E)$ be endowed with the bounded weak* topology. Then $I$ is a convex and lower semicontinuous function, and for each compact set $\mathcal{K}\subset \Lambda$, we have
\begin{equation}\label{weak LDP ub}
\varlimsup_{t\to\infty}\;
\frac{1}{t}\log \mathbb  P_{\gamma} \Big( (\mu_t,Q_t) \in \mathcal K \Big)
\le -\inf_{(\mu,Q)\in \mathcal K} I(\mu,Q).
\end{equation}
Moreover, if Condition \ref{condition:ccomp} is satisfied, then the above equation also holds for any closed set $\mathcal{K}\subset\Lambda$.
\end{proposition}

\begin{proof}
We first prove that $I$ is convex and lower semicontinuous. Since we have proved $I_{F,h,\delta}$ is convex and lower semicontinuous in Lemma \ref{lemma:IhHdelta lower semicontinuous}, we only need to verify that
\begin{equation*} \label{e:muf}
I(\mu,Q) = \sup_{(F,h)\in\Gamma_0,\delta>0} I_{F,h,\delta}(\mu,Q),\quad(\mu,Q)\in\Lambda.
\end{equation*}

Case 1: $Q^+\neq Q^-$. Let $\delta$ be a constant satisfying $0<\delta<\max_{x\in V}|Q^+(x)-Q^-(x)|$. Since $(\mu,Q)\in \Lambda\setminus\mathcal{C}_{\delta}$, we have
\begin{equation*}
I(\mu,Q)=\infty= \sup_{(F,h)\in\Gamma_0,\delta>0}I_{F,h,\delta}(\mu,Q).
\end{equation*}

Case 2: $Q^+=Q^-$. By Lemma \ref{lem:I:=sup I_hH}, we have
\begin{equation*} \label{e:muf}
I(\mu,Q) = \sup_{(F,h)\in\Gamma_0}I_{F,h}(\mu,Q).
\end{equation*}
Recall the definition of $C_{\delta}$ in \eqref{C delta} and the definition of $I_{F,h,\delta}$ in \eqref{I h F delta}. It is easy to check that for any $\delta>0$, we have $(\mu,Q)\in C_{\delta}$ and $I_{F,h}(\mu,Q) = I_{F,h,\delta}(\mu,Q)$. Optimizing over $(F,h)\in\Gamma_0$ and $\delta>0$, we obtain
\begin{equation*}
I(\mu,Q) = \sup_{(F,h)\in\Gamma_0}I_{F,h}(\mu,Q)=\sup_{(F,h)\in\Gamma_0,\delta>0} I_{F,h,\delta}(\mu,Q).
\end{equation*}

We next prove \eqref{weak LDP ub}. Minimizing \eqref{e:ldopen} over $(F,h)\in\Gamma_0$ and $\delta>0$, we obtain
\begin{equation*}\label{min max}
\varlimsup_{t\to\infty}\frac{1}{t}\log\mathbb{P}_{\gamma}((\mu_t,Q_t)\in\mathcal{O}) \le
-\sup_{(F,h)\in\Gamma_0,\delta>0}\inf_{(\mu,Q)\in\mathcal O}I_{F,h,\delta}(\mu,Q).
\end{equation*}
Note that $I_{F,h,\delta}$ is lower semicontinuous for any $(F,h)\in \Gamma_0$ and $\delta>0$. For any compact set $\mathcal K\subset\Lambda$, it follows from the min-max lemma \cite[Lemma 3.3 in Appendix 2.3]{kipnis1998scaling} that
\begin{equation*}
\varlimsup_{t\to\infty}\frac{1}{t}\log\mathbb{P}_{\gamma}((\mu_t,Q_t)\in\mathcal{K}) \le
-\inf_{(\mu,Q)\in\mathcal K}\sup_{(F,h)\in\Gamma_0,\delta>0}I_{F,h,\delta}(\mu,Q)=-\inf_{(\mu,Q)\in\mathcal K}I(\mu,Q).
\end{equation*}
Finally, under Condition \ref{condition:ccomp}, it follows from Proposition \ref{proposition:etem} that $(\mu_t,Q_t)$ is exponentially tight. Hence the upper bound in \eqref{weak LDP ub} also holds for any closed set $\mathcal{K}\subset\Lambda$.
\end{proof}

\subsection{Lower bound}\label{subsection:lb}
We next prove the lower bound of the LDP. Before giving the proof, we recall the following lemma, which can be found in \cite[Lemma 5.2]{bertini2015large}.
\begin{lemma}\label{lemma:entropy lower bound}
Let $\{P_t\}$ be a family of probability measures on a completely regular topological space $\mathcal X$ and let $J\colon \mathcal X\to[0,\infty]$ be a function. Assume that for each $x\in\mathcal X$, there exists a family of probability measures $\{\tilde{P}_t^x\}$ weakly convergent to $\delta_x$ such that
\begin{equation}
\label{entb}
\varlimsup_{t\to\infty} \frac 1t H\big(\tilde{P}_t^x \big| P_t\big)
\le J(x).
\end{equation}
Then the family $\{P_t\}$ satisfies the large deviation lower bound with rate function $\sce J$. Here $\sce J$ is the lower semicontinuous envelope of $J$, i.e.
\begin{equation*}
(\sce J) \, (x) = \sup_{U \in\mathcal N_x} \; \inf_{y\in U} \; J(y),
\end{equation*}
where $\mathcal N_x$ denotes the collection of the open neighborhoods of $x$.
\end{lemma}

In order to prove the lower bound, we will apply Lemma \ref{lemma:entropy lower bound} for completely regular topological space $\Lambda$. Recall the definitions of $\mathcal{D}$ and $\tilde{Q}$ in \eqref{def:D0} and \eqref{def:pQ,psi mu}, respectively. Let $\mathcal{D}_2\subset \mathcal D_1\subset \mathcal{D}$ be defined by
\begin{equation}\label{def:D1}
\begin{split}
\mathcal{D}_1 = \big\{&(\mu,Q)\in\mathcal{D}: I(\mu,Q)<\infty, \mu(x,(0,\infty))>0, \ \forall \, x\in V,
\\ &(\tilde{Q}_{xy})_{x,y\in V} \ {\rm defines \ an \ irreducible \ transition \ matrix \ on  \ } V \big\},
\end{split}
\end{equation}
\begin{equation}\label{def:D2}
\mathcal{D}_2 = \left\{(\mu,Q)\in\mathcal{D}_1: \mu(x,\{\infty\})=0, \ \forall \, x\in V\right\}. \qquad\qquad\qquad\,
\end{equation}
Moreover, let $J:\Lambda\to [0,\infty]$ be the restriction of the rate function $I$ to $\mathcal{D}_2$, i.e.
\begin{equation}\label{def:J}
J(\mu,Q)=\left\{
\begin{aligned}
&I(\mu,Q),  && \textrm{if } (\mu,Q)\in\mathcal D_2,\\
&\infty,   && \textrm{otherwise}.
\end{aligned}\right.
\end{equation}

In the following lemma, we will construct a family of probability measures $\{\tilde{P}^x_t\}$ on $\Lambda$ and prove \eqref{entb} with the upper bound given by \eqref{def:J}. The proof is in the spirit of that given in \cite[Proposition 5.1]{mariani2016large}, but some details are supplemented.

\begin{lemma}\label{t:plb}
Let $P_t = \mathbb P_{\gamma} \circ (\mu_t,Q_t)^{-1}$.  For each $(\mu,Q)\in  \Lambda$ there exists a family of probability measures $\{\tilde{P}^{(\mu,Q)}_t\}$ on $\Lambda$ weakly convergent to $\delta_{(\mu,Q)}$ such that
\begin{equation*}
\varlimsup_{t\to\infty} \frac 1t
H\left(\tilde{P}^{(\mu,Q)}_t \big| P_t \right)
\le J(\mu,Q).
\end{equation*}
\end{lemma}

\begin{proof}
By the definition of $J$ in \eqref{def:J}, we may restrict the proof to $(\mu,Q)\in \mathcal D_2$. When $(\mu,Q)\in \mathcal D_2$, we have
\begin{equation*}
Q_x=\mu\left(x,\frac{1}{\tau}\right)>0,\qquad \mu(x,\{\infty\})=0,\quad x\in V,
\end{equation*}
and it is easy to check that
\begin{equation*}
\sum_{x\in V} Q_x H\big( \tilde{Q}_{x,\cdot} \, |\,  p_{x,\cdot}  \big)
<\infty,\qquad \sum_{x\in V} Q_x H\big( \tilde{\mu}_x \, | \, \psi_x  \big)<\infty.
\end{equation*}
It then follows from \cite[Theorem 2.1 and the remark after Theorem 2.1]{donsker1983asymptotic} that $\tilde{Q}_{x,\cdot}\ll p_{x,\cdot}$, $\tilde{\mu}_x\ll \psi_x$ and
\begin{equation}\label{integral condition1}
\sum_{(x,y)\in E} Q(x,y)\left|\log\left(\frac{Q(x,y)}{Q_x p_{xy}}\right)\right|<\infty,
\end{equation}
\begin{equation}\label{integral condition2}
\sum_{x\in V}\int_{(0,\infty)}\frac{1}{s}\left|\log\left(\frac{\mu(x,{\rm d}s)}{Q_x s \psi_x({\rm d}s)}\right)\right|\mu(x,{\rm d}s)<\infty.
\end{equation}
Let $(F,h)\in \Gamma$ be defined by
\begin{equation}\label{F h}
F(x,y)=\log\left(\frac{Q(x,y)}{Q_x p_{xy}}\right),\quad h(x,s)=\frac{1}{s}\log\left(\frac{\mu(x,{\rm d}s)}{Q_x s \psi_x({\rm d}s)}\right).
\end{equation}
In this way, the transition probability kernel $(P^F,\Psi^h)$ defined in \eqref{pF psih} are given by
\begin{equation*}
p^F_{xy}=\frac{1}{Q_x}Q(x,y)=\tilde{Q}_{xy},\qquad\psi_x^h({\rm d}s)=\frac{1}{Q_x s}\mu(x,{\rm d}s)=\tilde{\mu}_x({\rm d}s).
\end{equation*}
Let $\mathbb{P}_{\gamma}^{F,h}$ be the probability measure under which $(X,\tau)$ is a Markov renewal process with transition kernel $(P^F,\Psi^h)$ and initial distribution $\gamma$. Since $(\mu,Q)\in \mathcal{D}_2$, it is easy to see that $p^F=\tilde{Q}$ is irreducible. It then follows from \eqref{def:D0} that
\begin{equation*}
\sum_{x\in V}\nu_x^F p^F_{xy}=\frac{1}{\|Q\|}\sum_{x\in V}Q(x,y)=\frac{Q_y}{\|Q\|}=\nu_y^F,
\end{equation*}
where $\nu_x^F=Q_x/\|Q\|$ for any $x\in V$. This means that $\nu^F$ is the unique invariant distribution for $P^F$. Since $\mu(x,\{\infty\})=0$ for any $x\in V$, we have
\begin{equation}\label{E nu F mu Q}
\mathbb E^{F,h}_{\nu^F}(\tau_1)=\sum_{x\in V}\nu_x^F\int_{(0,\infty)}s\psi_x^h\left(\mathrm{d}s\right)=\sum_{x\in V} \frac{Q_x}{\|Q\|}\int_{(0,\infty)}s\frac{1}{Q_xs}\mu(x,{\rm d}s)=\frac{1}{\|Q\|}<\infty.
\end{equation}
By the strong law of large numbers for semi-Markov processes, for any $f\in C_b(V\times (0,\infty])$ and $g\in C_b(E)$, we have
\begin{equation}\label{f mut}
\lim_{t\to \infty} \langle\mu_t,f\rangle=\frac{1}{\mathbb E^{F,h}_{\nu^F}(\tau_1)}\sum_{x\in V}\nu^F_x\int_{(0,\infty)}s f(x,s)\psi^h_x(\mathrm{d}s)=\langle\mu,f\rangle,\quad\mathbb{P}^{F,h}_{\gamma}\text{-a.s.}
\end{equation}
\begin{equation}\label{g Qt}
\lim_{t\to \infty}\langle Q_t,g\rangle=\frac{1}{\mathbb E^{F,h}_{\nu^F}(\tau_1)}\sum_{(x,y)\in E}\nu^F_x p^F_{xy}g(x,y)=\langle Q,g\rangle,\quad\mathbb{P}^{F,h}_{\gamma}\text{-a.s.}
\end{equation}

We now construct the family of probability measures $\{\tilde{P}^{(\mu,Q)}_t\}$. For any $\epsilon>0$ and $t\ge 0$, let $T_t = \lfloor\|Q\|(1+\epsilon)t\rfloor$ and let $\mathbb{P}_{\gamma,t,\epsilon}$ be the probability measure under which the law of the process $(X,\tau) = \{(X_k)_{k\ge0},(\tau_k)_{k\ge1}\}$ satisfies the following requirements:
\begin{itemize}
\item[(a)] $\{(X_k)_{0\le k\le T_t},(\tau_{k})_{1\le k\le T_t}\}$ is a Markov renewal process with transition kernel $(P^F,\Psi^h)$ and initial distribution $\gamma$.
\item[(b)] Conditioned on $X_{T_t}$, the processes $\{(X_k)_{0\le k\le T_t},(\tau_{k})_{1\le k\le T_t}\}$ and $\{(X_k)_{k\ge T_t},(\tau_{k})_{k\ge T_t+1}\}$ are independent of each other. Moreover, $\{(X_k)_{k\ge T_t},(\tau_{k})_{k\ge T_t+1}\}$ is a Markov renewal process with transition kernel $(P,\Psi)$ and initial distribution $\delta_{X_{T_t}}$.
\end{itemize}
Intuitively, under $\mathbb{P}_{\gamma,t,\epsilon}$, the process $(X,\tau)$ has the transition kernel $(P^F,\Psi^h)$ before time $T_t$ and has the transition kernel $(P,\Psi)$ after time $T_t$. Set $\tilde{P}_{t,\epsilon}=\mathbb{P}_{\gamma,t,\epsilon}\circ(\mu_t,Q_t)^{-1}$ and let $\epsilon(t)\downarrow 0$ to be chosen later such that $\tilde{P}^{(\mu,Q)}_t:=\tilde{P}_{t,\epsilon(t)}$ weakly converges to $\delta_{(\mu,Q)}$ as $t\to \infty$. In other words, for any $G\in C_{b}(\Lambda)$, we have
\begin{equation}\label{weak convergence}
\lim_{t\to\infty}\int_{\Lambda}G{\rm d}\tilde{P}_{t,\epsilon(t)}=G(\mu,Q).
\end{equation}
For any $\epsilon>0$, it follows from \eqref{E nu F mu Q} and the strong law of large numbers for Markov renewal processes \cite[Theorem 3.13]{limnios2012semi} that
\begin{equation*}
\lim_{t\to\infty}\frac{S_{T_t}}{t}=\lim_{t\to\infty}\frac{S_{T_t}}{T_t}\frac{T_t}{t}=\mathbb E^{\mu,Q}_{\nu^F}(\tau_1) \|Q\| (1+\epsilon)=1+\epsilon,\quad \mathbb{P}_{\gamma}^{F,h}\text{-a.s.}
\end{equation*}
For any $t>0$, let
\begin{equation*}
D_{t,\epsilon}=\left\{S_{T_t}>t\right\}=\left\{N_t+1\le T_t\right\}\in \mathcal{F}_{T_t},
\end{equation*}
where $\mathcal{F}_{T_t}=\sigma((X_k,\tau_k)_{0\le k\le T_t})$. It is easy to see that $\mathbb{P}_{\gamma}^{F,h}|_{\mathcal{F}_{T_t}}=\mathbb{P}_{\gamma,t,\epsilon}|_{\mathcal{F}_{T_t}}$. Then we have
\begin{equation}\label{P t delta x}
\lim_{t\to \infty}\mathbb{P}_{\gamma,t,\epsilon}\left(D_{t,\epsilon}\right)=\lim_{t\to \infty}\mathbb{P}_{\gamma,t,\epsilon}\left(S_{T_t}>t\right)=\lim_{t\to\infty}\mathbb{P}_{\gamma}^{F,h}\left(\frac{S_{T_t}}{t}>1\right)=1.
\end{equation}
Before determining $\epsilon(t)$, we give an estimation of the relative entropy. We observe that for any $\epsilon>0$,
\begin{equation}\label{troppilabel}
\begin{split}
H\left( \tilde{P}_{t,\epsilon} \big| P_t \right)
\le
H\left( \mathbb P_{\gamma,t,\epsilon} \big|
\mathbb P_{\gamma}  \right)
&= \mathbb E^{F,h}_{\gamma} \left[\log\frac{{\rm d}\mathbb{P}^{F,h}_{\gamma}|_{\mathcal{F}_{T_t}}}{{\rm d}\mathbb{P}_{\gamma}|_{\mathcal{F}_{T_t}}}\right]\\
&=\mathbb E^{F,h}_{\gamma}\left[\sum_{i=1}^{T_t}\left(F(X_{i-1},X_i)+\tau_ih(X_{i-1},\tau_i)\right)\right].
\end{split}
\end{equation}
Indeed, the first inequality follows from the variational characterization of the relative entropy \cite[Section 2]{donsker1983asymptotic} and the last equality is a straightforward computation of the Radon-Nikodym density
(similarly to \eqref{formula:change measure}).

Combining \eqref{integral condition1}, \eqref{integral condition2}, and \eqref{F h}, we have $\langle Q,|F|\rangle<\infty$ and $\langle \mu,|h|\rangle<\infty$.
Note that $(X_k,X_{k+1},\tau_{k+1})_{k\ge 0}$ is a Markov process. By the ergodic theorem of Markov processes, we have
\begin{equation}\label{E mu Q x}
\begin{split}
\lim_{t\to \infty}\frac{1}{t}\mathbb E^{F,h}_{\gamma} \left[\sum_{i=1}^{T_t}\Big(F(X_{i-1},X_i)+\tau_ih(X_{i-1},\tau_i)\Big)\right]
=&(1+\epsilon)\|Q\|\frac{1}{\|Q\|}\left[\langle Q,F\rangle+\langle \mu,h\rangle\right]\\
\le& (1+\epsilon)I(\mu,Q).
\end{split}
\end{equation}

We next construct $\epsilon(t)\downarrow 0$. Let $\epsilon(t)=1/n$ for any $t_{n-1}<t\le t_n$ be a step function, where $t_n$ is an increasing sequence such that
\begin{equation*}\label{tn}
\mathbb{P}_{\gamma,t,1/n}\left(D_{t,1/n}\right)\ge 1-\frac{1}{n}\quad\text{and}\quad  \frac{1}{t}H\left( \tilde{P}_{t,1/n} \Big| P_t \right)\le \left(1+\frac{1}{n-1}\right)I(\mu,Q),\quad t>t_{n-1}.
\end{equation*}
It follows from \eqref{P t delta x}, \eqref{troppilabel}, and \eqref{E mu Q x} that such $t_n\uparrow \infty$ exist. Then we have
\begin{equation}\label{limit}
\lim_{t\to \infty}\mathbb{P}_{\gamma,t,\epsilon(t)}\left(D_{t,\epsilon(t)}\right)=1,\qquad \varlimsup_{t\to \infty}\frac{1}{t}H\left(\tilde{P}_{t,\epsilon(t)}\Big|P_t\right)\le I(\mu,Q).
\end{equation}
Finally, we prove \eqref{weak convergence}. Note that
\begin{align*}
\int_{\Lambda}G{\rm d}\tilde{P}_{t,\epsilon(t)}=\int_{\Omega}G(\mu_t,Q_t){\rm d}\mathbb{P}_{\gamma,t,\epsilon(t)}=\int_{\Omega}G(\mu_t,Q_t)1_{D_{t,\epsilon(t)}}{\rm d}\mathbb{P}_{\gamma,t,\epsilon(t)}+\int_{\Omega}G(\mu_t,Q_t)1_{D_{t,\epsilon(t)}^c}{\rm d}\mathbb{P}_{\gamma,t,\epsilon(t)}.
\end{align*}
It thus follows from the first equality of \eqref{limit} that
\begin{equation*}
\left|\int_{\Omega}G(\mu_t,Q_t)1_{D_{t,\delta(t)}^c}{\rm d}\mathbb{P}_{\gamma,t,\delta(t)}\right|\le\|G\|_{\infty}\mathbb{P}_{\gamma,t,\epsilon(t)}\left(D_{t,\epsilon(t)}^c\right)\to 0,\qquad \text{ as }\ t\to \infty.
\end{equation*}
On the other hand, since $D_{t,\epsilon(t)}=\{N_t+1\le T_t\}$, we have $G(\mu_t,Q_t)1_{D_{t,\epsilon(t)}}\in \mathcal{F}_{T_t}$. Note that $\mathcal{P}(V\times (0,\infty])$ is endowed with the weak convergence topology and $L^1_+(E)$ is endowed with the bounded weak* topology. It is a direct consequence of \eqref{f mut} and \eqref{g Qt} that
\begin{equation*}
\lim_{t\to\infty}G(\mu_t,Q_t)=G(\mu,Q),\quad \mathbb{P}_{\gamma}^{F,h}\text{-a.s.}
\end{equation*}
Moreover, it follows from the first equality of \eqref{limit} that
\begin{equation*}
\lim_{t\to\infty}\mathbb{P}_{\gamma}^{F,h}\left(D_{t,\epsilon(t)}\right)=\lim_{t\to\infty}\mathbb{P}_{\gamma,t,\epsilon(t)}\left(D_{t,\epsilon(t)}\right)=1.
\end{equation*}
Then we obtain
\begin{equation}\label{FmuQ}
\begin{split}
\int_{\Omega}G(\mu_t,Q_t)1_{D_{t,\epsilon(t)}}&{\rm d}\mathbb{P}_{\gamma,t,\epsilon(t)}
=\int_{\Omega}G(\mu_t,Q_t)1_{D_{t,\epsilon(t)}}{\rm d}\mathbb{P}_{\gamma}^{F,h}\\
\to&\int_{\Omega}G(\mu,Q){\rm d}\mathbb{P}_{\gamma}^{F,h}=G(\mu,Q),\qquad \text{ as }\ t\to \infty,
\end{split}
\end{equation}
where the convergence in \eqref{FmuQ} follows from the dominated convergence theorem.
\end{proof}

To prove the lower bound of the LDP, we only need to prove that the rate function $I$ coincides with $\sce J$. Before giving the desired results, we need the following two lemmas.

\begin{lemma}\label{proposition:graph}
Suppose that Assumption \ref{ass:irreducibility} is satisfied. Then for any $\mu=(\mu_x)_{x\in V}\in \mathcal{P}(V)$ satisfying $\mu_x>0$ for any $x\in V$, there exist a constant $C>0$ and an irreducible, positive recurrent Markov chain $(\hat{X}_k)_{k\ge 0}$ with state space $V$, transition probability matrix $\hat{P}=(\hat{p}_{xy})_{x,y\in V}$, and invariant distribution $\hat{\nu} = (\hat{\nu}_x)_{x\in V}$ such that $\hat{E}\subset E$ and $\hat{\nu}_x\le C\mu_x$ for any $x\in V$, where $\hat{E}=\{(x,y)\in V\times V:\hat{p}_{xy}>0\}$.
\end{lemma}

\begin{proof}
We first construct a sequence of finite subsets $V_n\uparrow V$ by induction. Without loss of generality, we assume that $V=\{z_i\}_{i\ge1}$. Since Assumption \ref{ass:irreducibility} holds, it is clear that $(V,E)$ is connected. Thus, there exists a self-avoiding cycle $(x_1(:=z_1),x_2\cdots,x_{k_0})$ of elements of $V$ such that $(x_i,x_{i+1})\in E$ when $i=1,\cdots,k_0$ and the sum in the indices is modulo $k_0$ (cycle $(x_1,\cdots,x_{k_0})$ is called self-avoiding if $x_i\neq x_j$ for any $1\le i\neq j\le k_0$). Set
\begin{equation*}
V_0=\{x_1,x_2,\cdots,x_{k_0}\}.
\end{equation*}
Suppose that we have constructed $V_n$ and $z_1,\cdots,z_q\in V_n$, $z_{q+1}\notin V_n$. Since $(V,E)$ is connected, it is easy to see that there exist $x_r\in V_n$ for some $1\le r\le k_n=|V_n|$ and a sequence of distinct states $w_1,\cdots,w_{m_1-1}\in V\setminus V_n$ such that $(w_i,w_{i+1})\in E$ when $i=0,\cdots,m_1-1$, where $w_0:=x_r$ and $w_{m_1}:=z_{q+1}$. Similarly, there exist $x_l\in V_n$ and a sequence of distinct states $w_{m_1+1},\cdots,w_{m_2}\in V\setminus V_n$ such that $(w_i,w_{i+1})\in E$ when $i=m_1,\cdots,m_2$, where $w_{m_2+1}:=x_l$.

If $\{w_1,\cdots,w_{m_1}\}\cap\{w_{m_1+1},\cdots,w_{m_2}\}=\emptyset$, let $m=m_2$ and $y_i=w_i$ for $i=1,\cdots,m$.

If $\{w_1,\cdots,w_{m_1}\}\cap\{w_{m_1+1},\cdots,w_{m_2}\}\neq\emptyset$, let $\kappa_1$ be the minimum integer such that $w_{\kappa_1}\in\{w_{m_1+1},\cdots,w_{m_2}\}$. Then there exists $m_1+1\le \kappa_2\le m_2$ such that $w_{\kappa_1}=w_{\kappa_2}$. Moreover, let $m=\kappa_1+m_2-\kappa_2$ and  $y_1=w_1,\cdots,y_{\kappa_1}=w_{\kappa_1},y_{\kappa_1+1}=w_{\kappa_2+1},\cdots,y_{m}=w_{m_2}$, it is clear that $y_1,\cdots,y_m$ is a sequence of distinct states of $V\setminus V_n$.

Set $k_{n+1}=k_n+m$ and
\begin{equation*}
V_{n+1}=\{x_1,\cdots x_{k_n},x_{k_n+1}(:=y_1),\cdots,x_{k_{n+1}}(:=y_m)\}.
\end{equation*}
Then we have constructed a sequence of subsets $\{V_n\}_{n\ge 0}$ such that $|V_n|<\infty$ and $V_n\subset V_{n+1}$. Moreover, following from the above constructions, it is easy to see that for any $z_q\in V$, there exists $V_n$ such that $z_q\in V_n$. This implies that $\cup_nV_n=V$.

We next construct a sequence of matrices $P^n=(p^n_{xy})_{x,y\in V}$ base on the previous construction of $\{V_n\}_{n\ge 0}$. Let
\begin{equation*}
p^0_{x_i,x_{i+1}}=1,\quad1\le i\le k_0-1,\qquad p^0_{x_{k_0},x_1}=1,\qquad p^0_{xy}=0, \quad\text{otherwise}.
\end{equation*}
Suppose that we have constructed $P^n=(p^n_{xy})_{x,y\in V}$. Let
\begin{equation}\label{pn+1}
\begin{split}
p^{n+1}_{x_r,x_{k_n+1}}=\alpha_{n+1},\qquad p^{n+1}_{x_i,x_{i+1}}=1,\quad k_n+1\le i\le k_{n+1}-1,\qquad p^{n+1}_{x_{k_{n+1}},x_l}=1,\\
p^{n+1}_{x_r,y}=(1-\alpha_{n+1})p^n_{x_r,y},\quad y\in V_n,\qquad p^{n+1}_{xy}=p^n_{xy},\quad\text{otherwise},\qquad\quad
\end{split}		
\end{equation}
where $x_r,x_l\in V_n$ are defined the same as in the above construction of $\{V_n\}_{n\ge 0}$ and $\alpha_{n+1}$ is a positive constant remaining to decide. Then the sequence of matrices $P^n$ is constructed by induction. It is easy to see that $P^n$ can be considered as a transition probability matrix with state space $V_n$, which corresponds to an irreducible Markov chain. We let $\nu^n\in\mathcal{P}(V)$ be the invariant distribution for $P^n$.
Now we decide $\{\alpha_n\}_{n\ge 0}$ and $C$ such that $\nu^n_x< C\mu_x$ for any $ n\in\mathbb{N}$ and $x\in V$ by induction. Let $C=1+(\max_{x\in V_0}(1/\mu_x))/k_0$ and $\alpha_0=0$. It is clear that
\begin{equation*}
\nu^0_x=\frac{1}{k_0}<C\mu_x,\quad x\in V_0,\qquad\nu^0_x=0<C\mu_x,\quad x\notin V_0.
\end{equation*}
Suppose that we have decided $(\alpha_k)_{0\le k\le n}$ such that $\nu^n_x< C\mu_x$ for any $x\in V$. Note that $\nu^{n+1}$ is continuous with respect to $\alpha_{n+1}$. In other words,
\begin{equation*}
\lim_{\alpha_{n+1}\downarrow 0}\nu^{n+1}_x=\nu^n_x,\quad  x\in V.
\end{equation*}
Since $V$ is countable, the strong convergence of $\nu^{n+1}$ to $\nu^n$ in $\mathcal{P}(V)$ is equivalent to the pointwise convergence of $\nu^{n+1}_x$ to $\nu^n_x$ for any $x\in V$. This implies $\lim_{\alpha_{n+1}\downarrow 0}\|\nu^{n+1}-\nu^n\|=0$. Hence, there exists $0<\alpha_{n+1}\le 1/2^{n+1}$ such that
\begin{equation}\label{nu}
\nu^{n+1}_x< C\mu_x,\quad x\in V_{n+1},\qquad \left\|\nu^{n+1}-\nu^n\right\|\le \frac{1}{2^n}.
\end{equation}
Since $\nu^{n+1}_x=0$ for any $x\notin V_{n+1}$, it is clear that $\nu^{n+1}_x< C\mu_x$ for any $x\in V$.

Finally, we construct $\hat{P}$. For any $x\in V$, since $V_n\uparrow V$, there exists $N$ such that $x\in V_N$. It follows from \eqref{pn+1} that
\begin{equation*}
\sum_{y\in V}\left|p^{n+1}_{xy}-p^n_{xy}\right|\le 2\alpha_{n+1}\le \frac{1}{2^{n}},\qquad  n\ge N.
\end{equation*}
This shows that $\hat{P}$ can be defined as $\hat{p}_{xy}=\lim_{n\to \infty}p^n_{xy}\ge 0$. By Fatou's lemma, we have
\begin{equation}\label{pnxy convergencce}
\sum_{y\in V}\left|p^n_{xy}-\hat{p}_{xy}\right|\le \varliminf_{m\to \infty}\sum_{y\in V}\left|p^n_{xy}-p^m_{xy}\right|\le \varliminf_{m\to \infty}\sum_{k=n}^{m-1}\frac{1}{2^k}=\frac{1}{2^{n-1}},\quad  n\ge N,
\end{equation}
which implies $\sum_{y\in V}\hat{p}_{xy}=1$ for any $x\in V$. In other words, $\hat{P}$ is a transition probability matrix. Note that if there exists $n$ such that $p^n_{xy}>0$, then
\begin{equation*}
\hat{p}_{xy}=\lim_{k\to \infty}p^k_{xy}\ge \left[\prod_{k=n+1}^{\infty}(1-\alpha_k)\right]p^n_{xy}\ge \left(1-\frac{1}{2^n}\right)p^n_{xy}>0.
\end{equation*}
Since $P^n$ corresponds to an irreducible Markov chain with state space $V_n$ and $V_n\uparrow V$, it is easy to see that $\hat{P}$ is irreducible.

On the other hand, it follows from \eqref{nu} that we can set $\hat{\nu}=\lim_{n\to \infty}\nu^n$. Moreover, it is clear that $\hat{\nu}\in \mathcal{P}(V)$ and $\hat{\nu}_x\le C\mu_x$ for any $x\in V$. It follows from \eqref{pnxy convergencce} that for any $y\in V$ and $N\in \mathbb{N}$,
\begin{align*}
\sum_{x\in V}\left|\nu^n_xp^n_{xy}-\hat{\nu}_x\hat{p}_{xy}\right|\le&\; \sum_{x\in V}\left|\nu^n_x-\hat{\nu}_x\right|p^n_{xy}+\sum_{x\in V_N}\hat{\nu}_x\left|p^n_{xy}-\hat{p}_{xy}\right|+\sum_{x\in V\setminus V_N}\hat{\nu}_x\left|p^n_{xy}-\hat{p}_{xy}\right|\\
\le &\; \left\|\nu^n-\hat{\nu}\right\|+\frac{1}{2^{n-1}}+1-\sum_{x\in V_N}\hat{\nu}_x,\quad n\ge N.
\end{align*}
Taking $n\to\infty$ and then taking $N\to \infty$ on both two sides of the above inequality, we have
\begin{equation*}
\hat{\nu}_y=\lim_{n\to \infty}\nu^n_y=\lim_{n\to \infty}\sum_{x\in V}\nu^n_xp^n_{xy}=\sum_{x\in V}\hat{\nu}_x\hat{p}_{xy},\quad y\in V.
\end{equation*}
Obviously, we have $\hat{E}\subset E$. Then we finish the proof of this lemma.
\end{proof}
\begin{lemma}\label{lemma:D_1}
Suppose that Assumptions~\ref{ass:irreducibility}-\ref{ass:locally finite} are satisfied. Then $\mathcal{D}_1\neq\emptyset$.
\end{lemma}
\begin{proof}
We first construct a transition kernel $(\hat{P}=(\hat{p}_{xy})_{x,y\in V},\hat{\Psi}=(\hat{\psi}_x)_{x\in V})$ satisfying the following four requirements:
\begin{itemize}
\item[(a)] $\hat{P}$ is an irreducible transition probability matrix with an unique invariant probability measure $\hat{\nu}$.
\item [(b)]$\hat{\psi}_x\ll\psi_x$ and $\hat{p}_{x,\cdot}\ll p_{x,\cdot}$ for any $x\in V$.
\item [(c)]$\sum_{x\in V}\hat{\nu}_x\int_{(0,\infty)}s\hat{\psi}_x(\mathrm{d}s)<\infty$.
\item [(d)]$\sum_{x\in V}\hat{\nu}_x[H(\hat{p}_{x,\cdot}|p_{x,\cdot})+H(\hat{\psi}_x|\psi_x)]<\infty$.
\end{itemize}
Without loss of generality, we assume that $V=\mathbb{N}$ is the set of nonnegative integers. Let $m_x=\sup\{c\ge 0:\psi_x((0,c))=0\}$, $A_x=[m_x,m_x+1]$, and $\hat{\psi}_x({\rm d}s)=\psi_x({\rm d}s|A_x):=1_{A_x}(s)\psi_x({\rm d}s)/\psi_x(A_x)$. Then $A_x$ is a bounded Borel subset of $(0,\infty)$ satisfying $\psi_x(A_x)>0$ and
\begin{equation}\label{tau psi}
\int_{(0,\infty)}s\hat{\psi}_x({\rm d}s)\le m_x+1, \qquad H(\hat{\psi}_x|\psi_x)=-\log\psi_x(A_x),\quad x\in\mathbb{N}.
\end{equation}
For any $x\in\mathbb{N}$, let $C_x=-\sum_{y\in\mathbb{N}}\log p_{xy}$. By Assumption \ref{ass:locally finite}, it is clear that $0\le C_x<\infty$.
Let $\mu\in \mathcal{P}(\mathbb{N})$ be defined by
\begin{equation}\label{mux}
\mu_x=\frac{1}{M}\frac{2^{-x}}{\max\{m_x+1,-\log\psi_x(A_x),C_x\}},\quad x\in \mathbb{N},
\end{equation}
where
\begin{equation*}
M=\sum_{y\in \mathbb{N}} \frac{2^{-y}}{\max\{m_y+1,-\log\psi_y(A_y),C_y\}}
\end{equation*}
is a normalization constant. By Proposition \ref{proposition:graph}, there exist a constant $C>0$ and a Markov chain $\hat{X}$ with transition probability matrix $\hat{P}$ and invariant distribution $\hat{\nu}$ such that $\hat{E}\subset E$ (i.e. $\hat{p}_{x,\cdot}\ll p_{x,\cdot}$ for any $x\in \mathbb{N}$) and $\hat{\nu}_x\le C\mu_x$ for any $x\in \mathbb{N}$. Note that
\begin{equation}\label{H hat p}
H(\hat{p}_{x,\cdot}|p_{x,\cdot})=\sum_{y\in \mathbb{N}}\hat{p}_{xy}\log\frac{\hat{p}_{xy}}{p_{xy}}\le \sum_{y:(x,y)\in E}\log\frac{1}{p_{xy}}\le C_x.
\end{equation}
Combining \eqref{tau psi}, \eqref{mux}, and \eqref{H hat p}, it is clear that $(\hat{P},\hat{\Psi})$ satisfies items (a)-(d) of the above requirements.

Finally, let $Z=\sum_{x\in V}\hat{\nu}_x\int_{(0,\infty)}s\hat{\psi}_x(\mathrm{d}s)$ and
\begin{equation*}
\mu^0(x,{\rm d}s) = \frac{\hat{\nu}_x \, s \, \hat{\psi}_x({\rm d}s)}{Z}, \qquad Q^0(x,y)=\frac{\hat{\nu}_x\, \hat{p}_{xy}}{Z}.
\end{equation*}
It is easy to check that $(\mu^0,Q^0)\in\mathcal{D}_1$. This completes the proof of this lemma.
\end{proof}
We are now in a position to finish the proof of the lower bound of the LDP.
\begin{proposition}
Suppose that Assumptions~\ref{ass:irreducibility}-\ref{ass:locally finite} are satisfied. Let $L^1_+(E)$ be endowed with the bounded weak* topology. Then under $\mathbb{P}_{\gamma}$, the law of $(\mu_t,Q_t)\in\Lambda$ satisfies an LDP lower bound with convex rate function $I$. Moreover, if Condition \ref{condition:ccomp} is satisfied, then $I$ is a good rate function.
\end{proposition}

\begin{proof}
Let $J$ be the functional defined in \eqref{def:J}. By Lemmas \ref{lemma:entropy lower bound} and \ref{t:plb}, we only need to prove $I= \sce J$. By Proposition \ref{weak LDP ub}, $I$ is convex and lower semicontinuous on $\Lambda$. It is then easy to see that $\sce J \ge I$.

We next prove the converse inequality. In fact, we only need to prove that for any $(\mu,Q)\in \Lambda$ with $I(\mu,Q)<\infty$, there exists a sequence $(\mu_n,Q_n)_{n\ge 0}$ in $\mathcal{D}_2$ such that $(\mu_n,Q_n)\to (\mu,Q)$ in $\Lambda$ and
\begin{equation}\label{D2}
\varlimsup_{n\to \infty}I(\mu_n,Q_n)\le I(\mu,Q).
\end{equation}
Here, we only prove that there exists a sequence $(\mu_n,Q_n)_{n\ge 0}$ in $\mathcal{D}_1$ such that $(\mu_n,Q_n)\to (\mu,Q)$ in $\Lambda$ and \eqref{D2} holds. The rest of the proof is similar to \cite[Lemma 2.5]{mariani2016large}.

By Lemma \ref{lemma:D_1}, there exists $(\mu^0,Q^0)\in\mathcal{D}_1$. For any $(\mu,Q)\in\mathcal{D}$ with $I(\mu,Q)<\infty$, let
\begin{equation*}
\mu_n = \left(1-\frac{1}{n}\right)\, \mu \, + \frac{1}{n}\, \mu^0, \qquad
Q_n = \left(1-\frac{1}{n}\right)\, Q+ \frac{1}{n}\, Q^0, \qquad n\ge 1.
\end{equation*}
Obviously, $(\mu_n,Q_n)\to(\mu,Q)$ in $\Lambda$ in the sense of the strong topology. Since $I$ is convex on $\Lambda$, we have
\begin{equation*}
I(\mu_n,Q_n) \leq \left(1-\frac{1}{n}\right) I(\mu,Q)+\frac{1}{n}I\left(\mu^0,Q^0\right).
\end{equation*}
Then it is easy to see that $(\mu_n,Q_n)\in\mathcal{D}_1$ and \eqref{D2} holds.  This completes the proof of the lower bound.

Finally, under Condition \ref{condition:ccomp}, it follows from Proposition \ref{proposition:etem} that $(\mu_t,Q_t)$ is exponentially tight.
This fact, together with the lower bound of the LDP implies that $I$ is a good rate function \cite[Lemma 1.2.18]{dembo1998large}.
\end{proof}

\begin{remark}\label{rmk:omega dense}
In fact, we can further prove that for any $(\mu,Q)\in \mathcal{D}$ with $I(\mu,Q)<\infty$, there exists a sequence $(\mu_n,Q)_{n\ge 0}$ in $\Lambda$ such that $\mu_n(x,\cdot)=\mu(x,\cdot)$ whenever $Q_x=0$, and $\mu_n(x,\{\infty\})=0$ whenever $Q_x>0$, $(\mu_n,Q)\to (\mu,Q)$ in $\Lambda$, and
\begin{equation*}
\varlimsup_{n\to \infty}I(\mu_n,Q)\le I(\mu,Q).
\end{equation*}
The proof is similar to that given in \cite[Lemma 2.5]{mariani2016large}.
\end{remark}

\section{Proof of Theorem \ref{LDP:strong topology}}\label{section:s topology}
Next we will prove the joint LDP for the empirical measure and empirical flow when $L^1_+(E)$ is endowed with the strong topology. Note that the bounded weak* topology is weaker than the strong topology \cite[Theorem 2.7.2]{megginson2012introduction}. In other words, any open (closed) subset of $L^1_+(E)$ under the bounded weak* topology is also open (closed) under the strong topology. Since we have established the joint LDP when $L^1_+(E)$ is endowed with the bounded weak* topology in Section \ref{section:bwt topology}, we only need to prove the exponential tightness of the empirical flow when $L^1_+(E)$ is endowed with the strong topology \cite[Corollary 4.2.6]{dembo1998large}. Before proving the exponential tightness, we introduce some notation.

Recall the definition of exit-current and entrance-current in \eqref{exit entrance current}. For any $t>0$, we define the associate empirical currents $Q^+_t:\Omega\to [0,\infty]^V$ and $Q^-_t:\Omega\to [0,\infty]^V$ as
\begin{equation}
\label{def:J_t}
\begin{split}
Q^+_t(x)=\sum_{y\in V}Q_t(x,y) =
\frac{1}{t}\sum_{k=1}^{N_t+1} \, 1_{(X_{k-1}=x)}, \quad Q^-_t(x)=\sum_{y\in V}Q_t(y,x) =
\frac{1}{t}\sum_{k=1}^{N_t+1} \, 1_{(X_{k}=x)}.
\end{split}
\end{equation}
It is clear that $Q^+_t,Q^-_t\in L^1_+(V)$, $\mathbb{P}_{\gamma}$-a.s. For any $J\in L^1_+(V)$ and $f:V\to\mathbb{R}$, we set $\langle J,f\rangle=\sum_{x\in V}J(x)f(x)$.

The exponential tightness of the empirical currents is stated in the following proposition. Here we also consider the case where the semi-Markov process starts from the general initial distribution $\gamma$ (see Remark \ref{remark:gamma}).
\begin{proposition}
\label{proposition:nut}
Assume Conditions \ref{condition:ccomp} and \ref{condition:ccomp2} to hold. Then there exists a sequence $\{\mathcal K_\ell\}$ of compact sets in $L^1_+(V)$ such that for any $\ell\in\mathbb N$,
\begin{equation}\label{et Q+Q-}
\varlimsup_{t\to \infty} \; \frac 1t
\log \mathbb  P_{\gamma} \big( Q^+_t \not\in \mathcal K_\ell \big) \le -\ell,
\end{equation}
where $L^1_+(V)$ is endowed with the strong topology. In particular, the empirical exit-current is exponentially tight.
\end{proposition}

\begin{proof}
We first construct the compact sets in $L^1_+(V)$ under the strong topology. Let $W_m  \uparrow V$ be an invading sequence of finite subsets of $V$, for any positive integer $\ell$, we let
\begin{equation*}
\mathcal K_{\ell}=  \left\{ J \in  L^1_+(V)\,:\, \|J\|\leq A_{\ell} \,, \; \left\langle J,1_{W^c_m}\right\rangle \leq \frac{1}{m},\; \; \forall m \geq \ell \right\} .
\end{equation*}
where $A_{\ell}$ is defined as in Proposition \ref{proposition:etem}. Similarly to the proof in \cite[Theorem 5.2]{bertini2015flows}, it is easy to prove that $\mathcal K_{ \ell} \subset L^1_+(V)$ is compact under the strong topology.

Now we prove \eqref{et Q+Q-}. It is easy to see that
\begin{equation} \label{J_t notin}
\mathbb P_{\gamma}  (Q^+_t \not \in \mathcal K_{\ell}) \leq \mathbb P_{\gamma} \bigl( \|Q^+_t \| \geq A_{\ell}) + \sum _{m \geq \ell} \mathbb P_{\gamma} \left(\left\langle Q^+_t,1_{W^c_m}\right\rangle >\frac{1}{m}  \right) \,.
\end{equation}
It then follows from \eqref{QT exponential} and \eqref{def:J_t} that
\begin{equation*}
\varlimsup_{t\to\infty}\frac{1}{t}\log\mathbb P_{\gamma} \bigl( \|Q^+_t \| \geq A_{\ell})=\varlimsup_{t\to\infty}\frac{1}{t}\log\mathbb P_{\gamma} \bigl( \|Q_t \| \geq A_{\ell}) \le -\ell.
\end{equation*}
On the other hand, let the function $\hat{u}$, the sequence of functions $u_n$, and the constants $c,C_{\gamma}$ be as in Condition \ref{condition:ccomp2} and item (c*) in Remark \ref{remark:gamma}. Taking $A=V$ in Lemma \ref{t:em1}, we obtain the local martingale
\begin{equation}\label{martingle un V}
\mathcal M^{u_n,V}_t =  \frac{u_n(X_{N_t+1})}{u_n(X_0)}
\exp\left\{ t \, \Big\langle \hat{\mu}_t, h^{u_n,V}
\Big\rangle\right\}= \frac{u_n(X_{N_t+1})}{u_n(X_0)}
\exp\left\{t \left\langle Q^+_t, \log\frac{u_n}{Pu_n}\right\rangle\right\}.
\end{equation}
Note that $u_n/Pu_n$ converges pointwise to $\hat{u}$. By Fatou's lemma, we have
\begin{align*}
\mathbb E_{\gamma} \left( \exp\left\{ t \Big\langle Q^+_t, \log\hat{u}\Big\rangle \right\} \right)&\le \sum_{x\in V}\gamma(x)\varliminf_{n\to\infty}\mathbb E_{x} \left( \exp\left\{ t \left\langle Q^+_t, \log\frac{u_n}{Pu_n}\right\rangle \right\} \right)\\
&=\sum_{x\in V}\gamma(x)\varliminf_{n\to\infty}\mathbb E_{x} \left( \frac{u_n(X_0)}{u_n(X_{N_t+1})}\mathcal M^{u_n,V}_t\right)\\
&\le \sum_{x\in V}\gamma(x)\frac{u_n(x)}{c}\le \frac {C_{\gamma}}{c},
\end{align*}
where the second inequality follows from item (b) in Condition \ref{condition:ccomp2} and the last inequality follows from item (c*) in Remark \ref{remark:gamma}.

Let $\{a_m\}_{m\geq 0}$ be a sequence of constants with $a_m\uparrow \infty$ to be chosen later and let $W_m =\{x\in V :\, \log\hat{u}(x) \le a_m\}$ be an invading sequence of $V$. In view of item (e) in Condition \ref{condition:ccomp2}, $W_m$ are finite sets. Let
\begin{equation*}
C = 1\vee(-\inf_{x\in V}\log\hat{u}(x))<\infty.
\end{equation*}
It is easy to see that $\log\hat{u}\ge a_{m}1_{W^c_{m}}-C$. By the exponential Chebyshev inequality, we obtain
\begin{equation*}
\begin{split}
\mathbb P_{\gamma}\left(\left\langle Q^+_t,1_{W_m^c}\right\rangle>\frac{1}{m}\right) \le
&\;\mathbb P_{\gamma} \left( \left\langle Q^+_t, \log\hat{u}\right\rangle +C\|Q^+_t\|>\frac{a_m}{m}\right)\\
\le&\;\mathbb P_{\gamma} \left( \left\langle Q^+_t, \log\hat{u}\right\rangle >
\frac{a_m}{2m}\right)+\mathbb P_{\gamma} \left( C\|Q_t\| >
\frac{a_m}{2m}\right)\\
\le&\;\frac{C_{\gamma}}{c}\exp\left\{ -t \frac{a_m}{2m}
\right\}+\mathbb P_{\gamma} \left( \|Q_t\| >
\frac{a_m}{2Cm}\right).
\end{split}
\end{equation*}
By choosing $a_m = 2m^2 +2Cm A_{m}$, the proof is now easily concluded from \eqref{J_t notin}.
\end{proof}

\begin{corollary}\label{corollary:nut}
Assume Conditions \ref{condition:ccomp} and \ref{condition:ccomp2} to hold. Then the empirical entrance-current $Q^-_t$ with $L^1_+(V)$ endowed with the strong topology is exponentially tight.
\end{corollary}
\begin{proof}
For any $\eta>0$, it is easy to see that
\begin{equation*}
\varlimsup_{t\to\infty}\frac{1}{t}\log\mathbb{P}_{\gamma}(\|Q^+_t-Q^-_t\|>\eta)=-\infty.
\end{equation*}
The result of this corollary now immediately follows from \cite[Lemma 3.13]{feng2006large}.
\end{proof}

We are now in a position to prove Theorem \ref{LDP:strong topology}.

\begin{proof}[Proof of Theorem \ref{LDP:strong topology}]
Let $(Z,\tau)=\{(Z_k)_{k\ge 0},(\tau_k)_{k\ge 1}\}$ be a Markov renewal process with the initial distribution $\gamma^Z$, where $Z_k = (X_{k-1},X_{k})$ and $X_{-1}$ can be any random variables such that $\gamma^Z\in\mathcal{P}(E)$. Note that the empirical entrance-current for $(Z,\tau)$ is exactly the empirical flow for $(X,\tau)$, i.e.
\begin{equation*}
Q_t(x,y)=\frac{1}{t}\sum_{k=1}^{N_t+1}1_{(X_{k-1}=x,X_k=y)}=\frac{1}{t}\sum_{k=1}^{N_t+1}1_{(Z_k=(x,y))}=Q_t^{-,Z}(x,y).
\end{equation*}
Next we will apply Corollary \ref{corollary:nut} to $(Z,\tau)$. It is easy to verify that $(Z,\tau)$ satisfies Assumptions \ref{ass:irreducibility}-\ref{ass:locally finite}. In order to apply Corollary \ref{corollary:nut}, we need to prove that $(Z,\tau)$ satisfies Conditions \ref{condition:ccomp} and \ref{condition:ccomp2}. Here we only verify Condition \ref{condition:ccomp} for $(Z,\tau)$ and the proof of Condition \ref{condition:ccomp2} is similar.

Let the functions $\hat{u}$, the sequence of functions $u_n$, the set $K$, and the constants $c,\sigma,C,\eta,C_{\gamma}$ be as in Condition \ref{condition:ccomp} and item (c*) in Remark \ref{remark:gamma}. By choosing $u_n^Z(x,y)=u_n(y)$, we immediately obtain item (b). Note that $Z$ has the transition probability $p^Z_{(x,y),(z,w)}=\delta_y(z)p_{zw}$. Then for any $(x,y)\in E$ and $n\ge 0$, we have
\begin{equation*}
P^Zu^Z_n(x,y)=\sum_{(z,w)\in E}\delta_y(z)p_{zw}u^Z_n(z,w)=\sum_{w\in V}p_{yw}u_n(w)=Pu_n(y)<\infty,
\end{equation*}
which implies item (a). For any $n\ge 0$, we have
\begin{equation*}
\sum_{(x,y)\in E}\gamma^Z(x,y)u^Z_n(x,y)=\sum_{(x,y)\in E}\gamma^Z(x,y)u_n(y)=\sum_{y\in V}\gamma(y)u_n(y)\le C_{\gamma}.
\end{equation*}
By choosing $C^Z_{\gamma}=C_{\gamma}$, we obtain item (c). Note that
\begin{equation*}
\lim_{n\to\infty}\frac{u^Z_n(x,y)}{P^Zu^Z_n(x,y)}=\lim_{n\to\infty}\frac{u_n(y)}{Pu_n(y)}=\hat{u}(y).
\end{equation*}
By choosing $\hat{u}^Z(x,y)=\hat{u}(y)$, we obtain item (d). It is easy to check that $L^Z\hat{u}^Z(x,y)=L\hat{u}(y)$. Then for each $\ell\in\mathbb{R}$, we have
\begin{equation*}
\left\{(x,y)\in E:L^Z\hat{u}^Z(x,y)\le \ell\right\}=\left\{(x,y)\in E:L\hat{u}(y)\le \ell\right\}.
\end{equation*}
Since $(V,E)$ is locally finite, item (e) also holds. Note that $\psi^Z_{(x,y)}=\psi_y$ and $\zeta^Z(x,y)=\zeta(y)$. By choosing $\sigma^Z=\sigma$, $C^Z=C$, $\eta^Z=\eta$, and $K^Z=(V\times K)\cap E$, we obtain $\hat{u}^Z(x,y)<\psi^Z_{(x,y)}(e^{\zeta^Z(x,y)\tau})$ for any $(x,y)\in (K^Z)^c$ and
\begin{equation*}
L^Z\hat{u}(x,y)=L\hat{u}(y)\ge-\sigma L\eta(y)-C1_K(y)\ge-\sigma^Z L^Z\eta^Z(x,y)-C^Z1_{K^Z}(x,y).
\end{equation*}
This completes the proof of item (f).
\end{proof}

\section{Proof of Proposition \ref{proposition:rate function for empirical flow}}\label{section:contraction principle}
Here we consider the marginal LDP for the empirical flow $Q_t$. Before giving the proof of Proposition \ref{proposition:rate function for empirical flow}, we need some notation and lemmas. Let $\mu$ and $\nu$ be two positive $\sigma$-finite measures on a measurable space $(\mathcal{X},\mathcal{F})$. For any sequence of non-negative measurable functions $(f_i)_{i\ge 1}$ on $\mathcal{X}$ and any sequence of non-negative constants $(b_i)_{i\ge 0}$ satisfying $0<g:=b_0+\sum_{i=1}^{\infty}b_if_i<\infty$, $\mu$-a.s., let the generalized relative entropy between $\mu$ and $\nu$ be defined by
\begin{equation*}
H_{g}(\mu|\nu)= \left\{
\begin{aligned}
&\int_{\mathcal{X}}  g\left(\log\frac{\mathrm{d}\mu}{\mathrm{d}\nu}\right)\mathrm{d}\mu,    && \text{if } \mu\ll\nu,\\
&\infty,    && \text{otherwise}.\\
\end{aligned}\right.
\end{equation*}
Similarly to the maximum entropy principle \cite[Theorem 12.1.1]{covert2006elementsofinformationtheory}, we have the following lemma.

\begin{lemma}\label{lemma:generalized maximum entropy principle}
Suppose that there exists a positive $\sigma$-finite measure $\mu^*$ satisfying $\mu^*\ll\nu$, $\nu\ll\mu^*$, and
\begin{equation}\label{equation for mu*}
g\log\frac{\mathrm{d}\mu^*}{\mathrm{d}\nu}=\lambda_0+\sum_{i=1}^{\infty}\lambda_i f_i,
\end{equation}
where the sequence of constant $(\lambda_i)_{i\ge 0}$ are chosen so that $\mu^*$ satisfies
the following constraints:
\begin{equation}\label{constraints}
\mu(\mathcal{X})=a_0,\quad\int_{\mathcal{X}} f_id\mu=a_i, \quad  i\ge 1,
\end{equation}
where $(a_i)_{i\geq 0}$ is a sequence of constants satisfying $0<\sum_{i=0}^{\infty}a_ib_i<\infty$. Then $\mu^*$ uniquely minimizes $H_{g}(\cdot|\nu)$ over all positive $\sigma$-finite probability measures satisfying \eqref{constraints}. Moreover, the minimum is given by
\begin{equation*}
H_{g}(\mu^*|\nu)=\sum_{i=0}^\infty a_i\lambda_i.
\end{equation*}
\end{lemma}

\begin{proof}
Let $\mu$ be a positive $\sigma$-finite measure satisfying \eqref{constraints}. Note that $\int_{\mathcal{X}} g\mathrm{d}\mu=\int_{\mathcal{X}} g\mathrm{d}\mu^*=\sum_{i=0}^{\infty}a_ib_i<\infty$. Then we have
\begin{equation}\label{inequality}
\begin{split}
\int_{\mathcal{X}} g\left(\log\frac{\mathrm{d}\mu}{\mathrm{d}\nu}\right)\mathrm{d}\mu
=&\; \int_{\mathcal{X}} g\left(\log\frac{\mathrm{d}\mu}{\mathrm{d}\mu^*}\right)\mathrm{d}\mu+\int_{\mathcal{X}} g\left(\log\frac{\mathrm{d}\mu^*}{\mathrm{d}\nu}\right)\mathrm{d}\mu\\
=&\; \left(\int_{\mathcal{X}} g\mathrm{d}\mu\right) H\left(\frac{g\mathrm{d}\mu}{\int_{\mathcal{X}} g\mathrm{d}\mu}\bigg|\frac{g\mathrm{d}\mu^*}{\int_{\mathcal{X}} g\mathrm{d}\mu^*}\right)+\int_{\mathcal{X}} g\left(\log\frac{\mathrm{d}\mu^*}{\mathrm{d}\nu}\right)\mathrm{d}\mu\\
\ge&\;\int_{\mathcal{X}} g\left(\log\frac{\mathrm{d}\mu^*}{\mathrm{d}\nu}\right)\mathrm{d}\mu,
\end{split}
\end{equation}
where the last inequality follows from the nonnegativity of the relative entropy. It then follows from \eqref{equation for mu*} that
\begin{align*}
\int_{\mathcal{X}} g\left(\log\frac{\mathrm{d}\mu^*}{\mathrm{d}\nu}\right)\mathrm{d}\mu=&\int_{\mathcal{X}}\left(\lambda_0+\sum_{i=1}^{\infty}\lambda_i f_i\right)\mathrm{d}\mu=\sum_{i=0}^{\infty}\lambda_i a_i=\int_{\mathcal{X}}\left(\lambda_0+\sum_{i=1}^{\infty}\lambda_i f_i\right)\mathrm{d}\mu^*\\
=&\int_{\mathcal{X}} g\left(\log\frac{\mathrm{d}\mu^*}{\mathrm{d}\nu}\right)\mathrm{d}\mu^*.
\end{align*}
Hence we have proved that $H_g(\mu|\nu)\ge H_g(\mu^*|\nu)$. Here the equality holds if and only if $\mathrm{d}\mu/\mathrm{d}\mu^*=1$. This shows that except for a set of measure zero, $\mu^*$ is unique.
\end{proof}

Recall the definition of $\zeta(x)$ in \eqref{def:zeta}. We also need the following lemma.

\begin{lemma}\label{lemma:function properties}
For any $x\in V$, let $G_x(\lambda)=\log(\psi_x(e^{\lambda\tau}))$ and $F_x(\lambda)=\psi_x(\tau e^{\lambda\tau})/\psi_x( e^{\lambda\tau})$, where $\psi_x(\tau e^{\lambda\tau})=\int_{(0,\infty)}se^{\lambda s}\psi_x(\mathrm{d}s)$. Suppose that $G_x(\zeta(x))=\infty$ for any $x\in V$. Then $G_x\in C^2(-\infty,\zeta(x))$ is a strictly increasing function, $F_x=\mathrm{d}G_x/\mathrm{d}\lambda$ is an increasing function, and
\begin{equation}\label{m_x}
\lim_{\lambda\to -\infty}F_x(\lambda) = m_x:=\sup\left\{\lambda\ge0:\psi_x(0,\lambda)=0\right\},
\end{equation}
\begin{equation}\label{M_x}
\lim_{\lambda\to\zeta(x)}F_x(\lambda) = M_x:=\inf\left\{\lambda\ge0:\psi_x(\lambda,\infty)=0\right\}.
\end{equation}
Moreover, if $\psi_x$ is not a Dirac measure, then $F_x$ is strictly increasing.
\end{lemma}

\begin{proof}
Without loss of generality, we can drop the dependence on $x$ in the proof. We first prove that $G\in C^1((-\infty,\zeta))$ and $\mathrm{d}G/\mathrm{d}\lambda=F$. Note that $f_{\epsilon}(s):=(e^{(\lambda+\epsilon)s}-e^{\lambda s})/\epsilon$ converge pointwise to $s e^{\lambda s}$ as $\epsilon\to 0$, $|f_{\epsilon}(s)|\le h(s):=e^{\lambda s}(e^{\eta s}-1)/\eta$ for every $\epsilon\in(-\eta,\eta)$, and $\langle\psi,|h|\rangle<\infty$ for $\eta>0$ small enough. By dominated convergence theorem, we immediately obtain the results. Moreover, since $F>0$, it is easy to see that $G$ is strictly increasing.

We next prove that $F\in C^1((-\infty,\zeta)$ is increasing. The proof of differentiability is similar to the above. Direct computations show that
\begin{equation*}
\frac{\mathrm{d}F}{\mathrm{d}\lambda}(\lambda)=\frac{\psi(\tau^2e^{\lambda\tau})\psi(e^{\lambda\tau})-(\psi(\tau e^{\lambda\tau}))^2}{(\psi( e^{\lambda\tau}))^2}\ge 0,
\end{equation*}
where the last inequality follows from Cauchy-Schwarz inequality. Moreover, the above equality holds if and only if $\psi$ is a Dirac measure.

Finally, we prove \eqref{m_x} and \eqref{M_x}. We only prove \eqref{M_x} and the proof of \eqref{m_x} is similar. Note that $F$ is an increasing function. Let $L=\lim_{\lambda\to\zeta}F(\lambda)$.

Case 1: $\zeta <\infty$. It is easy to see that $M=\infty$. For any $0\le\lambda\le\zeta$ and $N>1$, we have
\begin{equation}\label{psi e lambda tau}
\psi\left(e^{\lambda\tau}\right)=\int_{(0,N]} e^{\lambda s}\psi(\mathrm{d}s)+\int_{(N,\infty)} e^{\lambda s}\psi(\mathrm{d}s)\le e^{\lambda N}+\int_{(N,\infty)} e^{\lambda s}\psi(\mathrm{d}s).
\end{equation}
Note that $\lim_{\lambda\to \zeta}\psi(e^{\lambda\tau})=e^{G(\zeta)}=\infty$. Taking $\lambda\to\zeta$ on both sides of \eqref{psi e lambda tau}, we obtain
\begin{equation*}
\lim_{\lambda\to \zeta}\int_{(N,\infty)} e^{\lambda s}\psi(\mathrm{d}s)=\infty.
\end{equation*}
On the other hand, we have
\begin{equation}\label{F(lambda)}
F(\lambda)=\frac{\psi(\tau e^{\lambda \tau})}{\psi(e^{\lambda \tau})}
\ge  \frac{\int_{(N,\infty)}s e^{\lambda s}\psi(\mathrm{d}s)}{e^{\lambda N}+\int_{(N,\infty)} e^{\lambda s}\psi(\mathrm{d}s)}\ge  \frac{N\int_{(N,\infty)} e^{\lambda s}\psi(\mathrm{d}s)}{e^{\lambda N}+\int_{(N,\infty)} e^{\lambda s}\psi(\mathrm{d}s)}.
\end{equation}
Taking $\lambda\to \zeta$ on both sides of \eqref{F(lambda)}, we obtain $L\ge N$. Since $N>1$ is arbitrary, it is clear that $L=\infty=M$.

Case 2: $\zeta=\infty$. Note that
\begin{equation}\label{F(lambda)2}
F(\lambda)=\frac{\int_{(0,M]}s e^{\lambda s}\psi(\mathrm{d}s)}{\int_{(0,M]} e^{\lambda s}\psi(\mathrm{d}s)}
\le \frac{\int_{(0,M]}M e^{\lambda s}\psi(\mathrm{d}s)}{\int_{(0,M]} e^{\lambda s}\psi(\mathrm{d}s)}=M.
\end{equation}
Taking $\lambda\to\infty$ on both sides of \eqref{F(lambda)2}, we obtain $L \le M$.

If $M<\infty$, for any $0<\epsilon<M/2$ and $\lambda>0$, we have
\begin{equation}\label{F(lambda)3}
\begin{split}
\frac{1}{F(\lambda)}=&\;\frac{\int_{(0,M-2\epsilon]} e^{\lambda s}\psi(\mathrm{d}s)+\int_{(M-2\epsilon,M]} e^{\lambda s}\psi(\mathrm{d}s)}{\int_{(0,M]} s e^{\lambda s}\psi(\mathrm{d}s)}\\
\le &\; \frac{\int_{(0,M-2\epsilon]} e^{\lambda s}\psi(\mathrm{d}s)}{\int_{(M-\epsilon,M]} s e^{\lambda s}\psi(\mathrm{d}s)}+\frac{\int_{(M-2\epsilon,M]} e^{\lambda s}\psi(\mathrm{d}s)}{\int_{(M-2\epsilon,M]} s e^{\lambda s}\psi(\mathrm{d}s)}\\
\le &\;\frac{ e^{\lambda(M-2\epsilon)}}{ (M-\epsilon) e^{\lambda(M-\epsilon)}\psi((M-\epsilon,M])}+\frac{\int_{(M-2\epsilon,M]} e^{\lambda s}\psi(\mathrm{d}s)}{\int_{(M-2\epsilon,M]} (M-2\epsilon) e^{\lambda s}\psi(\mathrm{d}s)}\\
=&\;\frac{e^{-\lambda\epsilon}}{(M-\epsilon)\psi((M-\epsilon,M])}+\frac{1}{M-2\epsilon}.
\end{split}
\end{equation}
Taking $\lambda\to \infty$ on both sides of \eqref{F(lambda)3}, we obtain $L\ge M-2\epsilon$. Since $0<\epsilon<M/2$ is arbitrary, it is clear that $L\ge M$.

If $M=\infty$, for any $\epsilon>0$, $N>1$, and $\lambda>0$, we have
\begin{equation}\label{F(lambda)4}
\begin{split}
\frac{1}{F(\lambda)}=&\;\frac{\int_{(0,N]} e^{\lambda s}\psi(\mathrm{d}s)+\int_{(N,\infty)} e^{\lambda s}\psi(\mathrm{d}s)}{\int_{(0,\infty)} s e^{\lambda s}\psi(\mathrm{d}s)}\\
\le &\; \frac{\int_{(0,N]} e^{\lambda s}\psi(\mathrm{d}s)}{\int_{(N+\epsilon,\infty)} s e^{\lambda s}\psi(\mathrm{d}s)}+\frac{\int_{(N,\infty)} e^{\lambda s}\psi(\mathrm{d}s)}{\int_{(N,\infty)} s e^{\lambda s}\psi(\mathrm{d}s)}\\
\le &\;\frac{ e^{\lambda N}}{ (N+\epsilon) e^{\lambda(N+\epsilon)}\psi((N+\epsilon,\infty))}+\frac{\int_{(N,\infty)} e^{\lambda s}\psi(\mathrm{d}s)}{\int_{(N,\infty)} N e^{\lambda s}\psi(\mathrm{d}s)}\\
=&\;\frac{e^{-\lambda\epsilon}}{(N+\epsilon)\psi((N+\epsilon,\infty))}+\frac{1}{N}.
\end{split}
\end{equation}
Taking $\lambda\to \infty$ on both sides of \eqref{F(lambda)4}, we obtain $L\ge N$. Since $N>1$ is arbitrary, it is clear that $L=\infty=M$.
\end{proof}

By Lemma \ref{lemma:function properties}, we immediately obtain the following corollary.

\begin{corollary}\label{G_Q}
For any $Q\in L^1_+(E)$ satisfying $Q^+=Q^-$ and $V_+:=\{x\in V:Q_x>0\}\neq \emptyset$, let $G_Q(\lambda)=\sum_{x\in V}Q_xG_x(\lambda)$ and $F_Q(\lambda)=\sum_{x\in V}Q_xF_x(\lambda)$. Suppose that $G_Q(\zeta_Q)=\infty$. Then $G_Q\in C^2(-\infty,\zeta_Q)$ is a strictly increasing function, $F_Q=\mathrm{d}G_Q/\mathrm{d}\lambda$ is an increasing function, and
\begin{equation}\label{limit for F_Q}
\lim_{\lambda\to-\infty}F_Q(\lambda)=m_Q:=\sum_{x\in V}Q_xm_x,\qquad\lim_{\lambda\to\zeta_Q}F_Q(\lambda)=M_Q:=\sum_{x\in V}Q_xM_x,
\end{equation}
where $\zeta_Q=\sup\{\lambda\ge 0:G_Q(\lambda)<\infty\}$. Moreover, if $m_Q<M_Q$, then $F_Q$ is strictly increasing.
\end{corollary}

\begin{proof}
We only prove the second equality of \eqref{limit for F_Q}. The proof of the first equality is similar.

Case 1: $\zeta_Q=\infty$. By the monotone convergence theorem, we immediately obtain the desired result.

Case 2: $\zeta_Q<\infty$. Note that $\zeta_Q\le \zeta(x)$ for any $x\in V_+$. By Lemma \ref{lemma:function properties}, we have $F_x(\zeta_Q)\le F_x(\zeta(x))=M_x$. It then follows from the monotone convergence theorem that
\begin{equation*}
\lim_{\lambda\to\zeta_Q}F_Q(\lambda)=\lim_{\lambda\to\zeta_Q}\sum_{x\in V_+}Q_xF_x(\lambda)\le \sum_{x\in V_+}Q_xM_x.
\end{equation*}
On the other hand, it is easy to see that
\begin{equation*}
G_Q(\zeta_Q)\le F_Q(\zeta_Q)\zeta_Q+G_Q(0).
\end{equation*}
Since $G_Q(\zeta_Q)=\infty$, we have $F_Q(\zeta_Q)=\infty$.
This completes the proof of this corollary.
\end{proof}

For any $Q\in L^1_+(E)$ satisfying $Q^+=Q^-$ and $V_+\neq \emptyset$, let $\mathcal{X}=V_+\times (0,\infty)$ be endowed with the product topology and let $\mathcal{F}$ be the associated Borel $\sigma$-field. Let $\nu(x,\mathrm{d}s)=s Q_x \psi_x(\mathrm{d}s)$ be a $\sigma$-finite measure on $(\mathcal{X},\mathcal{F})$. Here we take the sequence of functions $(f_x)_{x\in V_+}$ and the sequence of constants $(b_0,(b_x)_{x\in V_+})$ as
\begin{equation*}
f_x(y,s)=\delta_x(y)\frac{1}{s},\quad (y,s)\in \mathcal{X},\qquad b_0=0,\quad b_x=1.
\end{equation*}
Then the associated function $g$ is given by $g(y,s)=b_0+\sum_{x\in V_+}b_xf_x(y,s)=1/s$ and the associated generalized relative entropy is given by
\begin{equation*}
H_g(\mu|\nu)=\sum_{x\in V_+}\int_{(0,\infty)}\frac{1}{s}\left(\log\frac{\mu(x,\mathrm{d}s)}{s Q_x\psi_x(\mathrm{d}s)}\right)\mu(x,\mathrm{d}s).
\end{equation*}

\begin{lemma}\label{corollary:G*}
Let $a>0$ be a constant and let $Q\in L^1_+(E)$ be a flow satisfying $Q^+=Q^-$ and $V_+\neq \emptyset$. If $G_Q(\zeta_Q)=\infty$, then we have
\begin{equation*}
\inf_{\mu}H_{g}(\mu|\nu)=G_Q^*(a)=\left\{\begin{aligned}
&a\lambda^*-G_Q(\lambda^*),
&&\text {if  }m_Q<a<M_Q, \\
&-\sum_{x\in V}Q_x\log\psi_x(\{m_x\}),
&&\text {if  }a=m_Q, \\		
&-\sum_{x\in V}Q_x\log\psi_x(\{M_x\}),
&&\text {if  }a=M_Q, \\	
&\infty,  &&\text {otherwise},
\end{aligned}\right.
\end{equation*}
where $\mu$ in the infimum ranges over all positive $\sigma$-finite measures on $(\mathcal{X},\mathcal{F})$ satisfying
\begin{equation}\label{constraints2}
\mu(\mathcal{X})=\sum_{x\in V_+}\mu(x,(0,\infty))=a,\qquad \int_{\mathcal{X}} f_x \mathrm{d}\mu=\int_{(0,\infty)}\frac{1}{s}\mu(x,\mathrm{d}s)=Q_x,\quad x\in V_+,
\end{equation}
$\lambda^*$ is any solution of the equation $F_Q(\lambda)=a$, and $G_Q^*(a)=\sup_{\lambda\in \mathbb{R}}\{a\lambda-G_Q(\lambda)\}$, $a\in\mathbb{R}$ is the Fenchel-Legendre transform of $G_Q$.
\end{lemma}

\begin{proof}
We will prove this lemma in three different cases.

Case 1: $m_Q<a<M_Q$. Let $\lambda_0=\lambda^*$ and $\lambda_x=-G_x(\lambda^*)$. Let $\mu^*(x,\mathrm{d}s)=s Q_x e^{\lambda^*s+\lambda_x}\psi_x(\mathrm{d}s)$ be a $\sigma$-finite measure $\mu^*$ on $(X,\mathcal{F})$. Then \eqref{equation for mu*} holds and
\begin{equation*}
\frac{1}{s}\log\frac{\mu^*(x,\mathrm{d}s)}{s Q_x \psi_x(\mathrm{d}s)}=\lambda^*+\lambda_x\frac{1}{s},\quad  (x,s)\in \mathcal{X}.
\end{equation*}
Moreover, it is easy to check that $\mu^*$ satisfies \eqref{constraints2}. By Lemma \ref{lemma:generalized maximum entropy principle}, we have \begin{equation*}
\inf_{\mu}H_g(\mu|\nu)=a\lambda^*+\sum_{x\in V_+}\lambda_x Q_x=a\lambda^*-G_Q(\lambda^*).
\end{equation*}
On the other hand, it follows from Corollary \ref{G_Q} and the intermediate value theorem that there exists a solution $\lambda^*\in (-\infty,\zeta_Q)$ of the equation $F_Q(\lambda)=a$. Straightforward computations show that
\begin{equation*}
\sup_{\lambda\in \mathbb{R}}\left\{a\lambda-G_Q(\lambda)\right\}=a\lambda^*-G_Q(\lambda^*).
\end{equation*}

Case 2: $a=m_Q$ or $a=M_Q$. Here we only consider the case of $a=m_Q$. The proof in the case of $a=M_Q$ is similar. Let $\mu$ be a $\sigma$-finite measure satisfying \eqref{constraints2}. Then we have
\begin{equation*}
a=m_Q=\sum_{x\in V_+}m_x\int_{[m_x,\infty)}\frac{1}{s}\mu(x,\mathrm{d}s)\le \mu(X)=a.
\end{equation*}
The above equality hold if and only if $\mu(x,\mathrm{d}s)=Q_xm_x\delta_{m_x}(\mathrm{d}s)$ for any $x\in V_+$. Thus we have
\begin{equation*}
\inf_{\mu}H_g(\mu|\nu)=-\sum_{x\in V_+}Q_x\log\psi_x(\{m_x\}).
\end{equation*}
On the other hand, direct computations show that
\begin{align*}
\sup_{\lambda\in \mathbb{R}}\left\{a\lambda-G_Q(\lambda)\right\} &= \lim_{\lambda\to-\infty}\left\{a\lambda-G_Q(\lambda)\right\}\\
&= \lim_{\lambda\to-\infty}\sum_{x\in V_+}Q_x\log\frac{e^{\lambda m_x}}{e^{\lambda m_x}\psi_x(\{m_x\})+\int_{(m_x,\infty)}e^{\lambda s}\psi(\mathrm{d}s)}.
\end{align*}
By the dominated convergence theorem, it is easy to see that
\begin{equation*}
\lim_{\lambda\to-\infty}\int_{(m_x,\infty)}e^{\lambda(s-m_x)}\psi_x(\mathrm{d}s)=0
\end{equation*}
for any $x\in V_+$. This implies that
\begin{equation*}
\sup_{\lambda\in \mathbb{R}}\left\{a\lambda-G_Q(\lambda)\right\} = -\sum_{x\in V_+}Q_x\log\psi_x(\{m_x\}).
\end{equation*}

Case 3: $a<m_Q$ or $a>M_Q$. Here we only consider the case of $a<m_Q$. The proof in the case of $a>M_Q$ is similar. Note that there is no $\sigma$-finite measure $\mu$ satisfying \eqref{constraints2}. Then we have $\inf_{\mu}H_g(\mu|\nu)=\infty$. On the other hand, take a sequence of constants $a_x$ such that $0<a_x<m_x$ for any $ x\in V_+$ and $a=\sum_{x\in V_+}Q_xa_x$. Straightforward computations show that
\begin{align*}
\sup_{\lambda\in \mathbb{R}}\left\{a\lambda-G_Q(\lambda)\right\} &= \lim_{\lambda\to-\infty}\left\{a\lambda-G_Q(\lambda)\right\}\\
&= \lim_{\lambda\to-\infty}\sum_{x\in V_+}Q_x\log\frac{1}{\int_{[m_x,\infty)}e^{\lambda(s-a_x)}\psi(\mathrm{d}s)}=\infty,
\end{align*}
where the last equality follows from the dominated convergence theorem.
\end{proof}

The following lemma gives the properties of $G_Q^*$.

\begin{lemma}\label{lemma:function properties for G*}
Let $Q\in L^1_+(E)$ be a flow satisfying $Q^+=Q^-$ and $V_+\neq\emptyset$. If $m_Q<M_Q$ and $G_Q(\zeta_Q)=\infty$, then $G_Q^*\in C^2(m_Q,M_Q)$, $\lambda^* = \mathrm{d}G_Q^*/\mathrm{d}a$ is a strictly increasing function, and
\begin{equation}\label{limit for lambda*}
\lim_{a\to m_Q}\lambda^*(a)=-\infty,\qquad \lim_{a\to M_Q}\lambda^*(a)=\zeta_Q.
\end{equation}
where $\lambda^*(a)$ is any solution of the equation $F_Q(\lambda)=a$.
\end{lemma}

\begin{proof}
It follows from Lemma \ref{G_Q} that $F_Q\in C^1(-\infty,\zeta_Q)$ is strictly increasing. Hence $\lambda^* = \mathrm{d}G_Q^*/\mathrm{d}a$ is the inverse function of $F_Q$. Moreover, $\lambda^*\in C^1(m_Q,M_Q)$ is strictly increasing and \eqref{limit for lambda*} follows from \eqref{limit for F_Q}. By Lemma \ref{corollary:G*}, it is easy to see that $G_Q^*\in C^2(m_Q,M_Q)$ and
\begin{equation*}
\frac{dG_Q^*}{da}(a)=\lambda^*(a)+a\frac{d\lambda^*}{da}(a)-\frac{d\lambda^*}{da}(a)F_Q(\lambda^*(a))=\lambda^*(a).
\end{equation*}
This completes the proof of this lemma.
\end{proof}

Finally, we also need the following lemma to ensure $G_Q(\zeta_Q)=\infty$.

\begin{lemma}\label{G_Q hold}
Let $Q\in L^1_+(E)$ be a flow satisfying $Q^+=Q^-$ and $V_+\neq\emptyset$. If Condition \ref{condition:ccomp4} is satisfied, then we have $\zeta_Q=\inf_{x\in V_+}\zeta(x)$ and $G_Q(\zeta_Q)=\infty$.
\end{lemma}

\begin{proof}
Let $\psi$, $\zeta$, and $q_x$ be as in Condition \ref{condition:ccomp4}. It is easy to see that
\begin{equation*}
\zeta(x)=q_x\zeta,\qquad G_x(\lambda)=G\left(\frac{\lambda}{q_x}\right),
\end{equation*}
where $G(\lambda)=\log\psi(e^{\lambda\tau})$. By item (c) in Condition \ref{condition:ccomp4}, it is easy to see that there exists $y\in V_+$ such that $q_y = \inf_{x\in V_+}q_x>0$. By Lemma \ref{lemma:function properties}, it is clear that $G$ is increasing. Then for any $\lambda_0<\inf_{x\in V_+}\zeta(x)=q_{y}\zeta$, we have
\begin{equation*}
G_Q(\lambda_0)=\sum_{x\in V_+}Q_xG\left(\frac{\lambda_0}{q_x}\right)\le \sum_{x\in V_+}Q_xG\left(\frac{\lambda_0}{q_y}\right)<\infty.
\end{equation*}
By the definition of $\zeta_Q$, it is easy to see that $\zeta_Q\ge \lambda_0$. Since $\lambda_0<\inf_{x\in V_+}\zeta(x)$ is arbitrary, we have $\zeta_Q\ge \inf_{x\in V_+}\zeta(x)$. On the other hand, we have
\begin{equation*}
\varliminf_{\lambda\to\inf_{x\in V_+}\zeta(x)}G_Q(\lambda)\ge Q_y\varliminf_{\lambda\to q_y\zeta}G_y(\lambda)=Q_y\varliminf_{\lambda\to q_y\zeta}G\left(\frac{\lambda}{q_y}\right)=\infty.
\end{equation*}
This implies the proof of this lemma.
\end{proof}

We are now in a position to prove Proposition \ref{proposition:rate function for empirical flow}.

\begin{proof}[Proof of Proposition \ref{proposition:rate function for empirical flow}]
The proof of the marginal LDP for the empirical measure $\pi_t$ is similar to that given in \cite[Proposition 1.2]{mariani2016large}. Here we only focus on the marginal LDP for the empirical flow $Q_t$. Note that $(\mu,Q)\mapsto Q$ is a continuous map from $\Lambda$ to $L^1_+(E)$. By the contraction principle \cite[Theorem 4.2.1]{dembo1998large}, the law of $Q_t\in L^1_+(E)$ satisfies an LDP with good rate function
\begin{equation*}
\hat{I}_2(Q)=\inf\{I(\mu,Q):(\mu,Q)\in \mathcal{D}\}.
\end{equation*}

We next prove $\hat{I}_2=I_2$ under Condition \ref{condition:ccomp4}. By Remark \ref{rmk:omega dense}, we obtain
\begin{equation*}
\hat{I}_2(Q)=\inf\{I(\mu,Q):(\mu,Q)\in \mathcal{E}(Q)\},
\end{equation*}
where $\mathcal{E}(Q)=\{(\mu,Q)\in \mathcal{D}:\mu(x,\{\infty\})=0 \text{ for any }x\in V_+\}$. Recall the definition of $I$ in \eqref{def:I}. For any $Q\in L^1_+(E)$ satisfying $Q^+=Q^-$, we have
\begin{align*}
\hat{I}_2(Q)=&\;\sum_{x\in V}Q_x H\left(\tilde{Q}_{x,\cdot}\big|p_{x,\cdot}\right)+\inf_{(\mu,Q)\in \mathcal{E}(Q)}\sum_{x\in V}\left[Q_x H\left(\tilde{\mu}_x|\psi_{x}\right)+\zeta(x)\mu(x,\{\infty\})\right]\\
=&\;\sum_{x\in V}Q_x H\left(\tilde{Q}_{x,\cdot}\big|p_{x,\cdot}\right)+\inf_{(\mu,Q)\in \mathcal{E}(Q)}\left(\sum_{x\in V_+}Q_x H\left(\tilde{\mu}_x|\psi_{x}\right)+\sum_{x\notin V_+}\zeta(x)\mu(x,\{\infty\})\right).
\end{align*}
If $V_+=\emptyset$, then we have
\begin{equation*}
\begin{split}
\inf_{(\mu,Q)\in \mathcal{E}(Q)}\left(\sum_{x\in V_+}Q_x H\left(\tilde{\mu}_x|\psi_{x}\right)+\sum_{x\notin V_+}\zeta(x)\mu(x,\{\infty\})\right)&=\inf_{\sum_{x\in V}\mu(x,\{\infty\})=1}\zeta(x)\mu(x,\{\infty\})\\
&=\inf_{x\in V}\zeta(x),
\end{split}
\end{equation*}
which implies \eqref{I_2}. If $V_+\neq \emptyset$, then we have
\begin{equation*}
\mathcal{E}(Q)=\bigcup_{0< a\le 1}\mathcal{E}(Q,a),
\end{equation*}
where $\mathcal{E}(Q, a) = \{(\mu,Q)\in \mathcal{E}(Q) : \sum_{x\in V_+}\mu(x, (0,\infty)) = a\text{  and  } \sum_{x\notin V_+}\mu(x,\{\infty\})=1-a\}$. It then follows from Lemmas \ref{corollary:G*} and \ref{G_Q hold} that
\begin{align*}
&\;\inf_{(\mu,Q)\in \mathcal{E}(Q)}\left(\sum_{x\in V_+}Q_x H\left(\tilde{\mu}_x|\psi_{x}\right)+\sum_{x\notin V_+}\zeta(x)\mu(x,\{\infty\})\right)\\
=&\;\inf_{0< a\le 1}\inf_{(\mu,Q)\in \mathcal{E}(Q,a)}\left(\sum_{x\in V_+}Q_x H\left(\tilde{\mu}_x|\psi_{x}\right)+\sum_{x\notin V_+}\zeta(x)\mu(x,\{\infty\})\right)\\
=&\;\inf_{0< a\le 1}\left[\inf_{(\mu,Q)\in \mathcal{E}(Q,a)}\left(\sum_{x\in V_+}\int_{(0,\infty)}\frac{1}{s}\left(\log\frac{\mu(x,\mathrm{d}s)}{s Q_x\psi_x(\mathrm{d}s)}\right)\mu(x,\mathrm{d}s)\right)+(1-a)\inf_{x\notin V_+}\zeta(x)\right]\\
=&\;\inf_{0<a\le 1}\left[G_Q^*(a)+(1-a)\inf_{x\notin V_+}\zeta(x)\right].
\end{align*}
We next calculate the last term of the above equation in five different cases.

Case 1: $m_Q>1$. Since $G_Q^*(a)=\infty$ for any $0<a\le 1$, we have
\begin{equation*}
\inf_{0<a\le 1}\left[G_Q^*(a)+(1-a)\inf_{x\notin V_+}\zeta(x)\right]=\infty.
\end{equation*}
On the other hand, similarly to the proof of Case 3 in Lemma \ref{corollary:G*}, we have
\begin{equation*}
\sup_{\lambda<\inf_{x\in V}\zeta(x)}\left\{\lambda-\sum_{x\in V}Q_x\log\psi_x\left(e^{\lambda\tau}\right)\right\} = \lim_{\lambda\to-\infty}\left\{\lambda-\sum_{x\in V}Q_x\log\psi_x\left(e^{\lambda\tau}\right)\right\}=\infty.
\end{equation*}

Case 2: $m_Q\le 1<M_Q$ and $\lambda^*(1)\le \inf_{x\notin V_+}\zeta(x)$. Straightforward computations show that
\begin{equation*}
\inf_{0<a\le 1}\left[G_Q^*(a)+(1-a)\inf_{x\notin V_+}\zeta(x)\right]=G_Q^*(1).
\end{equation*}
On the other hand, it follows from Lemma \ref{lemma:function properties for G*} that
\begin{equation*}
\lambda^*(1)<\lim_{a\to M_Q}\lambda^*(a)=\inf_{x\in V_+}\zeta(x).
\end{equation*}
This implies that $\lambda^*(1)\le \inf_{x\in V}\zeta(x)$. Direct computations show that
\begin{equation*}
\sup_{\lambda<\inf_{x\in V}\zeta(x)}\left\{\lambda-\sum_{x\in V}Q_x\log\psi_x\left(e^{\lambda\tau}\right)\right\} = \lambda^*(1)-\sum_{x\in V}Q_x\log\psi_x\left(e^{\lambda^*(1)\tau}\right) = G_Q^*(1).
\end{equation*}

Case 3: $m_Q\le 1<M_Q$ and $\lambda^*(1)> \inf_{x\notin V_+}\zeta(x)$. It follows from Lemma \ref{lemma:function properties for G*} that there exists $\eta\in (m_Q,1)$ such that $\lambda^*(\eta)=\inf_{x\notin V_+}\zeta(x)$. Straightforward computations show that
\begin{equation}\label{temp}
\begin{split}
\inf_{0<a\le 1}\left[G_Q^*(a)+(1-a)\inf_{x\notin V_+}\zeta(x)\right]=&\;G_Q^*(\eta)+(1-\eta)\inf_{x\notin V_+}\zeta(x)\\
=&\;\eta\lambda^*(\eta)-\sum_{x\in V}Q_x\log\psi_x\left(e^{\lambda^*(\eta)\tau}\right)+(1-\eta)\lambda^*(\eta)\\
=&\;\lambda^*(\eta)-\sum_{x\in V}Q_x\log\psi_x\left(e^{\lambda^*(\eta)\tau}\right)\\
=&\;\sup_{\lambda<\inf_{x\in V}\zeta(x)}\left\{\lambda-\sum_{x\in V}Q_x\log\psi_x\left(e^{\lambda\tau}\right)\right\}.
\end{split}
\end{equation}
Note that $\inf_{x\notin V_+}\zeta(x)<\lambda^*(1)\le \lim_{a\to M_Q}\lambda^*(a)=\inf_{x\in V_+}\zeta(x)$. Then the last equability in \eqref{temp} follows from the fact that $\inf_{x\in V}\zeta(x)=\inf_{x\notin V_+}\zeta(x)=\lambda^*(\eta)$.

Case 4: $M_Q\le 1$ and $\inf_{x\notin V_+}\zeta(x)=\infty$. Since $M_Q<\infty$, it is easy to see that $\inf_{x\in V_+}\zeta(x)=\infty$. Then we have $\inf_{x\in V}\zeta(x)=\infty$ and
\begin{equation*}
\inf_{0<a\le 1}\left[G_Q^*(a)+(1-a)\inf_{x\notin V_+}\zeta(x)\right] = G^*(1) = \sup_{\lambda<\inf_{x\in V}\zeta(x)}\left\{\lambda-\sum_{x\in V}Q_x\log\psi_x\left(e^{\lambda\tau}\right)\right\}.
\end{equation*}

Case 5: $M_Q\le 1$ and $\inf_{x\notin V_+}\zeta(x)<\infty$. Note that $G^*_Q(a)=\infty$ for any $a> M_Q$. Then we have
\begin{align*}
\inf_{0<a\le 1}\left[G_Q^*(a)+(1-a)\inf_{x\notin V_+}\zeta(x)\right] = \inf_{0<a\le M_Q}\left[G_Q^*(a)+(1-a)\inf_{x\notin V_+}\zeta(x)\right].
\end{align*}
By Lemma \ref{lemma:function properties for G*}, we have $\lambda^*(M_Q)=\infty$. Since $\inf_{x\in V_+}\zeta(x)=\infty$, there exists $\eta\in (m_Q,M_Q)$ such that $\lambda^*(\eta)=\inf_{x\notin V_+}\zeta(x)$. Similarly to the proof in Case 3, we have
\begin{align*}
\inf_{0<a\le M_Q}\left[G_Q^*(a)+(1-a)\inf_{x\notin V_+}\zeta(x)\right]=&\;G_Q^*(\eta)+(1-\eta)\inf_{x\notin V_+}\zeta(x)\\
=&\;\sup_{\lambda<\inf_{x\in V}\zeta(x)}\left\{\lambda-\sum_{x\in V}Q_x\log\psi_x\left(e^{\lambda\tau}\right)\right\}.
\end{align*}
This completes the proof of this proposition.
\end{proof}

\section*{Acknowledgements}
We are grateful to Dr. L.Youming  for stimulating discussions. C.\ J.\ acknowledges support from National Natural Science Foundation of China with grant No. U1930402 and grant No. 12131005.  D.-Q.\ J.\ acknowledges support from National Natural Science Foundation of China with grant No. 11871079 and grant No. 12090015.

\begin{appendices}

\section{Another compactness condition}\label{appendix: A}
In \cite{bertini2015flows}, the authors imposed an alternative compactness condition, which is different from Condition \ref{condition:ccomp2} proposed in the present paper, and proved the joint LDP for the empirical measure and empirical flow of continuous-time Markov chains when $L^1_+(E)$ is endowed with the strong topology. In fact, we can also prove the joint LDP for semi-Markov processes under a similar compactness condition. To see this, we introduce some notation. For any set $\hat E \subset E$, we define an $\hat{E}$-dependent function $R: V\to\mathbb R$ by
\begin{equation}\label{defR}
R(y)=  \sum _{z: (y,z)\in \hat E} p_{yz},\quad y\in V.
\end{equation}

\begin{condition}\label{condition:strong topology}
Suppose there exists a set $\hat E \subset E$ satisfying the following three requirements:
\begin{itemize}
\item[(a)] For any $y\in V$, there exists $z\in V$ such that $(y,z) \in \hat E$;
\item[(b)] The function $R : V \to (0, \infty)$ defined in \eqref{defR} vanishes at infinity;
\item[(c)] For any $x\in V$, there exist constants $a_0, \lambda >0$ such that for any $a <a_0$
one can find a set $W=W(x,a)\subset E$ satisfying the following properties:
\begin{itemize}
\item[(i)] The complement $E \setminus W$ is finite;			
\item[(ii)] If $(y,z)\in W$, then $ R(y)<a$;
\item[(iii)] For each path exiting from $x$, the number of edges in $\hat{E}\cap W$ is at least $\lambda$-times the total number of edges in $W$. In other words, for any path $x_0,x_1, \dots, x_k$  with $x_0=x$ and $(x_i,x_{i+1} ) \in E$, we have
\begin{equation}\label{broccolo}
	\sharp\Big \{ i: (x_{i-1} , x_{i})  \in \hat E \cap W,1\le i\le k  \Big\}\, \geq \, \lambda\, \sharp \Big\{ i: (x_{i-1} , x_{i})  \in W,1\le i\le k\Big\} \,.
\end{equation}	
Obviously, we have $\lambda<1$.		
\end{itemize}
\end{itemize}
\end{condition}

Note that both Conditions \ref{condition:ccomp2} and \ref{condition:strong topology} only depend on the embedded chain $X$ of the semi-Markov process. The following proposition shows that Condition \ref{condition:ccomp2} is weaker than Condition \ref{condition:strong topology}. Due to this reason, we impose the former rather than the latter in the main text.

\begin{proposition}\label{pro:condition}
Suppose that Assumption \ref{ass:irreducibility} and Condition \ref{condition:strong topology} are satisfied. Then Condition \ref{condition:ccomp2} holds.
\end{proposition}

\begin{proof}
Let the constants $a_0,\lambda>0$ be as in item (c) of Condition \ref{condition:strong topology}. For any $x_0\in V$ and $m\in \mathbb{N}$ satisfying $1/m<a_0$, we will next construct the sequence of functions $u_n$, $n\ge m$ such that Condition \ref{condition:ccomp2} holds. We first construct auxiliary functions $F_n:E\to (0,\infty)$ by induction. Let $W_n\subset E$ be a sequence of sets defined as
\begin{equation*}
W_m=W\left(x_0,\frac{1}{m}\right),\qquad W_{n}=W\left(x_0,\frac{1}{n}\right)\cap \left(W_{n-1}\cup \hat{E}\right),\quad n\ge m+1.
\end{equation*}
It is easy to see that $W_n$ also satisfies item (c) in Condition \ref{condition:strong topology} for each $n\geq m$. Set
\begin{equation*}\label{Fm}
F_m(y,z)= \left\{\begin{aligned}
&m^{1-\lambda}, && \text{if } (y,z) \in \hat E\cap W_m\,,\\
&m^{-\lambda}, && \text{if } (y,z) \in \hat E^c\cap W_m\,,\\
&1, && \text{if } (y,z) \in E \setminus W_m\,,\end{aligned}\right.	
\end{equation*}
and for $n\ge m+1$, set
\begin{equation*}\label{Fn}
F_{n}(y,z)= \left\{\begin{aligned}
&n^{1-\lambda}, && \text{if } (y,z) \in \hat E\cap W_{n}\,,\\
&n^{-\lambda}, && \text{if } (y,z) \in \hat E^c\cap W_{n}\,,\\
&F_{n-1}(y,z), && \text{if } (y,z) \in E \setminus W_{n}\,.\end{aligned}\right.
\end{equation*} 	
Obviously, $F_n(x,y)\le n^{1-\lambda}$ for any $(x,y)\in E$. For any $x\in V$, let $G_x$ be the collection of all paths in $(V,E)$ with initial state $x_0$ and terminal state $x$, i.e.
\begin{equation*}
G_x=\left\{(x_0,x_1,\cdots,x_k):k\in\mathbb{N},x_k=x,x_i\in V,\text{ and }(x_i,x_{i+1})\in E\right\}.
\end{equation*}
It then follows from Assumption \ref{ass:irreducibility} that $(V,E)$ is connected. This implies that $G_x\neq \emptyset$ for any $x\in V$. For each $n\geq m$, let $u_n:V\to(0,\infty)$ be a function defined by
\begin{equation*}
u_n(x)=\inf_{\mathbf{x}\in G_x}\prod_{i=1}^{k}F_n(x_{i-1},x_i),\quad  x\in V.
\end{equation*}
Now we verify Condition \ref{condition:ccomp2} for the sequence of functions $u_n$.

(a) Obviously, if $(y,z)\in E$, then we have $u_n(z)\le u_n(y)F_n(y,z)\le u_n(y) n^{1-\lambda}$. For any $x\in V$ and $n\ge m$, we obtain
\begin{equation*}
Pu_n(x)=\sum_{y:(x,y)\in E}p_{xy}u_n(y)\le \sum_{y:(x,y)\in E} p_{xy} u_n(x) n^{1-\lambda} = u_n(x) n^{1-\lambda} < \infty.
\end{equation*}

(b) Recall the definition of auxiliary functions $F_n$ given above. For any path $\mathbf{x}=(x_0,x_1,\cdots,x_k)\in G_x$, it follows from item (c)-(iii) in Condition \ref{condition:strong topology} that
\begin{equation}\label{prod F}
\begin{split}
\prod_{i=1}^{k}F_m(x_{i-1},x_i)&=m^{[(1-\lambda)\#\{i:(x_{i-1},x_{i})\in \hat{E}\cap W_m,1\le i\le k\}-\lambda\#\{i:(x_{i-1},x_{i})\in \hat{E}^c\cap W_m,1\le i\le k\}]}\\
&=m^{[\#\{i:(x_{i-1},x_{i})\in \hat{E}\cap W_m,1\le i\le k\}-\lambda\#\{i:(x_{i-1},x_{i})\in  W_m,1\le i\le k\}
	]}\ge 1.
\end{split}
\end{equation}
Similarly, for $n\ge m+1$, we have
\begin{equation}\label{prod F2}
\begin{split}
&\prod_{i=1}^{k}\frac{F_{n}(x_{i-1},x_i)}{F_{n-1}(x_{i-1},x_i)}=\prod_{i:(x_{i-1},x_{i})\in \hat{E}\cap W_{n}}\frac{n^{1-\lambda}}{F_{n-1}(x_{i-1},x_{i})}\prod_{i:(x_{i-1},x_{i})\in \hat{E}^c\cap W_{n}}\frac{n^{-\lambda}}{F_{n-1}(x_{i-1},x_{i})}\\
&\qquad\qquad\ge (\frac{n}{n-1})^{[(1-\lambda)\#\{i:(x_{i-1},x_{i})\in \hat{E}\cap W_{n},1\le i\le k\}-\lambda\#\{i:(x_{i-1},x_{i})\in \hat{E}^c\cap W_{n},1\le i\le k\}]} \ge 1.
\end{split}
\end{equation}
Indeed, $\hat{E}^c\cap W_{n}=\hat{E}^c\cap W(x_0,1/n)\cap W_{n-1}$, which implies the first inequality in \eqref{prod F2}. For any $x\in V$, optimizing over $G_x$ on both sides of \eqref{prod F}, we have $u_m(x)\ge 1$. Similarly, it follows from \eqref{prod F2} that $u_{n}(x)\ge u_{n-1}(x)$ for $n\ge m+1$. This implies that $u_n$ is an increasing sequence of functions and $u_n(x)\ge 1$ for any $x\in V$ and $n\ge m$.

(c) For any $y\in V$, it follows from item (a) in Condition \ref{condition:strong topology} that $R(y)>0$. For any $x\in V$, since $G_x\neq\emptyset$, there exists $\mathbf{x}=(x_0,x_1\cdots,x_k)\in G_x$. Hence, we can let $N= m\vee (1+\max_{0\le i\le k}\lfloor1/R(x_i)\rfloor)$. By item (c)-(ii) in Condition \ref{condition:strong topology}, it is easy to see that $(x_{i-1},x_{i})\notin W_N$ for any $0\le i\le k$. Then for any $n\ge N$, we have
\begin{equation*}
u_n(x)\le \prod_{i=1}^{k}F_n(x_{i-1},x_i)= \prod_{i=1}^{k}F_N(x_{i-1},x_i)\le N^{(1-\lambda)k}.
\end{equation*}
Obviously, $N$ and $k$ only depend on $x$.

(d) For any $x\in V$, let $\mathbf{x}$ and $N$ be defined as in (c). Since $R(x)>1/N$, it is easy to see that $(x,y)\notin W_N$ for any $(x,y)\in E$. Then for any $n\ge N$, we have
\begin{equation}\label{Pun}
Pu_n(x)\le \sum_{y:(x,y)\in E}p_{xy}u_n(x)F_n(x,y)= \sum_{y:(x,y)\in E}p_{xy}u_n(x)F_N(x,y)\le N^{1-\lambda}u_n(x).
\end{equation}
Note that $u_n$ is an increasing sequence of functions with an upper bound. Then we define $u^*(x)=\lim_{n\to \infty}u_n(x)$. It follows from \eqref{Pun} and the monotone convergence theorem that
\begin{equation*}
Pu^*(x)=\lim_{n\to \infty}Pu_n(x)\le N^{1-\lambda}u^*(x).
\end{equation*}
Thus $u_n/Pu_n$ converges pointwise to the function $\hat{u}=u^*/Pu^*:V\to (0,\infty)$.

(e) Let $V_n$ and $V'_n$ be two sequences of subsets of $V$ defined by
\begin{equation*}
V_n=\left\{x\in V: (x,y)\in E \text{ implies } (x,y)\in W_n \text{ for any }y\in V\right\},\quad V'_n=\left\{x\in V:R(x)< \frac{1}{n}\right\},\ n\ge m.
\end{equation*}
Obviously, we have $V_n\subset V'_n$. For any $ x\in V_{j}\subset V'_{j}=\cup_{i=j}^{\infty}(V'_{i}\setminus V'_{i+1})$, there exists $ i\ge j$ such that $x\in V'_{i}\setminus V'_{i+1}$. If $(x,y)\in E$, it is easy to see that $(x,y)\notin W_{n}$ for $n\ge i+1$. This implies that $F_{n}(x,y)=F_{i}(x,y)$ for $n\ge i$. Then we have
\begin{equation*}
u^*(y)=\lim_{n\to \infty} u_n(y) \le \lim_{n\to\infty} u_n(x) F_n(x,y) = u^*(x) F_{i}(x,y).
\end{equation*}
Note that for any $(x,y)\in E$, we have $(x,y)\in W_j$. Then we have
\begin{align*}
\frac{Pu^*(x)}{u^*(x)}\le &\sum_{y:(x,y)\in E}p_{xy}F_{i}(x,y)= \sum_{y:(x,y)\in \hat{E}}p_{xy}F_{i}(x,y)+\sum_{y:(x,y)\in \hat{E}^c}p_{xy}F_{i}(x,y)\\
&\qquad \le R(x)i^{1-\lambda}+j^{-\lambda}\le i^{-\lambda}+j^{-\lambda}\le 2j^{-\lambda}.
\end{align*}
This means that $\log\hat{u}(x)\ge \log(j^{\lambda}/2)$ for any $x\in V_j$. By item (c)-(i) in Condition \ref{condition:strong topology}, it is easy to see that
\begin{equation*}
V\setminus V_j= \{x\in V:\exists\ y \text{ such that }(x,y)\in E\setminus W_j\}
\end{equation*}
is a finite set. For any $\ell>0$, select $j$ such that $\log(j^{\gamma}/2)>\ell $. Then $\{x\in V:\log\hat{u}(x)\le \ell\}\subset V\setminus V_j$ is a finite set.
\end{proof}

\begin{corollary}\label{LDP:strong topology 2}
Suppose that Assumptions \ref{ass:irreducibility}-\ref{ass:locally finite} and Conditions \ref{condition:ccomp} and \ref{condition:strong topology} are satisfied. Then the results of Theorem \ref{LDP:strong topology} remain valid.
\end{corollary}
\begin{proof}
The results follow directly from Theorem~\ref{LDP:strong topology} and Proposition \ref{pro:condition}.
\end{proof}

\end{appendices}

\setlength{\bibsep}{5pt}
\footnotesize\bibliographystyle{nature}

\end{document}